\documentclass[10pt]{article}

\usepackage{amssymb, amsthm, mathtools}
\usepackage[hidelinks, citecolor=blue]{hyperref}
\hypersetup{colorlinks=true, urlcolor = blue, linkcolor=blue}
\usepackage{stmaryrd}
\usepackage{verbatim}
\usepackage{bm}
\usepackage{times}
\usepackage[integrals]{wasysym}
\usepackage{stmaryrd}
\usepackage{tikz}
\usepackage{pgfplots}
\pgfplotsset{compat=1.18}
\usetikzlibrary{arrows.meta, positioning, decorations.pathreplacing}
\tikzset{snake it/.style={decorate, decoration=snake}}
\usepackage{setspace}

\usepackage{color}

\newcommand{\nc}{\normalcolor}

\numberwithin{equation}{section}

\newtheorem{theorem}{Theorem}
\newtheorem{corollary}[theorem]{Corollary}
\newtheorem{lemma}[theorem]{Lemma}
\newtheorem{prop}[theorem]{Proposition}
\newtheorem{defn}[theorem]{Definition}
\newtheorem{rem}[theorem]{Remark}

\newcommand{\cc}{{\mathbb{C}}}
\newcommand{\zz}{{\mathbb{Z}}}

\newcommand{\pp}{{\mathbb{P}}}
\newcommand{\oh}{\mathcal{O}}
\newcommand{\im}{{\operatorname{Im}\,}}

\newcommand{\beq}[1]{\begin{equation} \label{#1}}
\newcommand{\eeq}{\end{equation}}

\usepackage[top=1.25in, bottom=1.25in, left=1.25in, right=1.25in]{geometry}
\newcommand*\samethanks[1][\value{footnote}]{\footnotemark[#1]}
\renewcommand{\nc}{\newcommand}
\nc{\rnc}{\renewcommand}
\nc{\R}{\mathbb{R}}
\nc{\C}{\mathbb{C}}
\rnc{\d}{\mathrm{d}}
\nc{\E}{\mathbb{E}}
\nc{\Z}{\mathbb{Z}}

\begin{document}
\title{Quantum diffusion and delocalization in one-dimensional band matrices via the flow method}

\author{Sofiia Dubova\thanks{Department of Mathematics, Harvard Univeristy} \and Kevin Yang\samethanks}
\maketitle

\begin{abstract}
We study a class of Gaussian random band matrices of dimension $N\times N$ and band-width $W$. We show that delocalization holds for bulk eigenvectors and that quantum diffusion holds for the resolvent, all under the assumption that $W\gg N^{8/11}$. This improves the best-known result of \cite{BFYY17,BYY20,YY21}. Our analysis is based on a flow method, and a refinement of it may lead to an improvement on the condition $W\gg N^{8/11}$.
\end{abstract}


\section{Introduction}
In \cite{W57}, Wigner introduced the random matrix universality class to be a model for highly correlated quantum systems. Although the Wigner ensemble is a mean-field model, it is conjectured to also describe the behaviors (e.g. spectral data) of many other models. A prototypical class of Wigner matrices is the \emph{Gaussian unitary ensemble} (GUE), whose entries are standard complex Gaussians independent up to the Hermitian constraint.

A historically important non-mean-field model that is conjecturally related to the random matrix universality class is the Anderson model on lattice $\Z^{d}$. This model is the operator $\Delta+\lambda V$, where $\Delta$ is the discrete Laplacian, where $\lambda>0$, and where $V$ is a random diagonal matrix with i.i.d. entries. It is believed \cite{A58} that for large $\lambda$, eigenvalues of $\Delta+\lambda V$ form a Poisson point process, and eigenvectors have a finite localization (i.e. support) length. For small $\lambda$, eigenvalue and eigenvector statistics of $\Delta+\lambda V$ are believed to agree with random matrix statistics, e.g. delocalization of eigenvectors. There has been a lot of work showing localization for large $\lambda$ \cite{AM93,BK05,CKM87,DRS02,DS20,FMSS85,FS88,GK13,LZ22}. However, delocalization has remained an outstanding problem with some results on the Bethe lattice \cite{AW11,AW13}.

Another important non-mean-field example, which is conjectured to exhibit a similar transition between random matrix and Poisson statistics, is the \emph{random band matrix} \cite{B18}. This is the model of interest in this paper. It is a Hermitian matrix $H$ of dimension $N\times N$. Its entries $H_{xy}$ are i.i.d. complex random variables up to the Hermitian constraint, and they vanish when the distance (with respect to periodic boundary conditions on $\{1,\ldots,N\}$) between the indices $x,y$ is larger than a fixed band width parameter $W$. (The matrix $H$ is also normalized so that the matrix $S$ of variances $S_{xy}=\E|H_{xy}|^{2}$ is a doubly stochastic matrix.) 

When $W\gg1$ in the large-$N$ limit, the global eigenvalue density of $H$ converges to the Wigner semicircle law \cite{BMP91}. However, based on Thouless’ conductance fluctuation theory and scaling arguments \cite{T77}, numerical simulations \cite{CMI90}, and non-rigorous supersymmetric heuristics \cite{E97}, \emph{local} eigenvalue statistics and eigenvector statistics are conjectured to exhibit the following transition \cite{B18} in the bulk (i.e. for eigenvalues with real part $|E|<2$ independent of $N$):
\begin{itemize}
\item If $W\gg N^{1/2}$, then bulk eigenvalues have GUE statistics, and eigenvectors are delocalized, i.e. there is no length-scale $\ell\ll N$ to which eigenvectors are localized.
\item If $W\ll N^{1/2}$, then bulk eigenvalues form a Poisson process, and eigenvectors are localized.
\end{itemize}

Thus, random band matrices and the Anderson model are believed to be similar when $\lambda\sim W^{-1}$. We note that a similar transition exists at the edge, i.e. for eigenvalues of order $N^{-2/3}$ from the spectral edge $2$ of the semicircle. In this case, GUE statistics occur for $W\gg N^{5/6}$ and Poisson statistics occur for $W\ll N^{5/6}$; this was shown by Sodin \cite{S10}, as was a detailed description of spectral statistics at the transition point $W\sim N^{5/6}$.

A long series of works established GUE statistics under improving assumptions on the band width. The works \cite{EK11,EKYY} showed that the percentage of localized eigenvectors in the bulk is $\mathrm{o}(1)$, first for $W\gg N^{6/7}$, then for $W\gg N^{4/5}$. This was improved to a stronger form of delocalization and GUE eigenvalue statistics in \cite{BFYY17,BYY20,YY21} for $W\gg N^{3/4}$. In addition, there has been work on bounds on the localization length \cite{BFYY17,BYY20,PSSS17,YY21}. We also mention \cite{S14}, which shows delocalization for Gaussian band matrices with a special variance profile $S_{xy}$. These use the supersymmetric method, and they work up to $W\gg N^{6/7}$. (Extensions to more general matrices without the Gaussian assumption were given in \cite{BE17}, which require $W\geq cN$.) The supersymmetric method has also been able to show a transition in the two-point correlations of bulk eigenvalues at $W\sim N^{1/2}$ \cite{SS19}, though it is unclear if these methods extend to delocalization versus localization of eigenvectors. On the other side of the transition, localization for all the eigenvectors was shown in \cite{S09,PSSS17,CPSS24}, first for $W\ll N^{1/8}$, then for $W\ll N^{1/7}$ in specific Gaussian models, and most recently for $W\ll N^{1/4}$. 

The goal of this paper is to show delocalization of bulk eigenvectors for a large class of Gaussian random band matrices with very general variance profiles assuming $W\gg N^{8/11}$. This improves on the best-known result of $W\gg N^{3/4}$ \cite{BFYY17,BYY20,YY21}. We also prove a phenomenon known as \emph{quantum diffusion}; this derives the resolvent of a diffusion operator from the resolvent of $H$. Showing delocalization through quantum diffusion, which has the benefit of being ``natural" in the sense explained after Theorem \ref{theorem:qd}, goes back at least to \cite{EK11b}.

Our analysis is built on the flow method. In this method, the spectral parameter $z$ for the resolvent of $H$ follows a constant-speed characteristic in the upper-half plane. Its imaginary part starts at an order $1$ value, and it ends at a small-scale $\eta$ of interest. Simultaneously, the entries of $H$ are realized as Brownian motions. The upshot to this method is the introduction of dynamical ideas. It was used to prove a local semicircle law for a class of Wigner matrices \cite{vSS18,vSS19}, but the application to band matrices seems to be new to this work.

Finally, we mention progress on random band matrices in higher dimension $d$ (in which the matrix size is $N=L^{d}$, and the predicted thresholds are in terms of $W$ and $L$). For $d\geq2$, some lower bounds on the localization length were obtained \cite{EK11,EK11b}. In \cite{XYYY24,YYY21,YYY22}, the authors showed a weak form of delocalization (see Corollary \ref{corollary:deloc}) for bulk eigenvectors in dimension $d\geq8$, as well as GUE eigenvalue statistics in dimension $d\geq7$. The former work assumes only (essentially) the optimal constraint that $W\gg L^{\varepsilon}$; the latter assumes instead that $W\gg L^{95/(d+95)}$ (which, for large dimension $d$, looks like $W\gg L^{\varepsilon}$). 

The next section states precisely our model and main results as well as an outline for the rest of the paper.
\subsection{Acknowledgements}
We thank Horng-Tzer Yau for helpful discussions. K.Y. was supported in part by NSF DMS-2203075.
%
%
%
\section{Main results}
We now introduce the matrix model precisely. First, let $\mathbb{T}_{N}=\Z/N\Z$ be the discrete torus of length $N$. We identify it with $\{1,\ldots,N\}$ and we will give it the periodic distance $|x-y|_{N}:=\min([x-y],[y-x])$ for any $x,y\in\{1,\ldots,N\}$, where $[\cdot]$ means taking the mod-$N$ equivalence class in $\{0,\ldots,N-1\}$.

Next, let $f$ be a symmetric and compactly supported probability density on $\R$. We let $S=(S_{xy};x,y=1,\ldots,N)$ be a doubly stochastic matrix whose entries are given by
\begin{align*}
S_{xy}:=Z_{N,W}^{-1}f\left(\frac{|x-y|_{N}}{W}\right).
\end{align*}
Here, $Z_{N,W}$ is a normalizing constant satisfying $Z_{N,W}\asymp W$, where $\asymp$ means bounded above and below up to fixed, positive factors. We also assume that $S$ admits a matrix square root $S^{1/2}$ satisfying the same properties for a possibly different but still symmetric and compactly supported probability density $f$. (This is an entirely technical assumption that can probably be removed, though it is already quite general. Since our main attention is in the size of $W$ relative to $N$, we do not pursue this.)

The matrix ensemble of interest in this paper is denoted by $H=(H_{xy};x,y=1,\ldots,N)$. It is Hermitian, and for any $x\leq y$, $H_{xy}$ is a complex Gaussian random variable satisfying $\E H_{xy}=0$ and $\E|H_{xy}|^{2}=S_{xy}$.

Next, we introduce the Stieltjes transform of the Wigner semicircle law below:
\begin{align*}
m(z):=\int_{-2}^{2}\frac{1}{2\pi}\frac{\sqrt{4-x^{2}}}{x-z}dx=\frac{-z+\sqrt{z^{2}-4}}{2},
\end{align*}
where the second identity follows from the standard fact (see (3.13) in \cite{EY}) that $m(z)$ is the unique solution to the following self-consistent equation:
\begin{align*}
z=-m(z)^{-1}-m(z), \quad\mathrm{Im}(z),\mathrm{Im}m(z)>0.
\end{align*}
\emph{Throughout this paper, we follow standard notation and write $z=E+i\eta$} with $E\in\R$ and $\eta>0$.

To conclude this subsection, we now introduce frequently used notation for inequalities. We use standard big-O notation. We often write $a\lesssim b$ if $a=\mathrm{O}(b)$ (and similarly $a\gtrsim b$ if $b=\mathrm{O}(a)$). We also introduce the following standard notion of stochastic domination in random matrix theory.
\begin{defn}\label{defn:prec}
Consider two sequences of random variables, denoted by $X = \left\{X_N(s) : N\in \zz_+, s\in S_N\right\}$ and $Y = \left\{Y_N(s) : N\in \zz_+, s\in S_N\right\}$, that are parametrized by $s$ in some index set $S_{N}$. We say that $X$ is stochastically dominated by $Y$ uniformly in $s$ and write $X\prec Y$ or $X = \oh(Y)$ if for any $\varepsilon, D>0$ we have
\[
\sup_{s\in S_N} \pp\left(X_N(s) > N^{\varepsilon} Y_N(s)\right) < N^{-D}
\]
for large enough $N$.
\end{defn}
\subsection{Quantum diffusion}
Our first main theorem is an analysis of the following \emph{$T$-matrix} and its comparison to the \emph{diffusion profile} $\Theta$:
\begin{align*}
T(z)_{ab}&:=T_{ab}:=\sum_{x,y}S^{\frac12}_{ax}|G_{xy}|^{2}S^{\frac12}_{yb};\\
\Theta(z)&:=\Theta:=\frac{|m(z)|^{2}S}{1-|m(z)|^{2}S}.
\end{align*}
We clarify that in the formula for $T(z)_{ab}$, the value $S^{1/2}_{\alpha\beta}$ refers to the $(\alpha,\beta)$ entry of the matrix square root.
\begin{theorem}\label{theorem:qd}
First, assume that $|E|<2$ is fixed. Assume that there exists a fixed $\upsilon>0$ so that $\eta\asymp W^{2}N^{-2}$ and $W\geq N^{8/11+\upsilon}$. We have 
\begin{align}
\max_{x,y}|T_{xy}-\Theta_{xy}|\prec W^{-\frac74}\eta^{-\frac32}.
\end{align}
\end{theorem}
\begin{rem}\label{rem:eta}
The assumptions in Theorem \ref{theorem:qd} imply that $\eta\geq W^{-3/4+\upsilon}$ for some $\upsilon>0$.
\end{rem}
Roughly speaking, Theorem \ref{theorem:qd} establishes $T\approx\Theta$. (In particular, the error bound in Theorem \ref{theorem:qd} is much smaller than the maximal entry size of $\Theta$. Indeed, by classical resolvent bounds for diffusions, as in Proposition 2.8 in \cite{EKYY} or Lemma \ref{lemma:thetaestimates}, the maximal entry of $\Theta$ is order $W^{-1}\eta^{-1/2}$. Now, use that $\eta\gg W^{-3/4}$ to get $W^{-7/4}\eta^{-3/2}\ll W^{-1}\eta^{-1/2}$.) This is often called \emph{quantum diffusion} for the following reason. Classical asymptotics (see Lemma 3.5 in \cite{EKYY} and Lemma 6.2 in \cite{EY}) show that $1-|m(z)|^{2}\asymp\eta$ in the bulk ($|E|<2$). Thus, one can rewrite $\Theta$ as 
\begin{align*}
\Theta\sim\frac{S}{1-|m(z)|^{2}+|m(z)|^{2}(S-\mathrm{Id})}\sim\frac{S}{\alpha\eta+(S-\mathrm{Id})},
\end{align*}
where $\alpha\asymp1$. The matrix $S$ is the transition operator for a random walk on the torus $\mathbb{T}_{N}\simeq\{1,\ldots,N\}$. It has jump length $W$ and diffusivity proportional to $W^{2}$. Thus, $S-\mathrm{Id}$ is the generator for such a random walk, which provides $\Theta$ the interpretation as the resolvent of the generator for a random walk on $\mathbb{T}_{N}$. Moreover, the spectral gap for a random walk on $\mathbb{T}_{N}$ with steps of variance $W^{2}$ is of order $W^{2}N^{-2}$. The assumption $\eta\asymp W^{2}N^{-2}$ therefore means that Theorem \ref{theorem:qd} shows diffusion until the relaxation time $\eta^{-1}\asymp N^{2}W^{-2}$ of the random walk (which is often called the \emph{Thouless time} \cite{ET72,S15,T77}). 

As mentioned above, Theorem \ref{theorem:qd} is the main result that leads to proofs of the other theorems in this paper. A conceptual benefit to proving delocalization using Theorem \ref{theorem:qd} is that $\Theta$ naturally explains the conjectured threshold $W\gg N^{1/2}$. Indeed, $S-\mathrm{Id}$ has a spectral gap of $W^{2}N^{-2}$, so it is natural to take $\eta\asymp W^{2}N^{-2}$. On the other hand, we want $\eta\gg N^{-1}$ in the bulk, since $N^{-1}$ is the typical eigenvalue gap size, which leads to $W^{2}N^{-2}\gg N^{-1}$, i.e. $W\gg N^{1/2}$. The roadblock to proving quantum diffusion all the way to $W\gg N^{1/2}$ seems to be technical. We explain this (and a possible angle to overcome this difficulty) in Section \ref{section:conjecture}.
\subsection{Local law and eigenvector delocalization}
We now present the following local law, whose proof ultimately follows by Theorem \ref{theorem:qd} and other gymnastics that are relatively standard (as explained in Section \ref{section:flowbulk}).
\begin{theorem}\label{theorem:locallaw}
Retain the setting of Theorem \ref{theorem:qd}. We have
\begin{align}
\max_{x,y}|G_{xy}-m(z)\delta_{xy}|^{2}\prec W^{-1}\eta^{-\frac12}.\label{eq:locallawbulk}
\end{align}
\end{theorem}
As a consequence of Theorem \ref{theorem:locallaw}, we deduce what \cite{EKYY} calls the ``complete delocalization of (bulk) eigenvectors". To state it, we first introduce some notation. Fix any index $x$ and any integer $\ell\geq1$. We let $P_{x,\ell}$ be the projection onto the complement of the radius-$\ell$ ball around $x$, so that $P_{x,\ell}(y):=\mathbf{1}[|x-y|\geq\ell]$ for any indices $x,y$. Next, for any eigenvalue $\lambda_{\alpha}$ of $H$, we choose a unit-norm eigenvector $\mathbf{u}_{\alpha}$ with this eigenvalue. Given any $\kappa>0$ we let
\begin{align*}
\mathcal{A}_{\varepsilon,\ell,\kappa}:=\left\{\alpha:\lambda_{\alpha}\in[-E+\kappa,E-\kappa]: \sum_{x}|\mathbf{u}_{\alpha}(x)\|P_{x,\ell}\mathbf{u}_{\alpha}\|\leq\varepsilon\right\}
\end{align*}
be the labeling set for bulk eigenvectors that are localized to scale $\ell$ up to a small error $\varepsilon>0$. (As in Remark 7.2 of \cite{EKYY}, this includes the set of bulk eigenvectors that are exponentially localized to scale $\ell$.) The result below shows that this set has a small density.
\begin{corollary}\label{corollary:deloc}
Retain the setting of Theorem \ref{theorem:qd}. For any $\ell\ll N$ and any $\varepsilon,\kappa>0$ fixed, we have 
\begin{align*}
\frac{|\mathcal{A}_{\varepsilon,\ell,\kappa}|}{N}\lesssim\sqrt{\varepsilon}+\mathcal{O}(N^{-c}),
\end{align*}
where $c>0$ is a fixed constant.
\end{corollary}
\begin{proof}
The proof is exactly the content of Proposition 7.1 in \cite{EKYY}, which we now verify applies to the current situation. We first note that their condition $\eta\geq W^{-1+\gamma}$ for some $\gamma>0$ is satisfied by our assumptions on $\eta$ in Theorem \ref{theorem:locallaw}; see Remark \ref{rem:eta}. Next, we observe that by \eqref{eq:locallawbulk}, we have $|G_{xy}-m(z)\delta_{xy}|^{2}\prec W^{-1}\eta^{-1/2}\lesssim N^{-1}\eta^{-1}$; the last bound follows by $\eta\asymp W^{2}N^{-2}$. Finally, since we assume that $\eta\asymp W^{2}N^{-2}$ and $N\ll W^{11/8}$, we have $N^{-1}\eta^{-1}\asymp N W^{-2}\ll W^{-5/8}$. Thus, the assumptions in Proposition 7.1 of \cite{EKYY}, namely the lower bound on $\eta$, the a priori bound on $|G_{xy}|^{2}$, and the ``admissible estimate" on $|G_{xy}-m(z)\delta_{xy}|^{2}$ (as defined in Definition 3.7 of \cite{EKYY}), are all satisfied. But Proposition 7.1 in \cite{EKYY} implies the desired estimate directly, so the proof is complete.
\end{proof}
\subsection{Outline for the rest of the paper}
Section \ref{section:flowbulk} sets up the flow method and shows Theorems \ref{theorem:qd} and \ref{theorem:locallaw} modulo a key estimate (Proposition \ref{prop:dstop}). Section \ref{section:drift} shows this estimate via graphical expansions. Section \ref{section:conjecture} explains refinements of our analysis that may lead to an improvement on the assumption $W\gg N^{8/11}$. Section \ref{section:aux} gives auxiliary matrix estimates.
%
%
%
\section{The SDE flow}\label{section:flowbulk}
Let $B_{t,ij}$ be independent standard complex Brownian motions for all $i,j=1,\ldots,N$. Consider the process
\begin{align*}
dH_{t,ij}=\sqrt{S_{ij}}dB_{t,ij}, \quad H_{0}=0.
\end{align*}
Next, for any $z=E+i\eta$ with $E\in\R$ and $\eta>0$, define the path
\begin{align*}
t\mapsto w_{t}:=\frac{-1}{m(z)}-tm(z)=z+(1-t)m(z),
\end{align*}
where $m(z)$ is the Stieltjes transform of the semicircle distribution at $z$ (so that the second identity follows by the self-consistent equation for $m(z)$). Note that $w_{1}=z$ and $w_{0}=-m(z)^{-1}=z+m(z)$. Now, consider the resolvent of $H_{t}$ at the points $w_{t}$ and the associated $T$-matrix:
\begin{align*}
G_{t}(z)&:=(H_{t}-w_{t})^{-1}.\\
T_{t}(z)&:=S^{\frac12}F_{t}(z)S^{\frac12}, \quad \mathrm{where} \ F_{t}(z)_{ij}:=|G_{t}(z)_{ij}|^{2}.
\end{align*}
An elementary application of Ito's lemma gives the following SDEs. 
\begin{lemma}\label{lemma:gsdes}
We have $G_{0}=m(z)$, and we have the SDE
\begin{align}
dG_{t}(z)=-G_{t}(z)dH_{t}G_{t}(z)+G_{t}(z)\{\mathcal{S}[G_{t}(z)]-m(z)\}G_{t}(z)dt,
\end{align}
where $\mathcal{S}:M_{N}(\C)\to M_{N}(\C)$ is the linear operator defined by
\begin{align*}
\mathcal{S}[X]_{ij}:=\delta_{ij}\sum_{k=1}^{N}S_{ik}X_{kk}.
\end{align*}
We also have the SDE
\begin{align*}
dT_{t}(z)&=T_{t}(z)^{2}dt-S^{\frac12}dM_{t}(z)S^{\frac12}+S^{\frac12}\Omega_{t}(z)S^{\frac12}dt,\\
dM_{t}(z)&:=\overline{G}_{t}(z)\odot\{G_{t}(z)dH_{t}G_{t}(z)\}+\overline{G_{t}(z)dH_{t}G_{t}(z)}\odot G_{t}(z),\\
\Omega_{t}(z)&:=\overline{G}_{t}(z)\odot\{G_{t}(z)\{\mathcal{S}[G_{t}(z)]-m(z)\}G_{t}(z)\}+\overline{G_{t}(z)\{\mathcal{S}[G_{t}(z)]-m(z)\}G_{t}(z)}\odot G_{t}(z).
\end{align*}
where $\odot$ denotes entry-wise multiplication of matrices.
\end{lemma}
\begin{proof}
Fix any indices $\alpha,\beta,a,b$. By resolvent perturbation, we have 
\begin{align*}
\partial_{H_{t,\alpha\beta}}G_{t}(z)_{ab}&=-G_{t}(z)_{a\alpha}G_{t}(z)_{\beta b}\\
\partial_{H_{t,\beta\alpha}}\partial_{H_{t,\alpha\beta}}G_{t}(z)&=G_{t}(z)_{a\beta}G_{t}(z)_{\alpha\alpha}G_{t}(z)_{\beta b}+G_{t}(z)_{a\alpha}G_{t}(z)_{\beta\beta}G_{t}(z)_{\alpha b}\\
\partial_{t}G_{t}(z)_{ab}&=G_{t}(z)^{2}_{ab}\partial_{t}w_{t}=-m(z)G_{t}(z)^{2}_{ab},
\end{align*}
where in the last line, $G_{t}(z)^{2}_{ab}$ is the $(a,b)$ entry of the squared matrix $G_{t}(z)^{2}$. The first SDE now follows by the Ito formula. Next, we have 
\begin{align*}
dF_{t}(z)_{ij}&=d|G_{t}(z)_{ij}|^{2}=G_{t}(z)_{ij}d\overline{G}_{t}(z)_{ij}+\overline{G}_{t}(z)_{ij}dG_{t}(z)_{ij}+d[G_{t}(z)_{ij},\overline{G}_{t}(z)_{ij}]\\
&=-dM_{t}(z)+\Omega_{t}(z)dt+d[G_{t}(z)_{ij},\overline{G}_{t}(z)_{ij}],
\end{align*}
where the last line follows by plugging in the $G_{t}(z)$ SDE. On the other hand, we have
\begin{align*}
d[G_{t}(z)_{ij},\overline{G}_{t}(z)_{ij}]&=d\left[\sum_{\alpha,\beta}G_{t}(z)_{i\alpha}dH_{t,\alpha\beta}G_{t}(z)_{\beta j},\sum_{\alpha,\beta}\overline{G}_{t}(z)_{i\alpha}d\overline{H}_{t,\alpha\beta}\overline{G}_{t}(z)_{\beta j}\right]\\
&=\sum_{\alpha,\beta}S_{\alpha\beta}|G_{t}(z)_{i\alpha}|^{2}|G_{t}(z)_{\beta j}|^{2}dt=[F_{t}(z)SF_{t}(z)]_{ij}dt.
\end{align*}
If we now multiply the above by $S^{1/2}$ on the left and right, we ultimately deduce
\begin{align*}
dT_{t}(z)=S^{\frac12}dM_{t}(z)S^{\frac12}+S^{\frac12}\Omega_{t}(z)S^{\frac12}dt+S^{\frac12}F_{t}(z)SF_{t}(z)S^{\frac12}dt.
\end{align*}
The last term is just $T_{t}(z)^{2}dt$, so the proof is complete.
\end{proof}
Our goal is to analyze the $T_{t}(z)$ equation. Define the \emph{time-dependent diffusion profile} to be 
\begin{align*}
\Theta_{t}=\frac{|m(z)|^{2}S}{1-t|m(z)|^{2}S}.
\end{align*}
Straightforward differentiation shows that $d\Theta_{t}=\Theta_{t}^{2}dt$, and direct inspection shows $T_{0}(z)=\Theta_{0}=|m(z)|^{2}S$. Now, define $\mathcal{E}_{t}(z):=T_{t}(z)-\Theta_{t}$. Our goal is to show that $\mathcal{E}_{t}(z)$ is small. We start by computing its evolution equation.
\begin{lemma}\label{lemma:esdes}
Retain the notation of Lemma \ref{lemma:gsdes}. We have $\mathcal{E}_{0}(z)=0$ and
\begin{align}
d\mathcal{E}_{t}(z)&=(\Theta_{t}\mathcal{E}_{t}(z)+\mathcal{E}_{t}(z)\Theta_{t})dt+\mathcal{E}_{t}(z)^{2}dt-S^{\frac12}dM_{t}(z)S^{\frac12}+S^{\frac12}\Omega_{t}(z)S^{\frac12}dt.
\end{align}
We also have the integral equation $\mathcal{E}_{t}(z)=\mathcal{E}^{M}_{t}(z)+\mathcal{E}^{D}_{t}(z)+\mathcal{E}^{S}_{t}(z)$, where
\begin{align*}
\mathcal{E}^{M}_{t}(z)&:=-\int_{0}^{t}\{\mathrm{Id}+(t-s)\Theta_{t}\}S^{\frac12}dM_{s}(z)S^{\frac12}\{\mathrm{Id}+(t-s)\Theta_{t}\},\\
\mathcal{E}^{D}_{t}(z)&:=\int_{0}^{t}\{\mathrm{Id}+(t-s)\Theta_{t}\}S^{\frac12}\Omega_{s}(z)S^{\frac12}\{\mathrm{Id}+(t-s)\Theta_{t}\}ds,\\
\mathcal{E}^{S}_{t}(z)&:=\int_{0}^{t}\{\mathrm{Id}+(t-s)\Theta_{t}\}\mathcal{E}_{s}(z)^{2}\{\mathrm{Id}+(t-s)\Theta_{t}\}ds.
\end{align*}
\end{lemma}
\begin{proof}
Recall that $d\Theta_{t}=\Theta_{t}^{2}dt$. Combining this with Lemma \ref{lemma:gsdes} gives
\begin{align*}
d\mathcal{E}_{t}(z)=T_{t}(z)^{2}dt-\Theta_{t}^{2}dt-S^{\frac12}dM_{t}(z)S^{\frac12}+S^{\frac12}\Omega_{t}(z)S^{\frac12}dt.
\end{align*}
To get the first SDE, it suffices to note that $T_{t}(z)^{2}-\Theta_{t}^{2}=\Theta_{t}\mathcal{E}_{t}(z)+\mathcal{E}_{t}(z)\Theta_{t}+\mathcal{E}_{t}(z)^{2}$. To show that the integral equation holds, it is enough to use pathwise uniqueness of solutions to SDEs and to verify that $\mathcal{E}^{M}_{t}(z)+\mathcal{E}^{D}_{t}(z)+\mathcal{E}^{S}_{t}(z)$ satisfies the same SDE as $\mathcal{E}_{t}(z)$. For the latter, we note that 
\begin{align*}
\partial_{t}\{\mathrm{Id}+(t-s)\Theta_{t}\}&=\partial_{t}\frac{1-s|m(z)|^{2}S}{1-t|m(z)|^{2}S}=\Theta_{t}\frac{1-s|m(z)|^{2}S}{1-t|m(z)|^{2}S}=\Theta_{t}\{\mathrm{Id}+(t-s)\Theta_{t}\},
\end{align*}
at which point verifying the SDE for $\mathcal{E}^{M}_{t}(z)+\mathcal{E}^{D}_{t}(z)+\mathcal{E}^{S}_{t}(z)$ amounts to just calculus.
\end{proof}
\subsection{Stopping time construction}
The RHS of the $\mathcal{E}_{t}(z)$ equation in Lemma \ref{lemma:esdes} has nonlinear powers of $\mathcal{E}_{t}(z)$ itself, so we will need a stopping time argument to show small-ness of these nonlinear powers. Fix $\delta_{\mathrm{stop}}>0$ small but independent of $N$, and fix any large $D=\mathrm{O}(1)$. We define the stopping times
\begin{align}
\tau_{\mathrm{stop},1}&:=\inf\left\{s\geq0: \max_{a,b}|\mathcal{E}_{s}(z)_{ab}|\geq W^{\delta_{\mathrm{stop}}}W^{-\frac34}|\mathrm{Im}w_{s}|^{-1}\cdot W^{-1}|\mathrm{Im}w_{s}|^{-\frac12}\right\}\wedge1,\\
\tau_{\mathrm{stop},2}&:=\inf\left\{s\geq0: \max_{a,b}\frac{|G_{s}(z)_{ab}-\delta_{ab}m(z)|^{2}}{(S^{1/2}T_{s}(z)S^{1/2})_{ab}+S^{1/2}_{ab}+W^{-D}}\geq W^{\delta_{\mathrm{stop}}/10}\right\}\wedge1,\\
\tau_{\mathrm{stop}}&:=\tau_{\mathrm{stop},1}\wedge\tau_{\mathrm{stop},2}.\label{eq:taubulk}
\end{align}
Since the typical size of $\max_{ab}(\Theta_{t})_{ab}$ is $W^{-1}|\mathrm{Im}w_{t}|^{-1/2}$, the first stopping time $\tau_{\mathrm{stop},1}$ is giving us the small-ness factor $W^{-3/4}|\mathrm{Im}w_{t}|^{-1}$, which indicates the importance of the assumption $\eta \gg W^{-3/4}$ in this paper (since $\mathrm{Im}w_{t}$ is decreasing in $t$ and $\mathrm{Im}w_{1}=\mathrm{Im}z=\eta$). We also note that the exponent $\delta_{\mathrm{stop}}/10$ is not itself particularly important, as any small multiple of $\delta_{\mathrm{stop}}$ would be sufficient.

The key technical result leading the proof of Theorems \ref{theorem:qd} and \ref{theorem:locallaw} is the following.
\begin{theorem}\label{theorem:stop}
We have $\mathbb{P}[\tau_{\mathrm{stop},i}\neq1]\lesssim_{D} N^{-D}$ for any $i=1,2$ and $D>0$.
\end{theorem}
Our strategy for proving Theorem \ref{theorem:stop} is to construct a ``stopped" version of the $\mathcal{E}_{t}(z)$ dynamics. To be more precise, let $\mathcal{E}^{\mathrm{stop}}$ denote the solution to
\begin{align*}
d\mathcal{E}^{\mathrm{stop}}_{t}(z)&=(\Theta_{t}\mathcal{E}^{\mathrm{stop}}_{t}(z)+\mathcal{E}^{\mathrm{stop}}_{t}(z)\Theta_{t})dt+\mathbf{1}_{t\leq\tau_{\mathrm{stop}}}\mathcal{E}^{\mathrm{stop}}_{t}(z)^{2}dt\\
&-\mathbf{1}_{t\leq\tau_{\mathrm{stop}}}S^{\frac12}dM_{t}(z)S^{\frac12}+\mathbf{1}_{t\leq\tau_{\mathrm{stop}}}S^{\frac12}\Omega_{t}(z)S^{\frac12}dt
\end{align*}
Now, by the Duhamel formula as in the proof of Lemma \ref{lemma:esdes}, we have 
\begin{align}
\mathcal{E}^{\mathrm{stop}}_{t}(z)=\mathcal{E}^{M,\mathrm{stop}}_{t}(z)+\mathcal{E}^{D,\mathrm{stop}}_{t}(z)+\mathcal{E}^{\mathrm{S},\mathrm{stop}}_{t}(z),\label{eq:estopformula}
\end{align}
where
\begin{align*}
\mathcal{E}^{M,\mathrm{stop}}_{t}(z)&:=-\int_{0}^{\tau_{\mathrm{stop}}\wedge t}\{\mathrm{Id}+(t-s)\Theta_{t}\}S^{\frac12}dM_{s}(z)S^{\frac12}\{\mathrm{Id}+(t-s)\Theta_{t}\},\\
\mathcal{E}^{D,\mathrm{stop}}_{t}(z)&:=\int_{0}^{\tau_{\mathrm{stop}}\wedge t}\{\mathrm{Id}+(t-s)\Theta_{t}\}S^{\frac12}\Omega_{s}(z)S^{\frac12}\{\mathrm{Id}+(t-s)\Theta_{t}\}ds,\\
\mathcal{E}^{S,\mathrm{stop}}_{t}(z)&:=\int_{0}^{\tau_{\mathrm{stop}}\wedge t}\{\mathrm{Id}+(t-s)\Theta_{t}\}\mathcal{E}_{s}(z)^{2}\{\mathrm{Id}+(t-s)\Theta_{t}\}ds.
\end{align*}
Since $\tau_{\mathrm{stop}}$ is a stopping time with respect to the Brownian filtration, by uniqueness of strong solutions to finite-dimensional Ito SDEs, we know that $\mathcal{E}^{\mathrm{stop}}_{t}(z)=\mathcal{E}_{t}(z)$ for all $t\leq\tau_{\mathrm{stop}}$. With this in mind, our goal is to now control each term in the previous display uniformly over all $t\in[0,1]$ with high probability, and the bounds we obtain will be better than the ones defining $\tau_{\mathrm{stop}}$. In particular, this will show that the stopping time $\tau_{\mathrm{stop}}$ is self-propagating, in which case Theorem \ref{theorem:stop} will follow since $\tau_{\mathrm{stop}}>0$ with probability $1$.

We start with the term $\mathcal{E}^{S,\mathrm{stop}}_{t}(z)$; we give a deterministic bound for this.
\begin{lemma}\label{lemma:sstop}
There exists $\delta>0$ so that $|\mathcal{E}^{S,\mathrm{stop}}_{t}(z)_{ab}|\lesssim W^{-\delta}W^{-\frac34}|\mathrm{Im}w_{t}|^{-1} W^{-1}|\mathrm{Im}w_{t}|^{-\frac12}$ for all $t\in[0,1]$.
\end{lemma}
\begin{proof}
We first claim the following estimates:
\begin{align}
\sup_{x}\sum_{y}\{\mathrm{Id}+(t-s)\Theta_{t}\}_{xy}+\sup_{y}\sum_{x}\{\mathrm{Id}+(t-s)\Theta_{t}\}_{xy} = 1+\mathrm{O}(|\mathrm{Im}w_{t}|^{-1}|\mathrm{Im}w_{s}|). \label{eq:sstop1}
\end{align}
Since the matrix $\mathrm{Id}+(t-s)\Theta_{t}$ is symmetric, it suffices to control only one of the two terms on the LHS; we choose the first. It suffices to show that $\sum_{y}(t-s)(\Theta_{t})_{xy}=\mathrm{O}(|\mathrm{Im}w_{t}|^{-1}|\mathrm{Im}w_{s}|)$, since $\sum_{y}\mathrm{Id}_{xy}=1$. Recall that $w_{s}=-m(z)^{-1}-sm(z)$, so that $\mathrm{Im}w_{s}=\mathrm{Im}w_{t}+(t-s)\mathrm{Im}m(z)$. Since $|\mathrm{Im}m(z)|$ is bounded uniformly away from $0$ in the bulk (see Lemma 6.2 in \cite{EY}), we deduce that $|t-s|=\mathrm{O}(|\mathrm{Im}w_{s}|)$. Thus, to prove \eqref{eq:sstop1}, it suffices to prove that $\sum_{y}(\Theta_{t})_{xy}=\mathrm{O}(|\mathrm{Im}w_{t}|^{-1})$ (this can be found in Lemma \ref{lemma:thetaestimates}).

By \eqref{eq:sstop1} and matrix multiplication, for any indices $a,b$, we have 
\begin{align*}
|\mathcal{E}^{S,\mathrm{stop}}_{t}(z)_{ab}|&\leq\int_{0}^{\tau_{\mathrm{stop}}\wedge t}\left[1+\mathrm{O}(|\mathrm{Im}w_{t}|^{-1}|\mathrm{Im}w_{s}|)\right]^{2}\max_{a,b}|\mathcal{E}_{s}(z)^{2}_{ab}|ds\\
&\lesssim\int_{0}^{\tau_{\mathrm{stop}}\wedge t}\max_{a,b}|\mathcal{E}_{s}(z)^{2}_{ab}|ds+\int_{0}^{\tau_{\mathrm{stop}}\wedge t}|\mathrm{Im}w_{t}|^{-2}|\mathrm{Im}w_{s}|^{2}\max_{a,b}|\mathcal{E}_{s}(z)^{2}_{ab}|ds.
\end{align*}
By matrix multiplication, definition of $\tau_{\mathrm{stop}}$, and $N\leq W^{11/8-\upsilon}$, we have the bound below for any $s\leq\tau_{\mathrm{stop}}$:
\begin{align*}
\max_{a,b}|\mathcal{E}_{s}(z)^{2}_{ab}|&\leq N\max_{a,b}|\mathcal{E}_{s}(z)_{ab}|^{2}\leq W^{2\delta_{\mathrm{stop}}}NW^{-\frac32}|\mathrm{Im}w_{s}|^{-2}W^{-2}|\mathrm{Im}w_{s}|^{-1}\\
&\leq W^{2\delta_{\mathrm{stop}}}W^{-\frac{17}{8}-\upsilon}|\mathrm{Im}w_{s}|^{-3}.
\end{align*}
We combine the previous two displays to get
\begin{align*}
|\mathcal{E}^{S,\mathrm{stop}}_{t}(z)_{ab}|&\lesssim\int_{0}^{t}W^{2\delta_{\mathrm{stop}}}W^{-\frac{17}{8}-\upsilon}|\mathrm{Im}w_{s}|^{-3}ds+\int_{0}^{t}|\mathrm{Im}w_{t}|^{-2}W^{2\delta_{\mathrm{stop}}}W^{-\frac{17}{8}-\upsilon}|\mathrm{Im}w_{s}|^{-1}ds.
\end{align*}
We now change variables $\sigma=\mathrm{Im}w_{s}$. The Jacobian factor is $\mathrm{Im}m(z)\asymp1$ (see Lemma 6.2 in \cite{EY}). Moreover, for any $s\in[0,t]$, we have $\mathrm{Im}w_{s}=\mathrm{Im}w_{t}+(t-s)\mathrm{Im}m(z)\geq \mathrm{Im}w_{t}$ (which implies $\mathrm{Im}w_{s}=\eta+(1-s)\mathrm{Im}m(z)\leq\eta+\mathrm{Im}m(z)=\mathrm{Im}w_{0}\lesssim1$ by taking $t=1$); we explained this in the paragraph after \eqref{eq:sstop1}. Thus, we have 
\begin{align*}
|\mathcal{E}^{S,\mathrm{stop}}_{t}(z)_{ab}|&\lesssim\int_{|\mathrm{Im}w_{t}|}^{\mathrm{Im}w_{0}}W^{2\delta_{\mathrm{stop}}}W^{-\frac{17}{8}-\upsilon}\sigma^{-3}d\sigma+\int_{|\mathrm{Im}w_{t}|}^{\mathrm{Im}w_{0}}|\mathrm{Im}w_{t}|^{-2}W^{2\delta_{\mathrm{stop}}}W^{-\frac{17}{8}-\upsilon}\sigma^{-1}d\sigma\\
&\lesssim W^{2\delta_{\mathrm{stop}}}W^{-\frac{17}{8}-\upsilon}|\mathrm{Im}w_{t}|^{-2}+W^{2\delta_{\mathrm{stop}}}W^{-\frac{17}{8}-\upsilon}|\mathrm{Im}w_{t}|^{-2}\log|\mathrm{Im}w_{t}|^{-1}\\
&\lesssim W^{2\delta_{\mathrm{stop}}}W^{-\frac38-\upsilon}|\mathrm{Im}w_{t}|^{-\frac12}\log|\mathrm{Im}w_{t}|^{-1}\cdot W^{-\frac34}|\mathrm{Im}w_{t}|^{-1}\cdot W^{-1}|\mathrm{Im}w_{t}|^{-\frac12}.
\end{align*}
Because $\delta_{\mathrm{stop}}$ is small and $\upsilon>0$ is fixed, and since $|\mathrm{Im}w_{t}|\geq\eta\geq W^{-3/4}$, the first factor on the RHS, namely $W^{2\delta_{\mathrm{stop}}}W^{-3/8-\upsilon}|\mathrm{Im}w_{t}|^{-1/2}\log|\mathrm{Im}w_{t}|^{-1}$, is $\mathrm{o}(1)$, and the claim follows.
\end{proof}
We now tackle the stochastic integral $\mathcal{E}^{M,\mathrm{stop}}_{z}(t)$. For this, we will prove a pointwise moment bound and use a crude continuity estimate in time to bootstrap to a uniform estimate over $t\in[0,1]$ with high probability. Ultimately, we arrive at the following outcome.
\begin{lemma}\label{lemma:mstop}
Recall $\delta_{\mathrm{stop}}>0$ from the definition of $\tau_{\mathrm{stop}}$. We have the stochastic domination estimate
\begin{align}
\max_{t\in[0,1]}\max_{a,b}\frac{|\mathcal{E}^{M,\mathrm{stop}}_{t}(z)_{ab}|}{W^{-\frac34}|\mathrm{Im}w_{t}|^{-1}\cdot  W^{-1}|\mathrm{Im}w_{t}|^{-\frac12}}\prec W^{\frac{\delta_{\mathrm{stop}}}{5}}.\label{eq:mstop}
\end{align}
\end{lemma}
\begin{proof}
Recall $dM_{t}(z)$ from Lemma \ref{lemma:gsdes}; using this, we write
\begin{align*}
\mathcal{E}^{M,\mathrm{stop}}_{t}(z)&=-\int_{0}^{\tau_{\mathrm{stop}}\wedge t}\{\mathrm{Id}+(t-s)\Theta_{t}(z)\}S^{\frac12}\left[\overline{G}_{s}(z)\odot\{G_{s}(z)dH_{s}G_{s}(z)\}\right]S^{\frac12}\{\mathrm{Id}+(t-s)\Theta_{t}\}\\
&-\int_{0}^{\tau_{\mathrm{stop}}\wedge t}\{\mathrm{Id}+(t-s)\Theta_{t}(z)\}S^{\frac12}\left[\overline{G_{s}(z)dH_{s}G_{s}(z)}\odot G_{s}(z)\right]S^{\frac12}\{\mathrm{Id}+(t-s)\Theta_{t}\}\\
&=:\mathcal{E}^{M,1}_{t}(z)+\mathcal{E}^{M,2}_{t}(z).
\end{align*}
We will control $\mathcal{E}^{M,1}_{t}(z)$; bounds for $\mathcal{E}^{M,2}_{t}(z)$ follow by the same argument. We further split
\begin{align}
\mathcal{E}^{M,1}_{t}(z)&=-\int_{0}^{\tau_{\mathrm{stop}}\wedge t}(t-s)^{2}\Theta_{t}S^{\frac12}\left[\overline{G}_{s}(z)\odot\{G_{s}(z)dH_{s}G_{s}(z)\}\right]S^{\frac12}\Theta_{t}\\
&-\int_{0}^{\tau_{\mathrm{stop}}\wedge t}(t-s)\Theta_{t}S^{\frac12}\left[\overline{G}_{s}(z)\odot\{G_{s}(z)dH_{s}G_{s}(z)\}\right]S^{\frac12}\\
&-\int_{0}^{\tau_{\mathrm{stop}}\wedge t}(t-s)S^{\frac12}\left[\overline{G}_{s}(z)\odot\{G_{s}(z)dH_{s}G_{s}(z)\}\right]S^{\frac12}\Theta_{t}\\
&-\int_{0}^{\tau_{\mathrm{stop}}\wedge t}S^{\frac12}\left[\overline{G}_{s}(z)\odot\{G_{s}(z)dH_{s}G_{s}(z)\}\right]S^{\frac12}\\
&=:\mathcal{E}^{M,11}_{t}(z)+\mathcal{E}^{M,12}_{t}(z)+\mathcal{E}^{M,13}_{t}(z)+\mathcal{E}^{M,14}_{t}(z).\label{eq:mstop0}
\end{align}
We now claim that for any indices $a,b$, we have the deterministic estimate
\begin{align}
[\mathcal{E}^{M,1j}_{t}(z)_{ab}]&\lesssim W^{\frac{\delta_{\mathrm{stop}}}{5}}W^{-\frac32}|\mathrm{Im}w_{t}|^{-2}\cdot W^{-2}|\mathrm{Im}w_{t}|^{-1}, \quad j=1,2,3,4,\label{eq:mstop1}
\end{align}
where the square brackets on the RHS mean quadratic variation. This, along with standard martingale arguments, will complete the proof, as we explain now. By the BDG inequality, since $\tau_{\mathrm{stop}}$ is a stopping time, for any $p\geq1$ finite and indices $a,b$, we have
\begin{align*}
\E|\mathcal{E}^{M,1j}_{t}(z)_{ab}|^{2p}\leq C_{p}\E[\mathcal{E}^{M,1j}_{t}(z)_{ab}]^{p}, \quad j=1,2,3,4,
\end{align*}
where the square brackets on the RHS indicate quadratic variation. Thus, by the Chebyshev inequality, we have $\mathbb{P}(|\mathcal{E}^{M,1j}_{t}(z)_{ab}|\geq N^{\delta}[\mathcal{E}^{M,1j}_{t}(z)_{ab}]^{1/2})\leq C_{p}N^{-2p\delta}$ for any $\delta>0$ and $p\geq1$. By \eqref{eq:mstop1}, we then get
\begin{align*}
\frac{|\mathcal{E}^{M,1j}_{t}(z)_{ab}|}{W^{-\frac34}|\mathrm{Im}w_{t}|^{-1}\cdot W^{-1}|\mathrm{Im}w_{t}|^{-\frac12}}\prec W^{\frac{\delta_{\mathrm{stop}}}{5}}.
\end{align*}
A union bound allows us to put a maximum over all indices $a,b$ and over all $t\in[0,1]\cap W^{-D}\Z$ for any $D>0$ fixed. Next, we control entries of $G_{s}(z)$ using its Frobenius norm, for example. Since the operator norm of $G_{s}(z)$ is $\mathrm{O}(|\mathrm{Im}w_{s}|^{-1})$, we get the crude bound $|G_{s}(z)_{ij}|\lesssim N|\mathrm{Im}w_{s}|^{-1}\lesssim W^{4}$ (recall that $\mathrm{Im}w_{s}\geq\mathrm{Im}w_{1}=\eta\gg W^{-3/4}$ and $N\ll W^{11/8}$). This controls the stochastic integrand in $\mathcal{E}^{M,1j}_{t}(z)$ by $W^{C}$ for some $C=\mathrm{O}(1)$, so we can then use a standard H\"{o}lder continuity argument for stochastic integrals to upgrade the supremum over $t\in[0,1]\cap W^{-D}\Z$ to a supremum over $t\in[0,1]$ (assuming $D>0$ is large enough but independent of $W$). The desired bound \eqref{eq:mstop} then follows, so we are left to prove \eqref{eq:mstop1}.

Throughout this argument, for convenience, we will denote $G_{xy}:=G_{s}(z)_{xy}$ for any indices $x,y$. We start with $j=1$. In this case, we directly compute as follows (note that $\Theta_{t},S^{1/2}$ commute):
\begin{align*}
&[\mathcal{E}^{M,11}_{t}(z)_{ab}]\\
&=\int_{0}^{\tau_{\mathrm{stop}}\wedge t}(t-s)^{4}\sum_{x,x',y,y',u,v}(\Theta_{t}S^{\frac12})_{ax}(\Theta_{t}S^{\frac12})_{ax'}\overline{G}_{xy}G_{x'y'}G_{xu}\overline{G}_{x'u}S_{uv}G_{vy}\overline{G}_{vy'}(\Theta_{t}S^{\frac12})_{yb}(\Theta_{t}S^{\frac12})_{y'b}ds.
\end{align*}
For any $u$-index, we define the following matrices with indices parameterized by $y,y'$:
\begin{align*}
\Upsilon^{u}_{yy'}&:=\sum_{v}S_{uv}G_{vy}\overline{G}_{vy'}(\Theta_{t}S^{\frac12})_{yb}(\Theta_{t}S^{\frac12})_{y'b},\\
\Omega^{u}_{yy'}&:=\sum_{x,x'}\overline{G}_{xy}G_{x'y'}G_{xu}\overline{G}_{x'u}(\Theta_{t}S^{\frac12})_{ax}(\Theta_{t}S^{\frac12})_{ax'}.
\end{align*}
In this notation, we have 
\begin{align}
[\mathcal{E}^{M,11}_{t}(z)_{ab}]&=\int_{0}^{\tau_{\mathrm{stop}}\wedge t}(t-s)^{4}\sum_{u}\sum_{y}(\Omega^{u}\Upsilon^{u,\ast})_{yy}ds.\label{eq:mstop2a}
\end{align}
We now note that $\Omega^{u}$ is a positive-semidefinite matrix; this can be readily verified by the observation that $\Omega^{u}$ has the form $G^{\ast}\mathcal{Q}GG^{\ast}\mathcal{Q}G$, where $\mathcal{Q}$ is a diagonal matrix with entries $\mathcal{Q}_{ij}=\delta_{ij}(\Theta_{t}S^{1/2})_{ai}$. By the von Neumann trace inequality, we can therefore bound the sum over $y$ by $\|\Upsilon^{u}\|_{\mathrm{op}}\sum_{y}\Omega^{u}_{yy}$, so that
\begin{align}
[\mathcal{E}^{M,11}_{t}(z)_{ab}]\leq\int_{0}^{\tau_{\mathrm{stop}}\wedge t}(t-s)^{4}\sum_{u}\|\Upsilon^{u}\|_{\mathrm{op}}\sum_{y}\Omega^{u}_{yy}ds.\label{eq:mstop2b}
\end{align}
First, we estimate as follows, in which $S^{u}$ is the diagonal matrix $S^{u}_{ij}=\delta_{ij}S_{ui}=\mathrm{O}(W^{-1})$:
\begin{align}
\|\Upsilon^{u}\|_{\mathrm{op}}&\leq\max_{\alpha,\beta}|(\Theta_{t}S^{\frac12})_{\alpha\beta}|^{2}\|G^{\ast}S^{u}G\|_{\mathrm{op}}\leq W^{-1}|\mathrm{Im}w_{s}|^{-2}\max_{\alpha,\beta}|(\Theta_{t}S^{\frac12})_{\alpha\beta}|^{2}. \label{eq:mstop3a}
\end{align}
Next, if we let $|G|^{2}=GG^{\ast}$ and $|G|^{4}=|G|^{2}|G|^{2}$, then we claim that the following holds: 
\begin{align}
\sum_{u,y}\Omega^{u}_{yy}&=\sum_{u,y,x,x'}\overline{G}_{xy}G_{x'y}G_{xu}\overline{G}_{x'u}(\Theta_{t}S^{\frac12})_{ax}(\Theta_{t}S^{\frac12})_{ax'}\label{eq:mstop3b}\\
&=\sum_{u,x,x'}(\Theta_{t}S^{\frac12})_{ax}(\Theta_{t}S^{\frac12})_{ax'}G_{xu}|G|^{2}_{x'x}\overline{G}_{x'u}\label{eq:mstop3c}\\
&=\sum_{x,x'}(\Theta_{t}S^{\frac12})_{ax}(\Theta_{t}S^{\frac12})_{ax'}|G|^{2}_{xx'}|G|^{2}_{x'x}\label{eq:mstop3d}\\
&\leq2\sum_{x,x'}|(\Theta_{t}S^{\frac12})_{ax}|^{2}|G|^{2}_{xx'}|G|^{2}_{x'x}+2\sum_{x,x'}|(\Theta_{t}S^{\frac12})_{ax'}|^{2}|G|^{2}_{xx'}|G|^{2}_{x'x}\label{eq:mstop3e}\\
&=4\sum_{x}|(\Theta_{t}S^{\frac12})_{ax}|^{2}|G|^{4}_{xx}\lesssim|\mathrm{Im}w_{s}|^{-3}\max_{k}\mathrm{Im}G_{kk}\sum_{x}|(\Theta_{t}S^{\frac12})_{ax}|^{2}\label{eq:mstop3f}\\
&\lesssim W^{\delta_{\mathrm{stop}}}|\mathrm{Im}w_{s}|^{-3}W^{-1}|\mathrm{Im}w_{t}|^{-\frac32},\label{eq:mstop3g}
\end{align}
The first three lines follow by matrix multiplication. The fourth line follows by Schwarz; for this, note that since $|G|^{2}$ is Hermitian, we have that $|G|^{2}_{xx'}|G|^{2}_{x'x}=||G|^{2}_{xx'}|^{2}$ is non-negative. The fifth line follows by matrix multiplication and $|G|^{4}_{xx}\leq|\mathrm{Im}w_{s}|^{-3}\max_{k}\mathrm{Im}G_{kk}$ (i.e. Ward). The last inequality uses $s\leq\tau_{\mathrm{stop}}$ in the quadratic variation, so that $|G_{kk}|$-terms are all $\mathrm{O}(W^{\delta_{\mathrm{stop}}})$. In the last inequality, we also use the bounds {$\max_{\alpha,\beta}|(\Theta_{t})_{\alpha\beta}|\lesssim W^{-1}|\mathrm{Im}w_{t}|^{-1/2}$} and {$\sum_{\beta}|(\Theta_{t}S^{1/2})_{\alpha\beta}|\lesssim |\mathrm{Im}w_{t}|^{-1}$} (see Lemma \ref{lemma:thetaestimates}). By the previous three displays, and recalling from the proof of Lemma \ref{lemma:sstop} that $(t-s)\lesssim|\mathrm{Im}w_{s}|$, we have
\begin{align}
[\mathcal{E}^{M,11}_{t}(z)_{ab}]\lesssim\int_{0}^{t}|\mathrm{Im}w_{s}|^{-1}W^{\delta_{\mathrm{stop}}}W^{-2}|\mathrm{Im}w_{t}|^{-\frac32}\max_{\alpha,\beta}|(\Theta_{t}S^{\frac12})_{\alpha\beta}|^{2}ds.\label{eq:mstop4a}
\end{align}
Again, as in the proof of Lemma \ref{lemma:sstop}, we can change variables $\sigma=|\mathrm{Im}w_{s}|$ to get
\begin{align}
[\mathcal{E}^{M,11}_{t}(z)_{ab}]\lesssim W^{\delta_{\mathrm{stop}}}W^{-2}|\mathrm{Im}w_{t}|^{-\frac32}\max_{\alpha,\beta}|(\Theta_{t}S^{\frac12})_{\alpha\beta}|^{2}\int_{\mathrm{Im}w_{t}}^{\mathrm{Im}w_{0}}\sigma^{-1}\lesssim W^{-4+\delta+\delta_{\mathrm{stop}}}|\mathrm{Im}w_{t}|^{-\frac52},\label{eq:mstop4b}
\end{align}
where the last bound follows since $\log|\mathrm{Im}w_{t}|^{-1}\lesssim\log\eta^{-1}\lesssim W^{\delta}$ for any $\delta>0$ and {$\max_{\alpha,\beta}|(\Theta_{t}S^{\frac12})_{\alpha\beta}|^{2}\lesssim W^{-2}|\mathrm{Im}w_{t}|^{-1}$ and $|\mathrm{Im}w_{t}|=\mathrm{O}(1)$}. (For these latter two bounds, see Lemma \ref{lemma:thetaestimates} and note that $\mathrm{Im}w_{t}=\eta+(1-t)\mathrm{Im}m(z)\lesssim1$.) If $\delta,\delta_{\mathrm{stop}}$ are small enough, the previous display immediately gives \eqref{eq:sstop1} for $j=1$.

We now move to $j=3$. The idea again is the same. First, we compute
\begin{align}
[\mathcal{E}^{M,13}_{t}(z)_{ab}]&=\int_{0}^{\tau_{\mathrm{stop}}\wedge t}(t-s)^{2}\sum_{x,x',y,y',u,v}S^{\frac12}_{ax}S^{\frac12}_{ax'}\overline{G}_{xy}G_{x'y'}G_{xu}\overline{G}_{x'u}S_{uv}G_{vy}\overline{G}_{vy'}(\Theta_{t}S^{\frac12})_{yb}(\Theta_{t}S^{\frac12})_{y'b}ds.\label{eq:mstop5a}
\end{align}
We now consider the following matrices for any $u$, which are indexed by $y,y'$:
\begin{align*}
\Upsilon^{u}_{yy'}&:=\sum_{v}S_{uv}G_{vy}\overline{G}_{vy'}(\Theta_{t}S^{\frac12})_{yb}(\Theta_{t}S^{\frac12})_{y'b},\\
\Xi^{u}_{yy'}&:=\sum_{x,x'}\overline{G}_{xy}G_{x'y'}G_{xu}\overline{G}_{x'u}S^{\frac12}_{ax}S^{\frac12}_{ax'}.
\end{align*}
The matrix $\Upsilon^{u}$ is the same as in our analysis for $j=1$, and $\Xi^{u}$ is the same as $\Omega^{u}$ but dropping the $\Theta_{t}$ matrices therein. With this notation, as in our estimates for $[\mathcal{E}^{M,11}_{t}(z)_{ab}]$ (see \eqref{eq:mstop2a}, \eqref{eq:mstop2b}, and \eqref{eq:mstop3a}), we have the following (note that $\Xi^{u}$ is also positive-semidefinite):
\begin{align*}
[\mathcal{E}^{M,13}_{t}(z)_{ab}]&\leq\int_{0}^{\tau_{\mathrm{stop}}\wedge t}(t-s)^{2}\sum_{u}\|\Upsilon^{u}\|_{\mathrm{op}}\sum_{y}\Xi^{u}_{yy}ds\\
&\leq W^{-1}\max_{\alpha,\beta}|(\Theta_{t})_{\alpha\beta}|^{2}\int_{0}^{\tau_{\mathrm{stop}}\wedge t}(t-s)^{2}|\mathrm{Im}w_{s}|^{-2}\sum_{u,y}\Xi^{u}_{yy}ds.
\end{align*}
To control the remaining summation over $u,y$, we follow exactly the calculation from \eqref{eq:mstop3b}-\eqref{eq:mstop3g}. The only difference is that all $\Theta_{t}S^{1/2}$-matrices are replaced by $S^{1/2}$. In particular, we can copy \eqref{eq:mstop3b}-\eqref{eq:mstop3f} verbatim, and then we can use that $S^{1/2}_{ab}=\mathrm{O}(W^{-1})$ and $\sum_{b}S^{1/2}_{ab}=\mathrm{O}(1)$ to get
\begin{align}
\sum_{u,y}\Xi^{u}_{yy}&\lesssim|\mathrm{Im}w_{s}|^{-3}\max_{k}\mathrm{Im}G_{kk}\sum_{x}|S^{\frac12}_{ax}|^{2}\lesssim W^{\frac{\delta_{\mathrm{stop}}}{10}}W^{-1}|\mathrm{Im}w_{s}|^{-3}.\label{eq:mstop5}
\end{align}
Combining the last two displays and using the change-of-variables $\sigma=\mathrm{Im}w_{s}$ (as in \eqref{eq:mstop4a} and \eqref{eq:mstop4b}) gives
\begin{align}
[\mathcal{E}^{M,13}_{t}(z)_{ab}|&\lesssim W^{\frac{\delta_{\mathrm{stop}}}{10}}W^{-2}\max_{\alpha,\beta}|(\Theta_{t}S^{\frac12})_{\alpha\beta}|^{2}\int_{0}^{t}(t-s)^{2}|\mathrm{Im}w_{s}|^{-5}ds\label{eq:mstop5b}\\
&\lesssim W^{\frac{\delta_{\mathrm{stop}}}{10}}W^{-2}\max_{\alpha,\beta}|(\Theta_{t}S^{\frac12})_{\alpha\beta}|^{2}\int_{\mathrm{Im}w_{t}}^{\mathrm{Im}w_{0}}\sigma^{-3}d\sigma\nonumber\\
&\lesssim W^{\frac{\delta_{\mathrm{stop}}}{10}}W^{-2}W^{-2}|\mathrm{Im}w_{t}|^{-1}|\mathrm{Im}w_{t}|^{-2},\nonumber
\end{align}
where the last line again uses {$\max_{\alpha,\beta}|(\Theta_{t}S^{1/2})_{\alpha\beta}|^{2}\lesssim W^{-2}|\mathrm{Im}w_{t}|^{-1}$} (see Lemma \ref{lemma:thetaestimates}). If we take $\delta_{\mathrm{stop}}>0$ small enough, then the last line of the above display is $\ll W^{-3/2}|\mathrm{Im}w_{t}|^{-2}W^{-2}|\mathrm{Im}w_{t}|^{-1}$ since $\mathrm{Im}w_{t}\geq\eta\gg W^{-3/4}$ (see Remark \ref{rem:eta}). Thus, \eqref{eq:mstop1} for $j=3$ follows. The case $j=2$ follows by the exact same argument if we swap $u$ with $v$ indices, and if we swap $(x,x')$ with $(y,y')$ indices, and if we swap $a$ with $b$ indices.

Thus, we are left with $j=4$. We will get two bounds on $[\mathcal{E}^{M,14}_{t}(z)_{ab}]$ and interpolate them. For the first bound, we will follow the same strategy as we used for $j=1,2,3$. To start, we compute
\begin{align*}
[\mathcal{E}^{M,14}_{t}(z)_{ab}]&=\int_{0}^{\tau_{\mathrm{stop}}\wedge t}\sum_{x,x',y,y',u,v}S^{\frac12}_{ax}S^{\frac12}_{ax'}\overline{G}_{xy}G_{x'y'}G_{xu}\overline{G}_{x'u}S_{uv}G_{vy}\overline{G}_{vy'}S^{\frac12}_{yb}S^{\frac12}_{y'b}ds.
\end{align*}
Now, define the following matrices for any $u$, which are indexed by $y,y'$:
\begin{align*}
\Gamma^{u}_{yy'}&:=\sum_{v}S_{uv}G_{vy}\overline{G}_{vy'}S^{\frac12}_{yb}S^{\frac12}_{y'b},\\
\Xi^{u}_{yy'}&:=\sum_{x,x'}\overline{G}_{xy}G_{x'y'}G_{xu}\overline{G}_{x'u}S^{\frac12}_{ax}S^{\frac12}_{ax'}.
\end{align*}
In this notation, we have 
\begin{align*}
[\mathcal{E}^{M,14}_{t}(z)_{ab}]&=\int_{0}^{\tau_{\mathrm{stop}}\wedge t}\sum_{u,y}(\Xi^{u}\Gamma^{u,\ast})_{yy}ds.
\end{align*}
Again, we note that $\Xi^{u}$ is positive semi-definite. This implies that $\sum_{y}(\Xi^{u}\Gamma^{u,\ast})_{yy}\leq\|\Gamma^{u}\|_{\mathrm{op}}\sum_{y}\Xi^{u}_{yy}$. Next, analogous to \eqref{eq:mstop3a}, we note that $\|\Gamma^{u}\|_{\mathrm{op}}\leq \max_{\alpha,\beta}|S^{1/2}_{\alpha\beta}|^{2}\|G^{\ast}S^{u}G\|_{\mathrm{op}}\lesssim W^{-3}|\mathrm{Im}w_{s}|^{-2}$. If we combine this paragraph with the previous display and \eqref{eq:mstop5}, then we get
\begin{align}
[\mathcal{E}^{M,14}_{t}(z)_{ab}]&\lesssim\int_{0}^{t}W^{\frac{\delta_{\mathrm{stop}}}{10}}W^{-4}|\mathrm{Im}w_{s}|^{-5}ds\nonumber\\
&\lesssim W^{\frac{\delta_{\mathrm{stop}}}{10}}W^{-4}\int_{\mathrm{Im}w_{t}}^{\mathrm{Im}w_{0}}\sigma^{-5}d\sigma\lesssim W^{\frac{\delta_{\mathrm{stop}}}{10}}W^{-4}|\mathrm{Im}w_{t}|^{-4}.\label{eq:mstop6}
\end{align}
The second line follows again by the change-of-variables $\sigma=\mathrm{Im}w_{s}$ as in \eqref{eq:mstop4b}. We now get a second bound on $[\mathcal{E}^{M,14}_{t}(z)_{ab}]$. Recall $\mathcal{E}^{M,14}_{t}(z)$ from \eqref{eq:mstop0}. Because $S^{1/2}$ is an averaging operator, when we compute estimate the bracket $[\mathcal{E}^{M,14}_{t}(z)_{ab}]$, we can drop the second $S^{1/2}$-factor in $\mathcal{E}^{M,14}_{t}(z)$. (In words, the quadratic variation of an average is bounded above by the average of the quadratic variations because of Schwarz.) So,
\begin{align}
[\mathcal{E}^{M,14}_{t}(z)_{ab}]&\leq\max_{\alpha,\beta}\left[\left\{\int_{0}^{\tau_{\mathrm{stop}}\wedge t}S^{\frac12}\left[\overline{G}_{s}(z)\odot\{G_{s}(z)dH_{s}G_{s}(z)\}\right]\right\}_{\alpha\beta}\right]\nonumber\\
&=\max_{\alpha,\beta}\int_{0}^{\tau_{\mathrm{stop}}\wedge t}\sum_{x,x',u,v}S^{\frac12}_{\alpha x}S^{\frac12}_{\alpha x'}\overline{G}_{x\beta}G_{x'\beta}G_{xu}\overline{G}_{x'u}S_{uv}|G_{v\beta}|^{2}ds.\label{eq:mstop7a}
\end{align}
Before time $\tau_{\mathrm{stop}}$, we can bound $|T_{s}(z)_{xy}|\lesssim|(\Theta_{s})_{xy}|+|\mathcal{E}_{s}(z)_{xy}|\lesssim W^{-1}|\mathrm{Im}w_{s}|^{-1/2}$. Before time $\tau_{\mathrm{stop}}$, this bound on $|T_{s}(z)_{xy}|$ also gives $\sum_{v}S_{uv}|G_{v\beta}|^{2}\lesssim W^{\delta_{\mathrm{stop}}/10}S_{u\beta}+W^{\delta_{\mathrm{stop}}/10}[S^{1/2}T_{s}(z)S^{1/2}]_{u\beta}\lesssim W^{\delta_{\mathrm{stop}}/10}S_{u\beta}+W^{\delta_{\mathrm{stop}}/10}W^{-1}|\mathrm{Im}w_{s}|^{-1/2}\lesssim W^{\delta_{\mathrm{stop}}/10}W^{-1}|\mathrm{Im}w_{s}|^{-1/2}$ since $\mathrm{Im}w_{s}=\mathrm{O}(1)$. Using this and the previous display, we claim
\begin{align}
[\mathcal{E}^{M,14}_{t}(z)_{ab}]&\lesssim\max_{\alpha,\beta}\int_{0}^{\tau_{\mathrm{stop}}\wedge t}W^{\frac{\delta_{\mathrm{stop}}}{10}}W^{-1}|\mathrm{Im}w_{s}|^{-\frac12}\sum_{x,x',u}S^{\frac12}_{\alpha x}S^{\frac12}_{\alpha x'}\overline{G}_{x\beta}G_{x'\beta}G_{xu}\overline{G}_{x'u}ds\nonumber\\
&=\max_{\alpha,\beta}\int_{0}^{\tau_{\mathrm{stop}}\wedge t}W^{\frac{\delta_{\mathrm{stop}}}{10}}W^{-1}|\mathrm{Im}w_{s}|^{-\frac12}(G^{\ast}S^{\frac12,\alpha}GG^{\ast}S^{\frac12,\alpha}G)_{\beta\beta}ds,\label{eq:mstop7b}
\end{align}
where $S^{1/2,\alpha}_{ij}=\delta_{ij}S^{1/2}_{\alpha i}$. The second line is a direct matrix multiplication calculation. To show the first line, we note that for any $u$, we have 
\begin{align*}
\sum_{x,x'}S^{1/2}_{\alpha x}S^{1/2}_{\alpha x'}\overline{G}_{x\beta}G_{x'\beta}G_{xu}\overline{G}_{x'u}=(G^{\ast}S^{\frac12,\alpha}GG^{\ast}S^{\frac12,\alpha}G)_{uu}
\end{align*}
which is non-negative since it is a diagonal entry of a covariance matrix. Therefore, in \eqref{eq:mstop7a}, we can bound the $\sum_{v}S_{uv}|G_{v\beta}|^{2}$-term as described after \eqref{eq:mstop7a} without the need to introduce absolute values. Thus, \eqref{eq:mstop7b} follows. We now control the $(\ldots)_{\beta\beta}$-term in \eqref{eq:mstop7b}. By Ward and general operator bounds, we have 
\begin{align}
(G^{\ast}S^{\frac12,\alpha}GG^{\ast}S^{\frac12,\alpha}G)_{\beta\beta}&=|\mathrm{Im}w_{s}|^{-1}(G^{\ast}S^{\frac12,\alpha}\mathrm{Im}GS^{\frac12,\alpha}G)_{\beta\beta}\label{eq:mstop7c}\\
&\lesssim|\mathrm{Im}w_{s}|^{-1}\|\sqrt{S^{\frac12,\alpha}}\mathrm{Im}G\sqrt{S^{\frac12,\alpha}}\|_{\mathrm{op}}(G^{\ast}S^{\frac12,\alpha}G)_{\beta\beta}\nonumber\\
&\lesssim W^{-1}|\mathrm{Im}w_{s}|^{-2}(G^{\ast}S^{\frac12,\alpha}G)_{\beta\beta},\nonumber
\end{align}
where $\sqrt{S^{1/2,\alpha}}$ is the matrix square root of $S^{1/2,\alpha}$. The last line follows by a direct bound on the entries of the diagonal matrix $\sqrt{S^{1/2,\alpha}}$ of $\mathrm{O}(W^{-1/2})$, as well as the operator norm bound $\|G\|_{\mathrm{op}}\leq|\mathrm{Im}w_{s}|^{-1}$. Next, observe that 
\begin{align}
(G^{\ast}S^{\frac12,\alpha}G)_{\beta\beta}&=\sum_{\gamma}S^{\frac12}_{\alpha\gamma}|G_{\gamma\beta}|^{2}\lesssim W^{\frac{\delta_{\mathrm{stop}}}{10}}W^{-1}|\mathrm{Im}w_{s}|^{-\frac12},\label{eq:mstop7d}
\end{align}
where the second bound is explained immediately after \eqref{eq:mstop7a}. If we combine the previous two displays with \eqref{eq:mstop7b}, we get
\begin{align}
[\mathcal{E}^{M,14}_{t}(z)_{ab}]&\lesssim\int_{0}^{t}W^{\frac{\delta_{\mathrm{stop}}}{5}}W^{-3}|\mathrm{Im}w_{s}|^{-3}ds\nonumber\\
&\lesssim\int_{\mathrm{Im}w_{t}}^{\mathrm{Im}w_{0}}W^{\frac{\delta_{\mathrm{stop}}}{5}}W^{-3}\sigma^{-3}d\sigma\lesssim W^{\frac{\delta_{\mathrm{stop}}}{5}}W^{-3}|\mathrm{Im}w_{t}|^{-2}.\label{eq:mstop8}
\end{align}
Again, the last line is by change-of-variables $\sigma=\mathrm{Im}w_{s}$. We now interpolate the bounds \eqref{eq:mstop6} and \eqref{eq:mstop8}:
\begin{align*}
[\mathcal{E}^{M,14}_{t}(z)_{ab}]&\lesssim\sqrt{W^{\frac{\delta_{\mathrm{stop}}}{10}}W^{-4}|\mathrm{Im}w_{t}|^{-4}}\sqrt{W^{\frac{\delta_{\mathrm{stop}}}{5}}W^{-3}|\mathrm{Im}w_{t}|^{-2}}\\
&\lesssim W^{\frac{\delta_{\mathrm{stop}}}{5}}W^{-\frac72}|\mathrm{Im}w_{t}|^{-3}\\
&=W^{\frac{\delta_{\mathrm{stop}}}{5}}W^{-\frac32}|\mathrm{Im}w_{t}|^{-2}\cdot W^{-2}|\mathrm{Im}w_{t}|^{-1}
\end{align*}
This is exactly \eqref{eq:mstop1} for $j=4$, so we are done.
\end{proof}
The final ingredient is to control the drift $\mathcal{E}^{D,\mathrm{stop}}_{t}(z)$. This is more complicated, as it requires integration-by-parts expansions of the integrand. (These expansions, fortunately, are concrete, and only two expansions are needed.) We state the estimate below, and then we defer its proof to a future section.
\begin{prop}\label{prop:dstop}
We have the stochastic domination estimate
\begin{align}
\max_{t\in[0,1]}\max_{a,b}\frac{|\mathcal{E}^{D,\mathrm{stop}}_{t}(z)_{ab}|}{W^{-\frac34}|\mathrm{Im}w_{t}|^{-1}\cdot  W^{-1}|\mathrm{Im}w_{t}|^{-\frac12}}\prec 1.\label{eq:dstop}
\end{align}
\end{prop}
\subsection{Basic Green's function estimates}
We now focus on estimates to propagate $\tau_{\mathrm{stop},2}$. First, we obtain an off-diagonal estimate assuming a weak a priori on-diagonal estimate.
\begin{lemma}\label{lemma:basicgf}
Fix any $t\in[0,1]$. Suppose that $\max_{x}|G_{t}(z)_{xx}|\prec1$ and $\max_{x\neq y}|G_{t}(z)_{xy}|\prec W^{-\delta}$ for some $\delta>0$. Then 
\begin{align}
\max_{a\neq b}\frac{|G_{t}(z)_{ab}|^{2}}{(S^{1/2}T_{t}(z)S^{1/2})_{ab}+S^{1/2}_{ab}}\prec1.
\end{align}
\end{lemma}
\begin{proof}
First, some notation. For any index $\alpha$, we let $G_{t}(z)^{(\alpha)}:=(H_{t}^{(\alpha)}-w_{t})^{-1}$, where $H_{t}^{(\alpha)}$ is the minor of $H_{t}$ obtained by removing the $\alpha$-row and $\alpha$-column. Also, for convenience, we will denote $G^{(\alpha)}_{xy}:=G_{t}(z)^{(\alpha)}_{xy}$. We record the following observations from the proof of Lemma 3.7 in \cite{YY21}:
\begin{align*}
|G_{xy}|^{2}\prec|G_{xx}|^{2}\sum_{w\neq x}S_{xw}|G^{(x)}_{wy}|^{2}&\lesssim|G_{xx}|^{2}\sum_{w\neq x}S_{xy}|G_{wy}|^{2}+|G_{xy}|^{2}\sum_{w\neq x}S_{xw}|G_{wx}|^{2}\\
&\prec\sum_{w\neq x}S_{xw}|G_{wy}|^{2}+W^{-2\delta}|G_{xy}|^{2},
\end{align*}
where the last bound follows by the assumptions $|G_{xx}|\prec1$ and $\sup_{w\neq x}|G_{wx}|^{2}\prec W^{-\delta}$. (In principle, every factor of and $S$-entry above should have a factor of $t$, since the entries of $H_{t}$ have variance $t$. But for the sake of an upper bound, we can drop all factors of $t$.) By the same token, but for $G^{\ast}$ instead of $G$, for any $w\neq y$, we have 
\begin{align*}
|G_{wy}|^{2}&=|G^{\ast}_{yw}|^{2}\prec|G^{\ast}_{yy}|^{2}\sum_{u\neq y}S_{yu}|G^{\ast,(y)}_{uw}|^{2}\\
&\lesssim|G^{\ast}_{yy}|^{2}\sum_{u\neq y}S_{yu}|G^{\ast}_{uw}|^{2}+|G^{\ast}_{yw}|^{2}\sum_{u\neq y}S_{yu}|G^{\ast}_{uy}|^{2}\\
&=|G_{yy}|^{2}\sum_{u\neq y}|G_{wu}|^{2}S_{uy}+|G_{wy}|^{2}\sum_{u\neq y}|G_{yu}|^{2}S_{uy}.
\end{align*}
If we now multiply by $S_{xw}$ and sum over all indices $w$ that are not equal to $x$, then we get the following, where the first line follows by first isolating $w=y$, where the second line follows by applying the previous display and re-inserting $w=y$ (which does not violate the upper bound being true since the summands are non-negative), and where the rest uses the a priori bounds from the statement of Lemma \ref{lemma:basicgf}:
\begin{align*}
\sum_{w\neq x}S_{xw}|G_{wy}|^{2}&\prec S_{xy}|G_{yy}|^{2}+\sum_{w\neq x,y}S_{xw}|G_{wy}|^{2}\\
&\lesssim S_{xy}|G_{yy}|^{2}+\sum_{w\neq x}S_{xw}|G_{yy}|^{2}\sum_{u\neq y}|G_{wu}|^{2}S_{uy}+\sum_{w\neq x}S_{xw}|G_{wy}|^{2}\sum_{u\neq y}|G_{yu}|^{2}S_{uy}\\
&\prec S_{xy}|G_{yy}|^{2}+\sum_{w,u}S_{xw}|G_{wu}|^{2}S_{uy}+\sum_{w\neq x}S_{xw}|G_{wy}|^{2}\sum_{u\neq y}|G_{yu}|^{2}S_{uy}\\
&\prec S_{xy}+\sum_{w,u}S_{xw}|G_{wu}|^{2}S_{uy}+W^{-2\delta}\sum_{w\neq x}S_{xw}|G_{wy}|^{2}\\
&=S_{xy}+(S^{\frac12}TS^{\frac12})_{xy}+W^{-2\delta}\sum_{w\neq x}S_{xw}|G_{wy}|^{2}.
\end{align*}
In the previous display, if we move the last term in the last line to the LHS of the first line, then we deduce
\begin{align*}
\sum_{w\neq x}S_{xw}|G_{wy}|^{2}\prec S_{xy}+(S^{\frac12}TS^{\frac12})_{xy}.
\end{align*}
Combining this with the first display of this proof gives
\begin{align*}
|G_{xy}|^{2}\prec S_{xy}+(S^{\frac12}TS^{\frac12})_{xy}+W^{-2\delta}|G_{xy}|^{2},
\end{align*}
which, by moving the last term onto the LHS, implies $\{S_{xy}+(S^{\frac12}TS^{\frac12})_{xy}\}^{-1}|G_{xy}|^{2}\prec1$. We can now take a union bound over indices to get the desired high probability bound for the maximum over all $x\neq y$.
\end{proof}
Next, we obtain an on-diagonal estimate assuming off-diagonal estimates and weak on-diagonal estimates.
\begin{lemma}
Suppose that we have the following two a priori estimates for some $\delta>0$:
\begin{align*}
\max_{a,b}|G_{t}(z)_{ab}-m(z)\delta_{ab}|&\leq W^{-\delta}\\
\max_{a\neq b}|G_{t}(z)_{ab}|&\prec \max_{a,b}|(S^{\frac12}T_{t}(z)S^{\frac12})_{ab}|+\max_{a,b}|S^{\frac12}_{ab}|.
\end{align*}
Then we have 
\begin{align*}
\max_{x}|G_{t}(z)_{xx}-m(z)|^{2}\prec \max_{a,b}|(S^{\frac12}T_{t}(z)S^{\frac12})_{ab}|+\max_{a,b}|S^{\frac12}_{ab}|+W^{-1}.
\end{align*}
\end{lemma}
\begin{proof}
We first record an observation from the proof of Lemma 5.3 in \cite{EKYY}. By (5.11) therein and the two displays following it, we have 
\begin{align*}
|G_{xx}-m|^{2}\prec\left|\sum_{k\neq x}\left(|H_{xk}|^{2}-S_{xk}\right)G_{kk}^{(x)}\right|^{2}+\left|\sum_{\substack{k,\ell\neq x\\k\neq\ell}}H_{xk}G_{k\ell}^{(x)}H_{\ell x}\right|^{2}+\Psi^{4}+M^{-1},
\end{align*}
where $M\gtrsim W$ according to (2.10) in \cite{EKYY}, and where $\Psi=\max_{a,b}|G_{ab}-m\delta_{ab}|$ as in the statement of Lemma 5.3 in \cite{EKYY}. Now, we use resolvent identities (see Lemma 3.3 in \cite{YY21}) to get
\begin{align*}
|G^{(x)}_{k\ell}|&\leq|G_{k\ell}|+\frac{|G_{kx}||G_{x\ell}|}{|G_{xx}|}, \quad \forall k,\ell.
\end{align*}
By the a priori assumption $\max_{a,b}|G_{ab}-m\delta_{ab}|\leq W^{-\delta}$ for some $\delta>0$, along with $|G_{xx}|\gtrsim|G_{xx}-m(z)|+|m(z)|\succ|m(z)|\asymp 1$, the previous inequality implies that $|G_{k\ell}^{(x)}|\prec1$. (For $|m(z)|\asymp1$, it suffices to prove that $\mathrm{Im}m(z)\asymp1$ in the bulk, which can be found in Lemma 6.2 in \cite{EY}.) Thus, following the proof of Lemma 5.3 in \cite{EKYY}, we have the stochastic domination bound
\begin{align*}
\left|\sum_{k\neq x}\left(|H_{xk}|^{2}-S_{xk}\right)G_{kk}^{(x)}\right|^{2}\prec W^{-1}.
\end{align*}
The proof of Lemma 5.3 in \cite{EKYY} also gives the following stochastic domination bound:
\begin{align*}
\left|\sum_{\substack{k,\ell\neq x\\k\neq\ell}}H_{xk}G_{k\ell}^{(x)}H_{\ell x}\right|^{2}&\prec\sum_{\substack{k,\ell\neq x\\k\neq\ell}}S_{xk}|G_{k\ell}^{(x)}|^{2}S_{\ell x}\\
&\lesssim\sum_{\substack{k,\ell\neq x\\k\neq\ell}}S_{xk}|G_{k\ell}|^{2}S_{\ell x}+|G_{xx}|^{-2}\sum_{\substack{k,\ell\neq x\\k\neq\ell}}S_{xk}|G_{kx}|^{2}|G_{x\ell}|^{2}S_{\ell x}\\
&\lesssim\max_{a,b}|(S^{\frac12}TS^{\frac12})_{ab}|+\max_{a,b}|S^{\frac12}_{ab}|,
\end{align*}
where the last line uses the a priori off-diagonal estimate for $G_{ab}$-entries, the lower bound $|G_{xx}|\gtrsim|m|-|G_{xx}-m|$, the estimate $|m|\asymp1$, and the a priori bound $|G_{xx}-m|\leq W^{-\delta}$ for some $\delta>0$. If we combine the previous two displays with the first display of this proof, then we arrive at 
\begin{align*}
\max_{x}|G_{xx}-m|^{2}&\prec \max_{a,b}|(S^{\frac12}TS^{\frac12})_{ab}|+\max_{a,b}|S^{\frac12}_{ab}|+\max_{a,b}|G_{ab}-m\delta_{ab}|^{4}+W^{-1}\\
&\prec\max_{a,b}|(S^{\frac12}TS^{\frac12})_{ab}|+\max_{a,b}|S^{\frac12}_{ab}|+W^{-\delta}\max_{x}|G_{xx}-m|^{2}+W^{-1},
\end{align*}
where the last bound holds by $\max_{a,b}|G_{ab}-m\delta_{ab}|^{4}\lesssim\max_{a\neq b}|G_{ab}|^{4}+W^{-\delta}\max_{x}|G_{xx}-m|^{2}$ and the a priori estimates $\max_{a\neq b}|G_{ab}|\lesssim \max_{a,b}|(S^{\frac12}TS^{\frac12})_{ab}|+\max_{a,b}|S^{\frac12}_{ab}|$ and $\max_{a,b}|G_{ab}-m\delta_{ab}|\lesssim W^{-\delta}$. Moving $W^{-\delta}\max_{x}|G_{xx}-m|^{2}$ to the LHS in the previous display gives the desired result.
\end{proof}
Notice that $|G_{t}(z)\delta_{ab}-m(z)\delta_{ab}|=0$ at $t=0$ by construction of the matrix flow. In particular, the a priori estimates needed in the previous lemmas are true at $t=0$ trivially, and thus true for short times $t=N^{-C}$ for some large but finite $C>0$ by a deterministic Lipschitz bound on the Green's function and the $w_{t}$-flow. Therefore, similar to Corollary 5.4 in \cite{EKYY}, we can use a continuity argument and the previous two lemmas to derive the following (in which the $W^{-C}$ term comes from bounding the short-time difference of $G_{t}(z)$ via its Lipschitz norm as discussed above).
\begin{corollary}\label{corollary:gfcont}
First, let $\mathcal{E}$ be the event where
\begin{align*}
\max_{a,b}(S^{\frac12}T_{t}(z)S^{\frac12})_{ab}\lesssim W^{-1}|\mathrm{Im}w_{t}|^{-\frac12}, \quad\forall t\in[0,1].
\end{align*}
(Observe that $W^{-1}|\mathrm{Im}w_{t}|^{-1/2}\ll W^{-5/8}$ by Remark \ref{rem:eta}.) We have the following estimate for any $C>0$ large but independent of $W$:
\begin{align*}
\mathbf{1}(\mathcal{E})\max_{t\in[0,1]}\max_{a,b}\frac{|G_{t}(z)_{ab}-m(z)\delta_{ab}|}{(S^{1/2}T_{t}(z)S^{1/2})_{ab}+S^{1/2}_{ab}+W^{-C}}\prec1.
\end{align*}
\end{corollary}
\subsection{Proof of Theorems \ref{theorem:stop}, \ref{theorem:qd}, and \ref{theorem:locallaw}}
\begin{proof}[Proof of Theorem \ref{theorem:stop}]
By union bound, we have 
\begin{align*}
\mathbb{P}[\tau_{\mathrm{stop},1}\neq1]&\leq\mathbb{P}[\{\tau_{\mathrm{stop,2}}\neq1\}\cap\{\tau_{\mathrm{stop},1}=1\}]+\mathbb{P}[\tau_{\mathrm{stop}}=\tau_{\mathrm{stop},1}<1].
\end{align*}
Note that $\mathcal{E}_{t}(z)$ is continuous in $t$ with probability $1$. Thus, on the event in the second probability on the RHS above, we know $|\mathcal{E}_{\tau_{\mathrm{stop}}}(z)_{ab}|\geq W^{\delta_{\mathrm{stop}}}W^{-3/4}|\mathrm{Im}w_{t}|^{-1}W^{-1}|\mathrm{Im}w_{t}|^{-1/2}$ for some $a,b$ indices. On the other hand, as we noted prior to Lemma \ref{lemma:sstop}, we also know that $\mathcal{E}_{\tau_{\mathrm{stop}}}(z)=\mathcal{E}^{\mathrm{stop}}_{\tau_{\mathrm{stop}}}(z)$. Thus,
\begin{align*}
\mathbb{P}[\tau_{\mathrm{stop}}=\tau_{\mathrm{stop},1}<1]&\leq\mathbb{P}\left[\max_{t\in[0,1]}\max_{a,b}\frac{|\mathcal{E}^{M,\mathrm{stop}}_{t}(z)_{ab}|}{W^{-\frac34}|\mathrm{Im}w_{t}|^{-1}\cdot  W^{-1}|\mathrm{Im}w_{t}|^{-\frac12}}\geq W^{\delta_{\mathrm{stop}}}\right].
\end{align*}
By Lemmas \ref{lemma:sstop} and \ref{lemma:mstop} and Proposition \ref{prop:dstop}, we know that the RHS of the previous display is $\leq C_{D}N^{-D}$ for any $D>0$. Combining this with the previous two displays, it now suffices to show $\mathbb{P}[\{\tau_{\mathrm{stop,2}}\neq1\}\cap\{\tau_{\mathrm{stop},1}=1\}]\to0$ as $N\to\infty$, i.e. that 
\begin{align*}
\max_{s\in[0,1]}\max_{a,b}\frac{|G_{s}(z)_{ab}-\delta_{ab}m(z)|^{2}}{(S^{1/2}T_{s}(z)S^{1/2})_{ab}+S^{1/2}_{ab}+W^{-D}}\prec1,
\end{align*}
assuming that 
\begin{align*}
\max_{t\in[0,1]}\max_{a,b}\frac{|\mathcal{E}_{t}(z)_{ab}|}{W^{\delta_{\mathrm{stop}}}W^{-\frac34}|\mathrm{Im}w_{t}|^{-1}\cdot W^{-1}|\mathrm{Im}w_{t}|^{-\frac12}}\lesssim1.
\end{align*}
The a priori estimate above implies $T_{t}(z)_{ab}\lesssim W^{-1}|\mathrm{Im}w_{t}|^{-1/2}$ since $T_{t}(z)=\Theta_{t}+\mathcal{E}_{t}(z)$, and $|(\Theta_{t})_{ab}|\lesssim W^{-1}|\mathrm{Im}w_{t}|^{-1/2}$ (see Lemma \ref{lemma:thetaestimates}). Since $S^{1/2}$ is an averaging operator, we also get that $(S^{1/2}T_{t}(z)S^{1/2})_{ab}\lesssim W^{-1}|\mathrm{Im}w_{t}|^{-1/2}$ for all $a,b$ and for all $t\in[0,1]$. The desired estimate now follows by Corollary \ref{corollary:gfcont}.
\end{proof}
\begin{proof}[Proof of Theorems \ref{theorem:qd} and \ref{theorem:locallaw}]
By Theorem \ref{theorem:stop} and definition of $\tau_{\mathrm{stop}}$ (see \eqref{eq:taubulk}, in which $\delta_{\mathrm{stop}}>0$ is any fixed, small parameter), we know that
\begin{align*}
\max_{a,b}|\mathcal{E}_{1}(z)_{ab}|&\prec W^{-\frac74}|\mathrm{Im}w_{1}|^{-\frac32}\\
\max_{a,b}|G_{1}(z)_{ab}-m(z)\delta_{ab}|^{2}&\prec \max_{a,b}(S^{\frac12}T_{1}(z)S^{\frac12}))_{ab}+\max_{a,b}S^{\frac12}_{ab}.
\end{align*}
Recall $\mathrm{Im}w_{1}=\eta$ and $T_{1}(z)=\Theta_{1}+\mathcal{E}_{1}(z)$; since $T_{1}=T$ and $\Theta_{1}=\Theta$ (see the display before Theorem \ref{theorem:qd}), Theorem \ref{theorem:qd} follows. Using this and Lemma \ref{lemma:thetaestimates} (to bound entries of $\Theta$), we deduce $\max_{a,b}T_{1}(z)_{ab}\prec W^{-1}|\mathrm{Im}w_{1}|^{-1/2}=W^{-1}\eta^{-1/2}$. Since $S^{1/2}$ is an averaging operator, we also get $\max_{a,b}(S^{1/2}T_{1}(z)S^{1/2})_{ab}\prec W^{-1}\eta^{-1/2}$. Since $S^{1/2}_{ab}\lesssim W^{-1}$, Theorem \ref{theorem:locallaw} now follows.
\end{proof}
%
%
%
\section{Bounds on the drift term \texorpdfstring{$\mathcal{E}^{D,\mathrm{stop}}_{t}(z)$}{E\^{} \{D,stop\}\_t(z)}}\label{section:drift}

In this section we prove Proposition \ref{prop:dstop}. The analysis of the drift term $\mathcal{E}^{D,\mathrm{stop}}_{t}(z)$ requires requires expanding the integrands using Gaussian integration by parts. We use underline notation for the fluctuation term arising from the integration by parts.

\begin{defn}
Consider a differentiable function $f: \cc^{N^2} \rightarrow \cc$ and a time $s\in[0,1]$. Given an expression $H_{s,\alpha\beta} f(G_s(z))$, for some $\alpha,\beta\in \llbracket1,N\rrbracket:=\{1,\ldots,N\}$, we define the \emph{renormalization} of this expression as
\begin{equation*}
    \underline{H_{s,\alpha\beta} f(G_s(z))} = H_{s,\alpha\beta} f(G_s(z)) - s S_{\alpha\beta}\partial_{H_{s,\beta\alpha}} f(G_s(z)).
\end{equation*}
The renormalization operation extends linearly to the linear combinations of the terms $H_{s,\alpha\beta} f(G_s(z))$.
\end{defn}

In the following lemma we provide the two operations on the resolvent expressions that will be needed.
\begin{lemma}\label{lemma:expansion}
Consider a differentiable function $f: \cc^{N^2} \rightarrow \cc$ and fix a time $s\in[0,1]$. Define a deterministic matrix $B_s = (I - sm(z)^2 S)^{-1}$ and let $m:=m(z)$. Then we have two identities.

\begin{itemize}
    \item{(Loop expansion)} For any $v\in \llbracket 1,N\rrbracket$, we have
    \begin{align*}
        (G_s(z)_{vv}-m)f(G_s(z)) &= sm\sum_{\alpha,\beta=1}^N B_{v\alpha} S_{\alpha\beta}(G_s(z)_{\alpha\alpha}-m)(G_s(z)_{\beta\beta}-m)f(G_s(z)) \\
        &- sm\sum_{\alpha\beta=1}^N B_{v\alpha}S_{\alpha\beta}G_s(z)_{\beta\alpha} \partial_{H_{s,\beta\alpha}} f(G_s(z)) \\
        &- m\sum_{\alpha=1}^N B_{v\alpha}\underline{(H_sG_s(z))_{\alpha\alpha}f(G_s(z))}.
    \end{align*}
    \item{(Regular vertex expansion)} For any $x, y, u\in \llbracket 1,N\rrbracket$
    \begin{align*}
        G_s(z)_{xu}G_s(z)_{uy} f(G_s(z)) &= mB_{uy} G_s(z)_{xy} f(G_s(z)) \\
        &+ sm\sum_{\alpha\beta=1}^N B_{u\alpha} S_{\alpha\beta} G_s(z)_{x\alpha}(G_s(z)_{\beta\beta}-m)G_s(z)_{\alpha y} f(G_s(z))\\
        &+ sm\sum_{\alpha\beta=1}^N B_{u\alpha}S_{\alpha\beta} G_s(z)_{x\beta}(G_s(z)_{\alpha\alpha}-m)G_s(z)_{\beta y} f(G_s(z)) \\
        &- sm\sum_{\alpha\beta=1}^N B_{u\alpha}S_{\alpha\beta} G_s(z)_{x\alpha} G_s(z)_{\beta y} \partial_{H_{s,\beta\alpha}} f(G_s(z)) \\
        &- m \sum_{\alpha=1}^N B_{u\alpha}\underline{G_s(z)_{x\alpha}(H_sG_s(z))_{\alpha y} f(G_s(z))}.
    \end{align*}
\end{itemize}
\end{lemma}

\begin{proof}
The loop expansion was obtained in Lemma 3.5 of \cite{YYY21}, and the regular vertex expansion -- in Lemma 3.14 of \cite{YYY21} in case $s=1$. Using those results for a band matrix with variance profile $sS$, extends the expansions to any fixed $s\in[0,1]$. Note also that, while \cite{YYY21} considers the high-dimensional band matrices, the results of lemmas 3.5 and 3.14 are identities and their proofs are independent of dimension.
\end{proof}

For a clearer presentation, we now introduce diagrammatic notation.
\begin{defn}
Given a standard oriented graph, we will assign it the following structures.
\begin{itemize}
\item Vertices in the graph will be given a label (e.g. $\alpha\in\{1,\ldots,N\}$).  
\item For a vertex $\alpha$, the symbol {\color{blue}$\Delta$} denotes a factor of $G_{s}(z)_{\alpha\alpha}-m(z)$.
\item A solid edge from vertex $\alpha$ to vertex $\beta$ represents a factor of either $G_{s}(z)_{\alpha\beta}$ or $\overline{G}_{s}(z)_{\alpha\beta}$. The former will be given to blue edges, and the latter will be assigned to red edges.
\item A waved edge from vertex $\alpha$ to vertex $\beta$, if given a neutral black color, indicates a factor of $S_{\alpha\beta}$. A waved edge that is blue corresponds to a factor $(B_{s})_{\alpha\beta}$ from Lemma \ref{lemma:expansion}.
\item An edge that is denoted by a double line corresponds to a factor of $\{\mathrm{Id}+(t-s)\Theta_{t}\}S^{1/2}$.
\item If a vertex denoted by $\alpha$ has degree at least $2$ or a hollow $\Delta$, then we will sum over all $\alpha\in\{1,\ldots,N\}$. 
\end{itemize}
\end{defn}
Let us illustrate this with an example. Recall $\mathcal{E}^{D,\mathrm{stop}}_{t}(z)$ from \eqref{eq:estopformula}, and take its integrand at time $s$. Upon recalling $\Omega_{s}(z)$ therein from Lemma \ref{lemma:gsdes}, we have that
\begin{align}
&[\{\mathrm{Id}+(t-s)\Theta_{t}\}S^{\frac12}\Omega_{s}(z)S^{\frac12}\{\mathrm{Id}+(t-s)\Theta_{t}\}]_{ab}\label{eq:d0}\\
&=\sum_{x,u,v,y}[\{\mathrm{Id}+(t-s)\Theta_{t}\}S^{\frac12}]_{ax}\overline{G}_{s}(z)_{xy}G_{s}(z)_{xu}S_{uv}[G_{s}(z)_{vv}-m(z)]G_{s}(z)_{uy}[\{\mathrm{Id}+(t-s)\Theta_{t}\}S^{\frac12}]_{yb}\nonumber\\
&+\sum_{x,u,v,y}[\{\mathrm{Id}+(t-s)\Theta_{t}\}S^{\frac12}]_{ax}G_{s}(z)_{xy}\overline{G}_{s}(z)_{xu}S_{uv}[\overline{G}_{s}(z)_{vv}-\overline{m(z)}]\overline{G}_{s}(z)_{uy}[\{\mathrm{Id}+(t-s)\Theta_{t}\}S^{\frac12}]_{yb}\nonumber
\end{align}
is equal to the sum of the graph below and its complex conjugate.
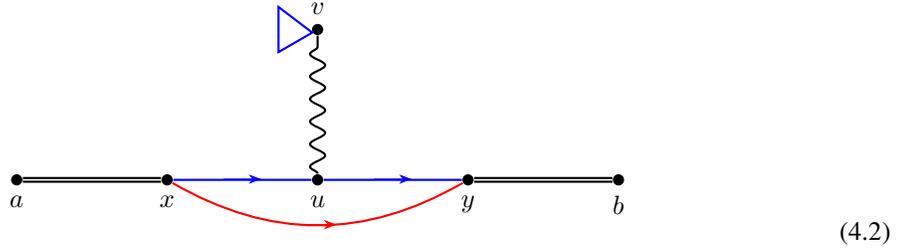
\begin{equation}
\hbox{\begin{tikzpicture}
    \tikzstyle{vertex} = [circle, fill=black, inner sep=1.5pt]

    \node[vertex, label=below:$a$] (a) at (0, 0) {};
    \node[vertex, label=below:$x$] (x) at (2, 0) {};
    \node[vertex, label=below:$u$] (u) at (4, 0) {};
    \node[vertex, label=below:$y$] (y) at (6, 0) {};
    \node[vertex, label=below:$b$] (b) at (8, 0) {};
    \node[vertex, label=above:$v$] (v) at (4,2) {};

    \draw[double, thick] (a) -- (x);
    \draw[double, thick] (y) -- (b);                  
    \draw[thick, blue] (x) -- (u) node[midway, sloped] 
    {\tikz[baseline=-0.5ex]\draw[-{Stealth}] (0,0)--(0.5,0);} ; 
    \draw[thick, blue] (u) -- (y) node[midway, sloped] {\tikz[baseline=-0.5ex]\draw[-{Stealth}] (0,0)--(0.5,0);} ;
    \draw[thick, red, bend right=30] (x) to node[midway, sloped] {\tikz[baseline=-0.5ex]\draw[-{Stealth}] (0,0)--(0.5,0);} (y);
    \draw[thick, decorate, decoration={snake, amplitude=1mm, segment length=3mm}] (u) -- (v);
    \draw[thick, blue] (v) -- ++(210:6mm) -- ++(90:6mm) -- cycle; 
\end{tikzpicture}}\label{eq:d1}
\end{equation}
We now use Lemma \ref{lemma:expansion} to unfold the graph in \eqref{eq:d1} further. We note that for any $\alpha,\beta$ indices, resolvent perturbation gives us $\partial_{H_{\beta\alpha}}G_{xy}=-G_{x\beta}G_{\alpha y}$ and $\partial_{H_{\beta\alpha}}\overline{G}_{xy}=-\overline{G}_{x\alpha}\overline{G}_{\beta y}$, where $G$ is the resolvent of $H$. Thus, by Lemma \ref{lemma:expansion} applied to the vertex $v$ in \eqref{eq:d1} and $f(G_{s}(z))$ equal to the rest of the graph, the above graph admits the following decomposition.
\begin{lemma}\label{lemma:expand1}
Fix any $s\in[0,1]$ and any indices $a,b$. We have
\begin{align*}
\eqref{eq:d1}&=\mathcal{G}_{1,s}(z)_{ab}+sm(z)\mathcal{G}_{2,s}(z)_{ab}+sm(z)\mathcal{G}_{3,s}(z)_{ab}+sm(z)\mathcal{G}_{4,s}(z)_{ab}+\mathcal{F}_{0,s}(z)_{ab}.
\end{align*}
The term $\mathcal{F}_{0,s}(z)$ is the fluctuation below:
\begin{align*}
\mathcal{F}_{0,s}(z)_{ab}&=-m(z)\sum_{x,u,v,y,\alpha}B_{v\alpha}\underline{(H_{s}G_{s}(z))_{\alpha\alpha}f(G_{s}(z))}, \\
f(G)&:=[\{\mathrm{Id}+(t-s)\Theta_{t}\}S^{\frac12}]_{ax}G_{xy}\overline{G}_{xu}S_{uv}\overline{G}_{uy}[\{\mathrm{Id}+(t-s)\Theta_{t}\}S^{\frac12}]_{yb}
\end{align*}
Moreover, $\mathcal{G}_{1,s}(z)_{ab},\ldots,\mathcal{G}_{4,s}(z)_{ab}$ are given by the following graphs:
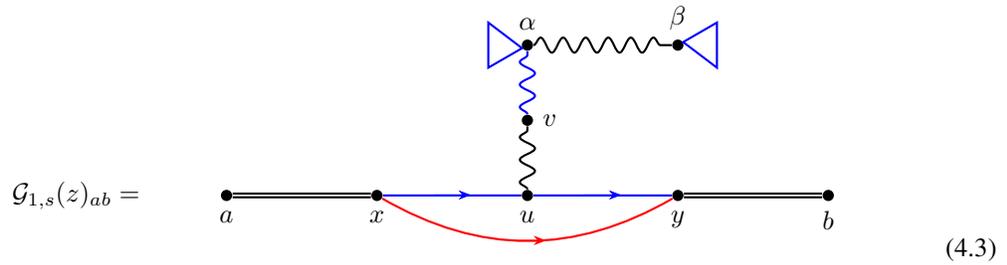
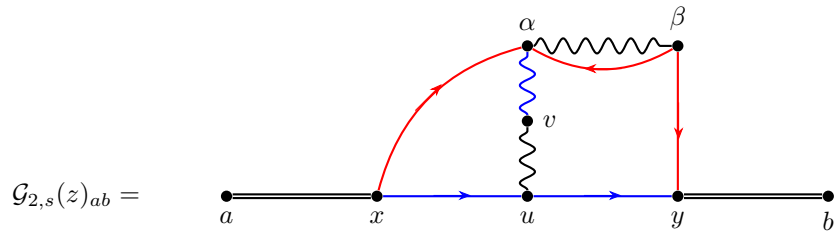
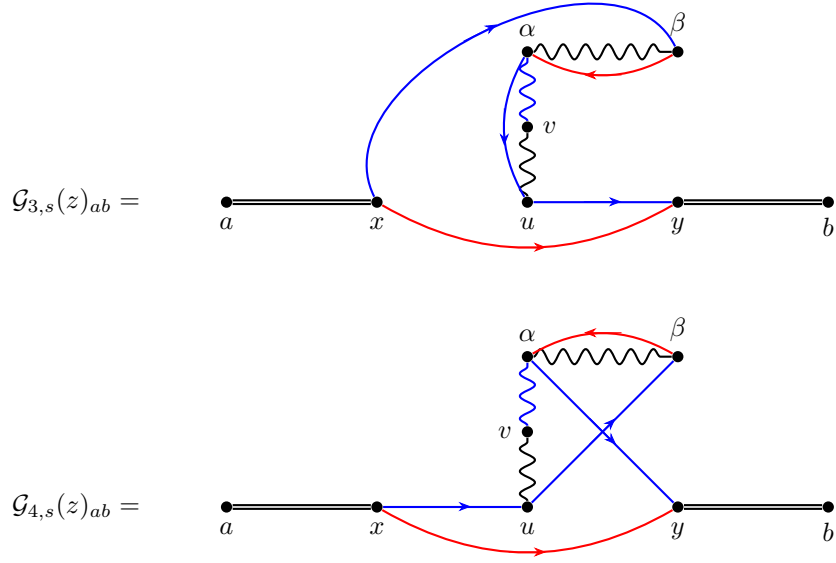
\begin{equation}
\hbox{\begin{tikzpicture}
    \tikzstyle{vertex} = [circle, fill=black, inner sep=1.5pt]
    
    \node at (-2, 0) {$\mathcal{G}_{1,s}(z)_{ab} =$};
    
    \node[vertex, label=below:$a$] (a) at (0, 0) {};
    \node[vertex, label=below:$x$] (x) at (2, 0) {};
    \node[vertex, label=below:$u$] (u) at (4, 0) {};
    \node[vertex, label=below:$y$] (y) at (6, 0) {};
    \node[vertex, label=below:$b$] (b) at (8, 0) {};
    \node[vertex, label=above:$\alpha$] (alpha) at (4,2) {};
    \node[vertex, label=above:$\beta$] (beta) at (6,2) {};
    \node[vertex, label=right:$v$] (v) at (4, 1) {};

    \draw[double, thick] (a) -- (x);
    \draw[double, thick] (y) -- (b);                  
    \draw[thick, blue] (x) -- (u) node[midway, sloped] 
    {\tikz[baseline=-0.5ex]\draw[-{Stealth}] (0,0)--(0.5,0);} ; 
    \draw[thick, blue] (u) -- (y) node[midway, sloped] {\tikz[baseline=-0.5ex]\draw[-{Stealth}] (0,0)--(0.5,0);} ;
    \draw[thick, red, bend right=30] (x) to node[midway, sloped] {\tikz[baseline=-0.5ex]\draw[-{Stealth}] (0,0)--(0.5,0);} (y);
    \draw[thick, blue, decorate, decoration={snake, amplitude=1mm, segment length=3mm}] (v) -- (alpha);
    \draw[thick, blue] (alpha) -- ++(210:6mm) -- ++(90:6mm) -- cycle; 
    \draw[thick, decorate, decoration={snake, amplitude=1mm, segment length=3mm}] (alpha) -- (beta);
    \draw[thick, decorate, decoration={snake, amplitude=1mm, segment length=3mm}] (u) -- (v);
    \draw[thick, blue] (beta) -- ++(30:6mm) -- ++(270:6mm) -- cycle; 
\end{tikzpicture}}\label{eq:d2}
\end{equation}

\begin{equation}
\hbox{\begin{tikzpicture}
    \tikzstyle{vertex} = [circle, fill=black, inner sep=1.5pt]
    
    \node at (-2, 0) {$\mathcal{G}_{2,s}(z)_{ab} =$};
    
    \node[vertex, label=below:$a$] (a) at (0, 0) {};
    \node[vertex, label=below:$x$] (x) at (2, 0) {};
    \node[vertex, label=below:$u$] (u) at (4, 0) {};
    \node[vertex, label=below:$y$] (y) at (6, 0) {};
    \node[vertex, label=below:$b$] (b) at (8, 0) {};
    \node[vertex, label=above:$\alpha$] (alpha) at (4,2) {};
    \node[vertex, label=above:$\beta$] (beta) at (6,2) {};
    \node[vertex, label=right:$v$] (v) at (4, 1) {};

    \draw[double, thick] (a) -- (x);
    \draw[double, thick] (y) -- (b);                  
    \draw[thick, blue] (x) -- (u) node[midway, sloped] 
    {\tikz[baseline=-0.5ex]\draw[-{Stealth}] (0,0)--(0.5,0);} ; 
    \draw[thick, blue] (u) -- (y) node[midway, sloped] {\tikz[baseline=-0.5ex]\draw[-{Stealth}] (0,0)--(0.5,0);} ;
    \draw[thick, blue, decorate, decoration={snake, amplitude=1mm, segment length=3mm}] (v) -- (alpha);
    \draw[thick, decorate, decoration={snake, amplitude=1mm, segment length=3mm}] (u) -- (v);
    \draw[thick, decorate, decoration={snake, amplitude=1mm, segment length=3mm}] (alpha) -- (beta);
    \draw[thick, red, bend left=30] (x) to node[midway, sloped] {\tikz[baseline=-0.5ex]\draw[-{Stealth}] (0,0)--(0.5,0);} (alpha);
    \draw[thick, red, bend left=30] (beta) to node[midway, sloped] {\tikz[baseline=-0.5ex]\draw[-{Stealth}] (0.5,0)--(0,0);} (alpha);
    \draw[thick, red] (beta) to node[midway, sloped] {\tikz[baseline=-0.5ex]\draw[-{Stealth}] (0,0)--(0.5,0);} (y);
\end{tikzpicture}}\nonumber
\end{equation}

\begin{equation}
\hbox{\begin{tikzpicture}
    \tikzstyle{vertex} = [circle, fill=black, inner sep=1.5pt]
    
    \node at (-2, 0) {$\mathcal{G}_{3,s}(z)_{ab} =$};
    
    \node[vertex, label=below:$a$] (a) at (0, 0) {};
    \node[vertex, label=below:$x$] (x) at (2, 0) {};
    \node[vertex, label=below:$u$] (u) at (4, 0) {};
    \node[vertex, label=below:$y$] (y) at (6, 0) {};
    \node[vertex, label=below:$b$] (b) at (8, 0) {};
    \node[vertex, label=above:$\alpha$] (alpha) at (4,2) {};
    \node[vertex, label=above:$\beta$] (beta) at (6,2) {};
    \node[vertex, label=right:$v$] (v) at (4, 1) {};

    \draw[double, thick] (a) -- (x);
    \draw[double, thick] (y) -- (b);                  
    \draw[thick, blue] (u) -- (y) node[midway, sloped] {\tikz[baseline=-0.5ex]\draw[-{Stealth}] (0,0)--(0.5,0);} ;
    \draw[thick, red, bend right=30] (x) to node[midway, sloped] {\tikz[baseline=-0.5ex]\draw[-{Stealth}] (0,0)--(0.5,0);} (y);
    \draw[thick, blue, decorate, decoration={snake, amplitude=1mm, segment length=3mm}] (v) -- (alpha);
    \draw[thick, decorate, decoration={snake, amplitude=1mm, segment length=3mm}] (u) -- (v);
    \draw[thick, decorate, decoration={snake, amplitude=1mm, segment length=3mm}] (alpha) -- (beta);
    \draw[thick, blue, bend left=90] (x) to node[midway, sloped] {\tikz[baseline=-0.5ex]\draw[-{Stealth}] (0,0)--(0.5,0);} (beta);
    \draw[thick, blue, bend right=30] (alpha) to node[midway, sloped] {\tikz[baseline=-0.5ex]\draw[-{Stealth}] (0,0)--(0.5,0);} (u);
    \draw[thick, red, bend left=30] (beta) to node[midway, sloped] {\tikz[baseline=-0.5ex]\draw[-{Stealth}] (0.5,0)--(0,0);} (alpha);
\end{tikzpicture}}\nonumber
\end{equation}

\begin{equation}
\hbox{\begin{tikzpicture}
    \tikzstyle{vertex} = [circle, fill=black, inner sep=1.5pt]
    
    \node at (-2, 0) {$\mathcal{G}_{4,s}(z)_{ab} =$};
    
    \node[vertex, label=below:$a$] (a) at (0, 0) {};
    \node[vertex, label=below:$x$] (x) at (2, 0) {};
    \node[vertex, label=below:$u$] (u) at (4, 0) {};
    \node[vertex, label=below:$y$] (y) at (6, 0) {};
    \node[vertex, label=below:$b$] (b) at (8, 0) {};
    \node[vertex, label=above:$\alpha$] (alpha) at (4,2) {};
    \node[vertex, label=above:$\beta$] (beta) at (6,2) {};
    \node[vertex, label=left:$v$] (v) at (4, 1) {};

    \draw[double, thick] (a) -- (x);
    \draw[double, thick] (y) -- (b);      
    \draw[thick, blue] (x) -- (u) node[midway, sloped] {\tikz[baseline=-0.5ex]\draw[-{Stealth}] (0,0)--(0.5,0);} ;            
    \draw[thick, red, bend right=30] (x) to node[midway, sloped] {\tikz[baseline=-0.5ex]\draw[-{Stealth}] (0,0)--(0.5,0);} (y);
    \draw[thick, blue, decorate, decoration={snake, amplitude=1mm, segment length=3mm}] (v) -- (alpha);
    \draw[thick, decorate, decoration={snake, amplitude=1mm, segment length=3mm}] (u) -- (v);
    \draw[thick, decorate, decoration={snake, amplitude=1mm, segment length=3mm}] (alpha) -- (beta);
    \draw[thick, blue] (u) to node[midway, sloped] {\tikz[baseline=-0.5ex]\draw[-{Stealth}] (0,0)--(0.5,0);} (beta);
    \draw[thick, blue] (alpha) to node[midway, sloped] {\tikz[baseline=-0.5ex]\draw[-{Stealth}] (0,0)--(0.5,0);} (y);
    \draw[thick, red, bend right=30] (beta) to node[midway, sloped] {\tikz[baseline=-0.5ex]\draw[-{Stealth}] (0.5,0)--(0,0);} (alpha);
\end{tikzpicture}}\nonumber
\end{equation}
\end{lemma}
We will now apply the regular vertex expansion from Lemma \ref{lemma:expansion} to each of the graphs $\mathcal{G}_{j,s}(z)$ for $j=1,2,3,4$ from Lemma \ref{lemma:expand1}. In particular, each said graph has a vertex $u$ with one incoming and one outgoing blue edge; this is where we expand. The expression we obtain is given in the following lemma. (In what follows, we describe the relevant graphs with words. We illustrate them as pictures when we analyze them in the proofs of Lemmas \ref{lemma:g1-g4;bulk1}, \ref{lemma:g1-g4;bulk2}, and \ref{lemma:g1-g4;bulk3}.)
\begin{lemma}\label{lemma:level2expand}
Fix $s\in[0,1]$, indices $a,b\in\llbracket 1,N\rrbracket$ and $i\in\llbracket1,4\rrbracket$. We have
\begin{align*}
    \mathcal{G}_{i,s}(z)_{ab} = m(z)\mathcal{G}_{i0,s}(z)_{ab} + sm(z)\sum_{j=1}^2 \mathcal{G}_{ij,s}(z)_{ab} - sm(z)\mathcal{G}_{i3,s}(z)_{ab} - m(z)\mathcal{F}_{i,s}(z)_{ab},
\end{align*}
where the fluctuation term $\mathcal{F}_{i,s}(z)_{ab}$ is obtained from $\mathcal{G}_{i,s}(z)_{ab}$ by replacing the outgoing blue $G_s(z)$-edge from vertex $u$ by $H_s G_s(z)$ and renormalizing, i.e. applying the underline operation. The main term $\mathcal{G}_{i0,s}(z)_{ab}$ is obtained from $\mathcal{G}_{i,s}(z)_{ab}$ by removing the two blue edges $G_s(z)_{\xi u}$ and $G_s(z)_{u\zeta}$ going through $u$ and replacing them with $G_{\xi \zeta}$ and $B_{u\zeta}$. The other three main terms introduce two new vertices $\gamma,\delta$ where $\gamma$ is connected to $u$ by a blue waved edge, $\gamma$ and $\delta$ are connected by a black waved edge. The the two blue edges $G_s(z)_{\xi u}$ and $G_s(z)_{u\zeta}$ going through $u$ are removed. The new $G$-edges are added as follows.
\begin{itemize}
    \item In the term $\mathcal{G}_{i1,s}(z)_{ab}$, add $G_s(z)_{\xi \gamma}$, $G_s(z)_{\gamma\zeta}$ and a blue hollow $\Delta$ at the vertex $\delta$.

    \item In the term $\mathcal{G}_{i2,s}(z)_{ab}$, add $G_s(z)_{\xi \delta}$, $G_s(z)_{\delta\zeta}$ and a blue hollow $\Delta$ at the vertex $\gamma$.

    \item In the term $\mathcal{G}_{i3,s}(z)_{ab}$, add $G_s(z)_{\xi \gamma}$, $G_s(z)_{\delta\zeta}$. Additionally, in this term we apply a derivative $\partial_{H_{s,\gamma,\delta}}$ to all other blue or red edges. This results in a sum of three more graphs, where each remaining $G$-edge is split in two and reattached to vertices $\gamma$ and $\delta$. 
\end{itemize} 
\end{lemma}

We now give two lemmas. The first controls the graphs $\mathcal{G}_{ij,s}(z)_{ab}$ for $j=0,1,2$, and the second controls $\mathcal{G}_{i3,s}(z)_{ab}$. 
\begin{lemma}\label{lemma:g1-g4;bulk1}
Recall $\tau_{\mathrm{stop}}$ from \eqref{eq:taubulk}, and recall that $N\ll W^{11/8}$. We have the deterministic estimate
\begin{align}
\max_{a,b}\max_{i=1,\ldots,4}\max_{j=0,1,2}\left|\int_{0}^{\tau_{\mathrm{stop}}\wedge t}\mathcal{G}_{ij,s}(z)_{ab}ds\right|&\lesssim W^{-\frac34}|\mathrm{Im}w_{t}|^{-1}\cdot W^{-1}|\mathrm{Im}w_{t}|^{-\frac12}. \label{eq:g1-g4;bulk1main}
\end{align}
\end{lemma}
\begin{proof}
We start with a few preliminaries. Recall that $T_{s}(z)=\Theta_{s}+\mathcal{E}_{s}(z)$, and recall $\tau_{\mathrm{stop}}$ from \eqref{eq:taubulk}. With this, for any $s\in[0,\tau_{\mathrm{stop}}]$, we have the following, where the last line below is a consequence of the fact that $S^{1/2}$ is an averaging operator:
\begin{align*}
&\sup_{x}\sum_{y}|G_{s}(z)_{xy}|\lesssim W^{\frac{\delta_{\mathrm{stop}}}{20}}\sup_{x}\sum_{y}|(S^{\frac12}T_{s}(z)S^{\frac12})_{xy}|^{\frac12}+W^{\frac{\delta_{\mathrm{stop}}}{20}}\sup_{x}\sum_{y}|S^{\frac12}_{xy}|^{\frac12}\\
&\lesssim W^{\frac{\delta_{\mathrm{stop}}}{20}}\sup_{x}\sum_{y}\left\{\sum_{\gamma,\delta}S^{\frac12}_{x\gamma}(\Theta_{s})_{\gamma\delta}S^{\frac12}_{\delta y}\right\}^{\frac12}+W^{\frac{\delta_{\mathrm{stop}}}{20}}\sup_{x}\sum_{y}\left\{\sum_{\gamma,\delta}S^{\frac12}_{x\gamma}|\mathcal{E}_{s}(z)_{\gamma\delta}|S^{\frac12}_{\delta y}\right\}^{\frac12}+W^{\frac{\delta_{\mathrm{stop}}}{20}}W^{\frac12}\\
&\lesssim W^{\frac{\delta_{\mathrm{stop}}}{20}}\sup_{x}\sum_{y}\left\{\sum_{\gamma,\delta}S^{\frac12}_{x\gamma}(\Theta_{s})_{\gamma\delta}S^{\frac12}_{\delta y}\right\}^{\frac12}+W^{\frac{\delta_{\mathrm{stop}}}{20}}N\max_{\gamma,\delta}|\mathcal{E}_{s}(z)_{\gamma\delta}|^{\frac12}+W^{\frac{\delta_{\mathrm{stop}}}{20}}W^{\frac12}
\end{align*}
We now use {$(S^{1/2}\Theta_{s}S^{1/2})_{xy}\lesssim W^{-1+\delta}|\mathrm{Im}w_{s}|^{-1/2}\exp\{-{K}|\mathrm{Im}w_{s}|^{1/2}W^{-1-\varepsilon}|x-y|_{N}\}+\exp\{-W^{\varepsilon}\}$}, where we recall that $|\cdot|_{N}$ is the periodic distance on $\Z/N\Z$. This holds for any $\varepsilon,\delta,K>0$ fixed (see Lemma \ref{lemma:sthetasrange}). We also use that for any $s\leq\tau_{\mathrm{stop}}$, we have $|\mathcal{E}_{s}(z)_{\gamma\delta}|\lesssim W^{\delta_{\mathrm{stop}}/10}W^{-3/4}|\mathrm{Im}w_{s}|^{-1}W^{-1}|\mathrm{Im}w_{s}|^{-1/2}$. Plugging these into the above implies the following (for a possibly different but still small $\varepsilon>0$):
\begin{align}
\sup_{x}\sum_{y}|G_{s}(z)_{xy}|&\lesssim W^{\frac{\delta_{\mathrm{stop}}}{20}}W^{\frac12+\varepsilon}|\mathrm{Im}w_{s}|^{-\frac34}+W^{\frac{\delta_{\mathrm{stop}}}{10}}NW^{-\frac78}|\mathrm{Im}w_{s}|^{-\frac34}\nonumber\\
&\lesssim W^{\frac{\delta_{\mathrm{stop}}}{20}}W^{\frac12+\varepsilon}|\mathrm{Im}w_{s}|^{-\frac34},\label{eq:g1-g4;bulk1}
\end{align}
where the last bound follows from $N\ll W^{11/8}$. Next, we observe that for any $s\leq\tau_{\mathrm{stop}}$ (see \eqref{eq:taubulk}), we have
\begin{align}
\max_{x,y}|G_{s}(z)_{xy}-\delta_{xy}m(z)|&\lesssim \left\{W^{\frac{\delta_{\mathrm{stop}}}{10}}\max_{x,y}(S^{\frac12}T_{s}(z)S^{\frac12})_{xy}+W^{\frac{\delta_{\mathrm{stop}}}{10}}\max_{x,y}S^{\frac12}_{xy}+W^{-D}\right\}^{\frac12}\nonumber\\
&\lesssim\left\{W^{\frac{\delta_{\mathrm{stop}}}{10}}\max_{x,y}(S^{\frac12}\Theta_{s}S^{\frac12})_{xy}+W^{\frac{\delta_{\mathrm{stop}}}{10}}\max_{x,y}|\mathcal{E}_{s}(z)_{xy}|+W^{\frac{\delta_{\mathrm{stop}}}{10}}W^{-1}\right\}^{\frac12}\nonumber\\
&\lesssim W^{\frac{\delta_{\mathrm{stop}}}{20}}W^{-\frac12}|\mathrm{Im}w_{s}|^{-\frac14},\label{eq:bulkgbound}
\end{align}
where the last line follows by Lemma \ref{lemma:thetaestimates} and a priori bounds for $\mathcal{E}_{s}(z)$ before $\tau_{\mathrm{stop}}$. We now prove \eqref{eq:g1-g4;bulk1main} for $j=0$. The graphs at hand are
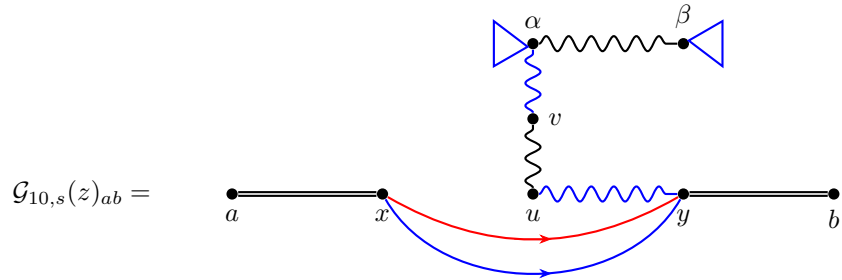
\begin{equation}
\hbox{\begin{tikzpicture}
    \tikzstyle{vertex} = [circle, fill=black, inner sep=1.5pt]
    
    \node at (-2, 0) {$\mathcal{G}_{10,s}(z)_{ab} =$};
    
    \node[vertex, label=below:$a$] (a) at (0, 0) {};
    \node[vertex, label=below:$x$] (x) at (2, 0) {};
    \node[vertex, label=below:$u$] (u) at (4, 0) {};
    \node[vertex, label=below:$y$] (y) at (6, 0) {};
    \node[vertex, label=below:$b$] (b) at (8, 0) {};
    \node[vertex, label=above:$\alpha$] (alpha) at (4,2) {};
    \node[vertex, label=above:$\beta$] (beta) at (6,2) {};
    \node[vertex, label=right:$v$] (v) at (4, 1) {};

    \draw[double, thick] (a) -- (x);
    \draw[double, thick] (y) -- (b);                  
    \draw[thick, red, bend right=30] (x) to node[midway, sloped] {\tikz[baseline=-0.5ex]\draw[-{Stealth}] (0,0)--(0.5,0);} (y);
    \draw[thick, blue, bend right=60] (x) to node[midway, sloped] {\tikz[baseline=-0.5ex]\draw[-{Stealth}] (0,0)--(0.5,0);} (y);
    \draw[thick, blue, decorate, decoration={snake, amplitude=1mm, segment length=3mm}] (u) -- (y);
    \draw[thick, blue, decorate, decoration={snake, amplitude=1mm, segment length=3mm}] (v) -- (alpha);
    \draw[thick, decorate, decoration={snake, amplitude=1mm, segment length=3mm}] (u) -- (v);
    \draw[thick, blue] (alpha) -- ++(210:6mm) -- ++(90:6mm) -- cycle; 
    \draw[thick, decorate, decoration={snake, amplitude=1mm, segment length=3mm}] (alpha) -- (beta);
    \draw[thick, blue] (beta) -- ++(30:6mm) -- ++(270:6mm) -- cycle; 
\end{tikzpicture}}\nonumber
\end{equation}

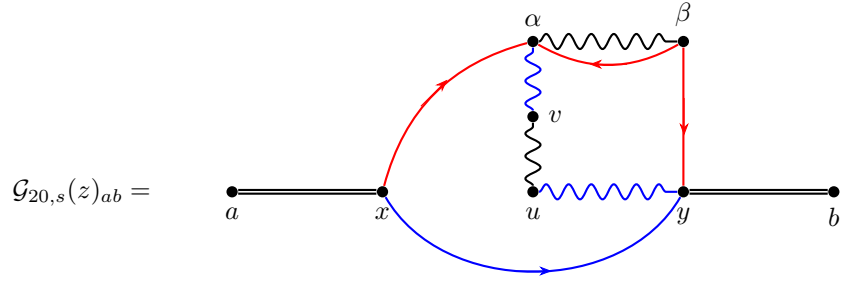
\begin{equation}
\hbox{\begin{tikzpicture}
    \tikzstyle{vertex} = [circle, fill=black, inner sep=1.5pt]
    
    \node at (-2, 0) {$\mathcal{G}_{20,s}(z)_{ab} =$};
    
    \node[vertex, label=below:$a$] (a) at (0, 0) {};
    \node[vertex, label=below:$x$] (x) at (2, 0) {};
    \node[vertex, label=below:$u$] (u) at (4, 0) {};
    \node[vertex, label=below:$y$] (y) at (6, 0) {};
    \node[vertex, label=below:$b$] (b) at (8, 0) {};
    \node[vertex, label=above:$\alpha$] (alpha) at (4,2) {};
    \node[vertex, label=above:$\beta$] (beta) at (6,2) {};
    \node[vertex, label=right:$v$] (v) at (4, 1) {};

    \draw[double, thick] (a) -- (x);
    \draw[double, thick] (y) -- (b);                  
    \draw[thick, blue, decorate, decoration={snake, amplitude=1mm, segment length=3mm}] (u) -- (y);
    \draw[thick, blue, decorate, decoration={snake, amplitude=1mm, segment length=3mm}] (v) -- (alpha);
    \draw[thick, decorate, decoration={snake, amplitude=1mm, segment length=3mm}] (u) -- (v);
    \draw[thick, decorate, decoration={snake, amplitude=1mm, segment length=3mm}] (alpha) -- (beta);
    \draw[thick, red, bend left=30] (x) to node[midway, sloped] {\tikz[baseline=-0.5ex]\draw[-{Stealth}] (0,0)--(0.5,0);} (alpha);
    \draw[thick, blue, bend right=60] (x) to node[midway, sloped] {\tikz[baseline=-0.5ex]\draw[-{Stealth}] (0,0)--(0.5,0);} (y);
    \draw[thick, red, bend left=30] (beta) to node[midway, sloped] {\tikz[baseline=-0.5ex]\draw[-{Stealth}] (0.5,0)--(0,0);} (alpha);
    \draw[thick, red] (beta) to node[midway, sloped] {\tikz[baseline=-0.5ex]\draw[-{Stealth}] (0,0)--(0.5,0);} (y);
\end{tikzpicture}}\nonumber
\end{equation}

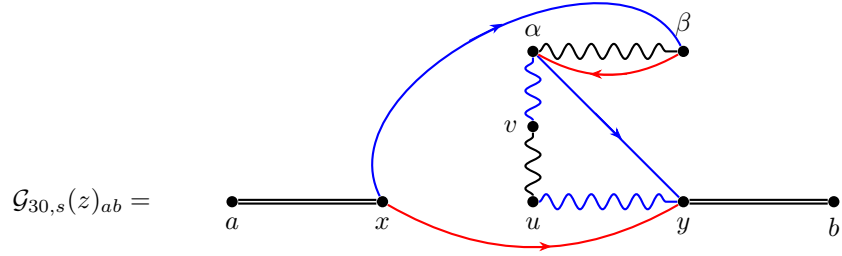
\begin{equation}
\hbox{\begin{tikzpicture}
    \tikzstyle{vertex} = [circle, fill=black, inner sep=1.5pt]
    
    \node at (-2, 0) {$\mathcal{G}_{30,s}(z)_{ab} =$};
    
    \node[vertex, label=below:$a$] (a) at (0, 0) {};
    \node[vertex, label=below:$x$] (x) at (2, 0) {};
    \node[vertex, label=below:$u$] (u) at (4, 0) {};
    \node[vertex, label=below:$y$] (y) at (6, 0) {};
    \node[vertex, label=below:$b$] (b) at (8, 0) {};
    \node[vertex, label=above:$\alpha$] (alpha) at (4,2) {};
    \node[vertex, label=above:$\beta$] (beta) at (6,2) {};
    \node[vertex, label=left:$v$] (v) at (4, 1) {};

    \draw[double, thick] (a) -- (x);
    \draw[double, thick] (y) -- (b);                  
    \draw[thick, blue] (alpha) -- (y) node[midway, sloped] {\tikz[baseline=-0.5ex]\draw[-{Stealth}] (0,0)--(0.5,0);} ;
    \draw[thick, red, bend right=30] (x) to node[midway, sloped] {\tikz[baseline=-0.5ex]\draw[-{Stealth}] (0,0)--(0.5,0);} (y);
    \draw[thick, blue, decorate, decoration={snake, amplitude=1mm, segment length=3mm}] (v) -- (alpha);
    \draw[thick, decorate, decoration={snake, amplitude=1mm, segment length=3mm}] (u) -- (v);
    \draw[thick, blue, decorate, decoration={snake, amplitude=1mm, segment length=3mm}] (u) -- (y);
    \draw[thick, decorate, decoration={snake, amplitude=1mm, segment length=3mm}] (alpha) -- (beta);
    \draw[thick, blue, bend left=90] (x) to node[midway, sloped] {\tikz[baseline=-0.5ex]\draw[-{Stealth}] (0,0)--(0.5,0);} (beta);
    \draw[thick, red, bend left=30] (beta) to node[midway, sloped] {\tikz[baseline=-0.5ex]\draw[-{Stealth}] (0.5,0)--(0,0);} (alpha);
\end{tikzpicture}}\nonumber
\end{equation}

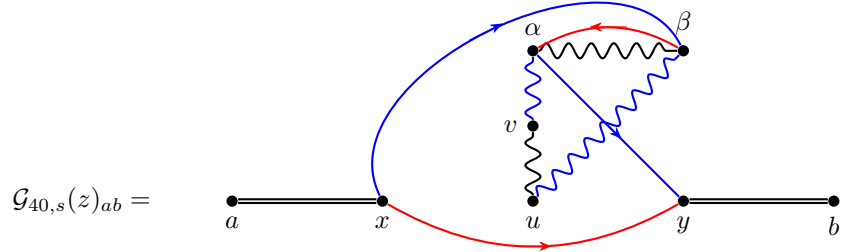
\begin{equation}
\hbox{\begin{tikzpicture}
    \tikzstyle{vertex} = [circle, fill=black, inner sep=1.5pt]
    
    \node at (-2, 0) {$\mathcal{G}_{40,s}(z)_{ab} =$};
    
    \node[vertex, label=below:$a$] (a) at (0, 0) {};
    \node[vertex, label=below:$x$] (x) at (2, 0) {};
    \node[vertex, label=below:$u$] (u) at (4, 0) {};
    \node[vertex, label=below:$y$] (y) at (6, 0) {};
    \node[vertex, label=below:$b$] (b) at (8, 0) {};
    \node[vertex, label=above:$\alpha$] (alpha) at (4,2) {};
    \node[vertex, label=above:$\beta$] (beta) at (6,2) {};
    \node[vertex, label=left:$v$] (v) at (4, 1) {};

    \draw[double, thick] (a) -- (x);
    \draw[double, thick] (y) -- (b);      
    \draw[thick, blue, bend left=90] (x) to node[midway, sloped] {\tikz[baseline=-0.5ex]\draw[-{Stealth}] (0,0)--(0.5,0);} (beta);
    \draw[thick, red, bend right=30] (x) to node[midway, sloped] {\tikz[baseline=-0.5ex]\draw[-{Stealth}] (0,0)--(0.5,0);} (y);
    \draw[thick, blue, decorate, decoration={snake, amplitude=1mm, segment length=3mm}] (v) -- (alpha);
    \draw[thick, decorate, decoration={snake, amplitude=1mm, segment length=3mm}] (u) -- (v);
    \draw[thick, decorate, decoration={snake, amplitude=1mm, segment length=3mm}] (alpha) -- (beta);
    \draw[thick, blue, decorate, decoration={snake, amplitude=1mm, segment length=3mm}] (u) -- (beta);
    \draw[thick, blue] (alpha) to node[midway, sloped] {\tikz[baseline=-0.5ex]\draw[-{Stealth}] (0,0)--(0.5,0);} (y);
    \draw[thick, red, bend right=30] (beta) to node[midway, sloped] {\tikz[baseline=-0.5ex]\draw[-{Stealth}] (0.5,0)--(0,0);} (alpha);
\end{tikzpicture}}\nonumber
\end{equation}
In what follows, we will always bound solid lines by \eqref{eq:bulkgbound}; this bound fails when the solid line has matching indices, but this implies that one less solid line has to be summed out, which saves us a factor of \eqref{eq:g1-g4;bulk1}. So, in this exceptional case of matching indices, we lose a saving factor of $W^{-1}|\mathrm{Im}w_{s}|^{-1/2}$, but we gain the better factor $W^{-1/2+\varepsilon}|\mathrm{Im}w_{s}|^{3/4}$; since $\varepsilon>0$ is small, our bounds remain in tact.

Now, for $\mathcal{G}_{10,s}(z)_{ab}$, we bound each {\color{blue}$\Delta$} and the red line by $\mathrm{O}(W^{\delta_{\mathrm{stop}}/20}W^{-1/2}|\mathrm{Im}w_{s}|^{-1/4})$, and then we sum out $\beta,\alpha,v,u$ in that order, with each sum contributing only $\mathrm{O}(1)$ (see Lemma \ref{lemma:Bestimates}). For $\mathcal{G}_{20,s}(z)_{ab}$, we bound all three red lines by $\mathrm{O}(W^{\delta_{\mathrm{stop}}/20}W^{-1/2}|\mathrm{Im}w_{s}|^{-1/4})$, and then we sum out $\beta,\alpha,v,u$ in that order. For $\mathcal{G}_{30,s}(z)_{ab}$, we bound the $\beta{\color{red}\to}\alpha$ line and $\alpha{\color{blue}\to}y$ line by $\mathrm{O}(W^{\delta_{\mathrm{stop}}/20}W^{-1/2}|\mathrm{Im}w_{s}|^{-1/4})$. Then, we bound the $\alpha\leftrightsquigarrow\beta$ arrow by $W^{-1}$. Next, we sum out the $\beta$-index via the $x{\color{blue}\to}\beta$ line, which, according to \eqref{eq:g1-g4;bulk1}, gives a factor of $\mathrm{O}(W^{\delta_{\mathrm{stop}}/20}W^{1/2+\varepsilon}|\mathrm{Im}w_{s}|^{-3/4})$. Finally, we sum out the remaining wavy lines. For $\mathcal{G}_{40,s}(z)_{ab}$, we bound the $\alpha{\color{blue}\to}y$ and $\beta{\color{red}\to}\alpha$ lines by $\mathrm{O}(W^{\delta_{\mathrm{stop}}/20}W^{-1/2}|\mathrm{Im}w_{s}|^{-1/4})$ each as well. Then, we bound the $\alpha\leftrightsquigarrow\beta$ line by $\mathrm{O}(W^{-1})$. Next, we sum out the remaining wavy lines, and then we sum out the $x{\color{blue}\to}\beta$ line to get a factor of $\mathrm{O}(W^{\delta_{\mathrm{stop}}/20}W^{1/2+\varepsilon}|\mathrm{Im}w_{s}|^{-3/4})$ by \eqref{eq:g1-g4;bulk1}. Ultimately, we deduce 
\begin{center}
\begin{tikzpicture}
    \tikzstyle{vertex} = [circle, fill=black, inner sep=1.5pt]
    
    \node at (-4.2, 0) {$\max_{i=1,\ldots,4}|\mathcal{G}_{i0,s}(z)_{ab}|\lesssim W^{\frac{3\delta_{\mathrm{edge}}}{20}}W^{-\frac32+\varepsilon}|\mathrm{Im}w_{s}|^{-\frac54}\times$};
    
    \node[vertex, label=below:$a$] (a) at (0, 0) {};
    \node[vertex, label=below:$x$] (x) at (1, 0) {};
    \node[vertex, label=below:$y$] (y) at (2, 0) {};
    \node[vertex, label=below:$b$] (b) at (3, 0) {};

    \draw[double, thick] (a) -- (x);
    \draw[double, thick] (y) -- (b);                  
    \draw[thick, purple] (x) -- (y) node[midway, sloped] 
    {\tikz[baseline=-0.5ex]\draw[-{Stealth}] (0,0)--(0.5,0);} ; 
\end{tikzpicture}
\end{center}
where the purple line means the absolute value $|G_{s}(z)_{xy}|$. Next, we consider two sub-cases.
\begin{itemize}
\item Split the $(y,b)$-double line into $(t-s)(\Theta_{t})_{yb}$ and $\mathrm{Id}_{yb}$, and consider first the term with $(t-s)(\Theta_{t})_{yb}$. To bound the diagram above, we bound {$(t-s)(\Theta_{t})_{yb}\lesssim (t-s)W^{-1}|\mathrm{Im}w_{t}|^{-1/2}$} (see Lemma \ref{lemma:thetaestimates}). Then, we sum over $y$ using \eqref{eq:g1-g4;bulk1} to get a factor of $\mathrm{O}(W^{\delta_{\mathrm{stop}}/20}W^{1/2+\varepsilon}|\mathrm{Im}w_{s}|^{-3/4})$. Finally, we sum over $x$ using \eqref{eq:sstop1} to get a factor of $\mathrm{O}(|\mathrm{Im}w_{t}|^{-1}|\mathrm{Im}w_{s}|)$. Ultimately, we get a bound on the diagram of $\mathrm{O}((t-s)W^{\delta_{\mathrm{stop}}/20}W^{-1/2}|\mathrm{Im}w_{t}|^{-3/2}|\mathrm{Im}w_{s}|^{1/4})$.
\item Take the term with $\mathrm{Id}_{yb}$. In this case, we bound the $x{\color{blue}\to}y$ line by $\mathrm{O}(W^{\delta_{\mathrm{stop}}/20}W^{-1/2}|\mathrm{Im}w_{s}|^{-1/4})$. Now, we only need to sum over $x$ using \eqref{eq:sstop1} to get a factor of $\mathrm{O}(|\mathrm{Im}w_{t}|^{-1}|\mathrm{Im}w_{s}|)$. Thus, we get a bound on the diagram of $\mathrm{O}(W^{\delta_{\mathrm{stop}}/20}W^{-1/2}|\mathrm{Im}w_{t}|^{-1}|\mathrm{Im}w_{s}|^{3/4})$.
\end{itemize}
Thus, the diagram above is bounded above by
\begin{align}
&\lesssim (t-s)W^{\varepsilon}W^{\frac{\delta_{\mathrm{stop}}}{20}}W^{-\frac12}|\mathrm{Im}w_{t}|^{-\frac32}|\mathrm{Im}w_{s}|^{\frac14}+W^{\frac{\delta_{\mathrm{stop}}}{20}}W^{-\frac12}|\mathrm{Im}w_{t}|^{-1}|\mathrm{Im}w_{s}|^{\frac34}\nonumber\\
&\lesssim W^{\varepsilon}W^{\frac{\delta_{\mathrm{stop}}}{20}}W^{-\frac12}|\mathrm{Im}w_{t}|^{-\frac32}|\mathrm{Im}w_{s}|^{\frac54}+W^{\frac{\delta_{\mathrm{stop}}}{20}}W^{-\frac12}|\mathrm{Im}w_{t}|^{-1}|\mathrm{Im}w_{s}|^{\frac34}. \label{eq:g1-g4;bulk0}
\end{align}
Indeed, recall that $(t-s)\lesssim|\mathrm{Im}w_{s}|$. We now plug the previous two points into our $|\mathcal{G}_{i0,s}(z)_{ab}|$ estimate, integrate over $s\in[0,t]$, and use the change of variables $\sigma=|\mathrm{Im}w_{s}|$ to get
\begin{align}
&\max_{i=1,2,3,4}\int_{0}^{\tau_{\mathrm{stop}}\wedge t}|\mathcal{G}_{i0,s}(z)_{ab}|ds\nonumber\\
&\lesssim W^{\frac{\delta_{\mathrm{stop}}}{5}+2\varepsilon}W^{-2}|\mathrm{Im}w_{t}|^{-\frac32}\int_{|\mathrm{Im}w_{t}|}^{|\mathrm{Im}w_{0}|}d\sigma+W^{\frac{\delta_{\mathrm{stop}}}{5}+2\varepsilon}W^{-2}|\mathrm{Im}w_{t}|^{-1}\int_{|\mathrm{Im}w_{t}|}^{|\mathrm{Im}w_{0}|}\sigma^{-\frac12}d\sigma\nonumber\\
&\lesssim W^{\frac{\delta_{\mathrm{stop}}}{5}+2\varepsilon}W^{-2}|\mathrm{Im}w_{t}|^{-\frac32}.\label{eq:g1-g4;bulk2}
\end{align}
Thus, we are left with $j=1,2$ in \eqref{eq:g1-g4;bulk1main}. We give the details for $k=1$; for $i=2$, the exact same argument works. (Indeed, the only difference between $j=1$ and $j=2$ is the location of a waved edge to a vertex that has attached to it one copy of {\color{blue}$\Delta$}. We will always bound the triangle by its size and sum out the waved edge, which contributes a factor of $\mathrm{O}(1)$ regardless of the location of this waved edge.) In this case, the diagrams at hand are
\begin{equation}
\hbox{\begin{tikzpicture}
    \tikzstyle{vertex} = [circle, fill=black, inner sep=1.5pt]
    
    \node at (-2, 0) {$\mathcal{G}_{11,s}(z)_{ab} =$};
    
    \node[vertex, label=below:$a$] (a) at (0, 0) {};
    \node[vertex, label=below:$x$] (x) at (2, 0) {};
    \node[vertex, label=below:$u$] (u) at (5, 1) {};
    \node[vertex, label=below:$y$] (y) at (6, 0) {};
    \node[vertex, label=below:$b$] (b) at (8, 0) {};
    \node[vertex, label=above:$\alpha$] (alpha) at (4,2) {};
    \node[vertex, label=above:$\beta$] (beta) at (6,2) {};
    \node[vertex, label=below:$\gamma$] (gamma) at (4,0) {};
    \node[vertex, label=above:$\delta$] (delta) at (2,2) {};
    \node[vertex, label=right:$v$] (v) at (4.5, 1.5) {};

    \draw[double, thick] (a) -- (x);
    \draw[double, thick] (y) -- (b);                  
    \draw[thick, blue] (x) -- (gamma) node[midway, sloped] 
    {\tikz[baseline=-0.5ex]\draw[-{Stealth}] (0,0)--(0.5,0);} ; 
    \draw[thick, blue] (gamma) -- (y) node[midway, sloped] {\tikz[baseline=-0.5ex]\draw[-{Stealth}] (0,0)--(0.5,0);} ;
    \draw[thick, red, bend right=30] (x) to node[midway, sloped] {\tikz[baseline=-0.5ex]\draw[-{Stealth}] (0,0)--(0.5,0);} (y);
    \draw[thick, blue, decorate, decoration={snake, amplitude=1mm, segment length=3mm}] (v) -- (alpha);
    \draw[thick, decorate, decoration={snake, amplitude=1mm, segment length=3mm}] (u) -- (v);
    \draw[thick, blue, decorate, decoration={snake, amplitude=1mm, segment length=3mm}] (gamma) -- (u);
    \draw[thick, decorate, decoration={snake, amplitude=1mm, segment length=3mm}] (gamma) -- (delta);
    \draw[thick, blue] (alpha) -- ++(210:6mm) -- ++(90:6mm) -- cycle; 
    \draw[thick, decorate, decoration={snake, amplitude=1mm, segment length=3mm}] (alpha) -- (beta);
    \draw[thick, blue] (beta) -- ++(30:6mm) -- ++(270:6mm) -- cycle; 
    \draw[thick, blue] (delta) -- ++(210:6mm) -- ++(90:6mm) -- cycle; 
\end{tikzpicture}}\nonumber
\end{equation}

\begin{equation}
\hbox{\begin{tikzpicture}
    \tikzstyle{vertex} = [circle, fill=black, inner sep=1.5pt]
    
    \node at (-2, 0) {$\mathcal{G}_{21,s}(z)_{ab} =$};
    
    \node[vertex, label=below:$a$] (a) at (0, 0) {};
    \node[vertex, label=below:$x$] (x) at (2, 0) {};
    \node[vertex, label=below:$u$] (u) at (5, 1) {};
    \node[vertex, label=below:$y$] (y) at (6, 0) {};
    \node[vertex, label=below:$b$] (b) at (8, 0) {};
    \node[vertex, label=above:$\alpha$] (alpha) at (4,2) {};
    \node[vertex, label=above:$\beta$] (beta) at (6,2) {};
    \node[vertex, label=below:$\gamma$] (gamma) at (4,0) {};
    \node[vertex, label=above:$\delta$] (delta) at (3,1) {};
    \node[vertex, label=left:$v$] (v) at (4.5, 1.5) {};

    \draw[double, thick] (a) -- (x);
    \draw[double, thick] (y) -- (b);                  
    \draw[thick, blue] (x) -- (gamma) node[midway, sloped] 
    {\tikz[baseline=-0.5ex]\draw[-{Stealth}] (0,0)--(0.5,0);} ; 
    \draw[thick, blue] (gamma) -- (y) node[midway, sloped] {\tikz[baseline=-0.5ex]\draw[-{Stealth}] (0,0)--(0.5,0);} ;
    \draw[thick, blue, decorate, decoration={snake, amplitude=1mm, segment length=3mm}] (v) -- (alpha);
    \draw[thick, decorate, decoration={snake, amplitude=1mm, segment length=3mm}] (u) -- (v);
    \draw[thick, decorate, decoration={snake, amplitude=1mm, segment length=3mm}] (alpha) -- (beta);
    \draw[thick, blue, decorate, decoration={snake, amplitude=1mm, segment length=3mm}] (gamma) -- (u);
    \draw[thick, decorate, decoration={snake, amplitude=1mm, segment length=3mm}] (gamma) -- (delta);
    \draw[thick, red, bend left=60] (x) to node[midway, sloped] {\tikz[baseline=-0.5ex]\draw[-{Stealth}] (0,0)--(0.5,0);} (alpha);
    \draw[thick, red, bend left=30] (beta) to node[midway, sloped] {\tikz[baseline=-0.5ex]\draw[-{Stealth}] (0.5,0)--(0,0);} (alpha);
    \draw[thick, red] (beta) to node[midway, sloped] {\tikz[baseline=-0.5ex]\draw[-{Stealth}] (0,0)--(0.5,0);} (y);
    \draw[thick, blue] (delta) -- ++(210:6mm) -- ++(90:6mm) -- cycle;
\end{tikzpicture}}\nonumber
\end{equation}

\begin{equation}
\hbox{\begin{tikzpicture}
    \tikzstyle{vertex} = [circle, fill=black, inner sep=1.5pt]
    
    \node at (-2, 0) {$\mathcal{G}_{31,s}(z)_{ab} =$};
    
    \node[vertex, label=below:$a$] (a) at (0, 0) {};
    \node[vertex, label=below:$x$] (x) at (2, 0) {};
    \node[vertex, label=below:$\gamma$] (gamma) at (4, 0) {};
    \node[vertex, label=below:$y$] (y) at (6, 0) {};
    \node[vertex, label=below:$b$] (b) at (8, 0) {};
    \node[vertex, label=above:$\alpha$] (alpha) at (4,2) {};
    \node[vertex, label=above:$\beta$] (beta) at (6,2) {};
    \node[vertex, label=below:$\delta$] (delta) at (5,1) {};
    \node[vertex, label=below:$u$] (u) at (3,1)  {};
    \node[vertex, label=left:$v$] (v) at (3.5, 1.5) {};

    \draw[double, thick] (a) -- (x);
    \draw[double, thick] (y) -- (b);                  
    \draw[thick, blue] (gamma) -- (y) node[midway, sloped] {\tikz[baseline=-0.5ex]\draw[-{Stealth}] (0,0)--(0.5,0);} ;
    \draw[thick, red, bend right=30] (x) to node[midway, sloped] {\tikz[baseline=-0.5ex]\draw[-{Stealth}] (0,0)--(0.5,0);} (y);
    \draw[thick, blue, decorate, decoration={snake, amplitude=1mm, segment length=3mm}] (v) -- (alpha);
    \draw[thick, decorate, decoration={snake, amplitude=1mm, segment length=3mm}] (u) -- (v);
    \draw[thick, blue, decorate, decoration={snake, amplitude=1mm, segment length=3mm}] (u) -- (gamma);
    \draw[thick, blue, decorate, decoration={snake, amplitude=1mm, segment length=3mm}] (gamma) -- (delta);
    \draw[thick, decorate, decoration={snake, amplitude=1mm, segment length=3mm}] (alpha) -- (beta);
    \draw[thick, blue, bend left=90] (x) to node[midway, sloped] {\tikz[baseline=-0.5ex]\draw[-{Stealth}] (0,0)--(0.5,0);} (beta);
    \draw[thick, blue] (alpha) -- (gamma) node[midway, sloped] {\tikz[baseline=-0.5ex]\draw[-{Stealth}] (0,0)--(0.5,0);} ;
    \draw[thick, red, bend left=30] (beta) to node[midway, sloped] {\tikz[baseline=-0.5ex]\draw[-{Stealth}] (0.5,0)--(0,0);} (alpha);
    \draw[thick, blue] (delta) -- ++(30:6mm) -- ++(270:6mm) -- cycle;
\end{tikzpicture}}\nonumber
\end{equation}

\begin{equation}
\hbox{\begin{tikzpicture}
    \tikzstyle{vertex} = [circle, fill=black, inner sep=1.5pt]
    
    \node at (-2, 0) {$\mathcal{G}_{41,s}(z)_{ab} =$};
    
    \node[vertex, label=below:$a$] (a) at (0, 0) {};
    \node[vertex, label=below:$x$] (x) at (2, 0) {};
    \node[vertex, label=below:$\gamma$] (gamma) at (4, 0) {};
    \node[vertex, label=below:$y$] (y) at (6, 0) {};
    \node[vertex, label=below:$b$] (b) at (8, 0) {};
    \node[vertex, label=above:$\alpha$] (alpha) at (4,2) {};
    \node[vertex, label=above:$\beta$] (beta) at (6,2) {};
    \node[vertex, label=above:$\delta$] (delta) at (2,1) {};
    \node[vertex, label=above:$u$] (u) at (2,2) {};
    \node[vertex, label=above:$v$] (v) at (3, 2) {};

    \draw[double, thick] (a) -- (x);
    \draw[double, thick] (y) -- (b);      
    \draw[thick, blue] (x) -- (gamma) node[midway, sloped] {\tikz[baseline=-0.5ex]\draw[-{Stealth}] (0,0)--(0.5,0);} ;            
    \draw[thick, red, bend right=30] (x) to node[midway, sloped] {\tikz[baseline=-0.5ex]\draw[-{Stealth}] (0,0)--(0.5,0);} (y);
    \draw[thick, blue, decorate, decoration={snake, amplitude=1mm, segment length=3mm}] (v) -- (alpha);
    \draw[thick, decorate, decoration={snake, amplitude=1mm, segment length=3mm}] (u) -- (v);
    \draw[thick, decorate, decoration={snake, amplitude=1mm, segment length=3mm}] (alpha) -- (beta);
    \draw[thick, blue] (gamma) to node[midway, sloped] {\tikz[baseline=-0.5ex]\draw[-{Stealth}] (0,0)--(0.5,0);} (beta);
    \draw[thick, blue] (alpha) to node[midway, sloped] {\tikz[baseline=-0.5ex]\draw[-{Stealth}] (0,0)--(0.5,0);} (y);
    \draw[thick, red, bend right=30] (beta) to node[midway, sloped] {\tikz[baseline=-0.5ex]\draw[-{Stealth}] (0.5,0)--(0,0);} (alpha);
    \draw[thick, blue, decorate, decoration={snake, amplitude=1mm, segment length=3mm}] (u) -- (gamma);
    \draw[thick, blue, decorate, decoration={snake, amplitude=1mm, segment length=3mm}] (gamma) -- (delta);
    \draw[thick, blue] (delta) -- ++(210:6mm) -- ++(90:6mm) -- cycle;
\end{tikzpicture}}\nonumber
\end{equation}
For $\mathcal{G}_{11,s}(z)_{ab}$, we bound each {\color{blue}$\Delta$} and red line by $\mathrm{O}(W^{\delta_{\mathrm{stop}}/20}W^{-1/2}|\mathrm{Im}w_{s}|^{-1/4})$, and then we sum out $\beta,\alpha,u,\delta$ in that order. For $\mathcal{G}_{21,s}(z)_{ab}$, we do the same. For $\mathcal{G}_{31,s}(z)_{ab}$, we bound each red line and {\color{blue}$\Delta$} and the $\alpha{\color{blue}\to}\gamma$ line by $\mathrm{O}(W^{\delta_{\mathrm{stop}}/20}W^{-1/2}|\mathrm{Im}w_{s}|^{-1/4})$ each. Then, we sum out $\delta$. For $\mathcal{G}_{41,s}(z)_{ab}$, we bound each {\color{blue}$\Delta$} and red line and the $\gamma{\color{blue}\to}\beta$ line by $\mathrm{O}(W^{\delta_{\mathrm{stop}}/20}W^{-1/2}|\mathrm{Im}w_{s}|^{-1/4})$ each. Then, we sum out $\delta,\beta$. In each case, we have a total of four factors of $\mathrm{O}(W^{\delta_{\mathrm{stop}}/20}W^{-1/2}|\mathrm{Im}w_{s}|^{-1/4})$, and the resulting bound looks like (upon relabeling indices)
\begin{center}
\begin{tikzpicture}
    \tikzstyle{vertex} = [circle, fill=black, inner sep=1.5pt]
    
    \node at (-2, 0) {$W^{\frac{\delta_{\mathrm{stop}}}{5}}W^{-2}|\mathrm{Im}w_{s}|^{-1}\times$};
    
    \node[vertex, label=below:$a$] (a) at (0, 0) {};
    \node[vertex, label=below:$x$] (x) at (2, 0) {};
    \node[vertex, label=below:$u$] (u) at (4, 0) {};
    \node[vertex, label=below:$\alpha$] (alpha) at (6, 0) {};
    \node[vertex, label=below:$y$] (y) at (8, 0) {};
    \node[vertex, label=below:$b$] (b) at (10, 0) {};

    \draw[double, thick] (a) -- (x);
    \draw[double, thick] (y) -- (b);                  
    \draw[thick, blue] (x) -- (u) node[midway, sloped] 
    {\tikz[baseline=-0.5ex]\draw[-{Stealth}] (0,0)--(0.5,0);} ; 
    \draw[thick, blue] (alpha) -- (y) node[midway, sloped] {\tikz[baseline=-0.5ex]\draw[-{Stealth}] (0,0)--(0.5,0);} ;
    \draw[thick, gray, decorate, decoration={snake, amplitude=1mm, segment length=3mm}] (u) -- (alpha);
\end{tikzpicture}
\end{center}
Above, the gray line indicates a matrix whose entries are non-negative and whose row and column sums are $\mathrm{O}(1)$, and each line is interpreted to mean its absolute value. (For example, it can indicate a black wavy line, a blue wavy line, or a composition of such wavy lines.) Now, we note that the diagram above is the same as the diagram we obtained in our bounds for $|\mathcal{G}_{i0,s}(z)_{ab}|$, except we need another factor of \eqref{eq:g1-g4;bulk1} to sum out another solid line. In particular, the diagram above is bounded by \eqref{eq:g1-g4;bulk0} times $W^{1/2+\varepsilon}|\mathrm{Im}w_{s}|^{-3/4}$. Thus,
\begin{align}
\max_{i=1,2,3,4}|\mathcal{G}_{i1,s}(z)_{ab}|&\lesssim W^{2\varepsilon}W^{\frac{\delta_{\mathrm{stop}}}{4}}W^{-2}|\mathrm{Im}w_{t}|^{-\frac32}|\mathrm{Im}w_{s}|^{-\frac12}+W^{\varepsilon}W^{\frac{\delta_{\mathrm{stop}}}{4}}W^{-2}|\mathrm{Im}w_{t}|^{-1}|\mathrm{Im}w_{s}|^{-1}.\label{eq:g1-g4;bulk3}
\end{align}
Integrating the previous bound over $s\in[0,t]$ and making the change of variables $\sigma=|\mathrm{Im}w_{s}|$ proves \eqref{eq:g1-g4;bulk1main} for $j=1$. As mentioned before, the proof for $j=2$ is identical, so the proof is complete.
\end{proof}
Before we present the next estimate, recall $\mathcal{G}_{i3,s}(z)_{ab}$ from Lemma \ref{lemma:level2expand}.
\begin{lemma}\label{lemma:g1-g4;bulk2}
Recall $\tau_{\mathrm{stop}}$ from \eqref{eq:taubulk}, and recall that $N\ll W^{11/8}$. We have the deterministic estimate
\begin{align}
\max_{a,b}\max_{i=1,\ldots,4}\left|\int_{0}^{\tau_{\mathrm{stop}}\wedge t}\mathcal{G}_{i3,s}(z)_{ab}ds\right|\lesssim W^{-\frac34}|\mathrm{Im}w_{t}|^{-1}\cdot W^{-1}|\mathrm{Im}w_{t}|^{-\frac12}.\label{eq:g1-g4;bulk2main}
\end{align}
\end{lemma}
\begin{proof}
We first consider $i=1,2$. For these indices, we note that $\mathcal{G}_{i,s}(z)_{ab}$ has the following form (which we explain afterwards):
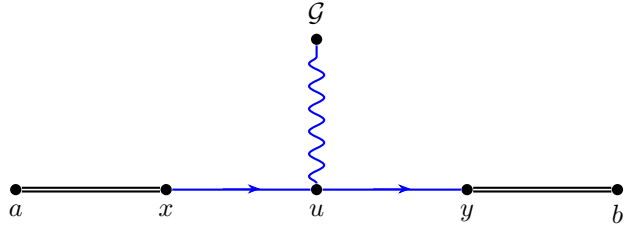
\begin{equation}
\hbox{\begin{tikzpicture}
    \tikzstyle{vertex} = [circle, fill=black, inner sep=1.5pt]
    
    \node[vertex, label=below:$a$] (a) at (0, 0) {};
    \node[vertex, label=below:$x$] (x) at (2, 0) {};
    \node[vertex, label=below:$u$] (u) at (4, 0) {};
    \node[vertex, label=below:$y$] (y) at (6, 0) {};
    \node[vertex, label=below:$b$] (b) at (8, 0) {};
    \node[vertex, label=above:$\mathcal{G}$] (G) at (4,2) {};

    \draw[double, thick] (a) -- (x);
    \draw[double, thick] (y) -- (b);                  
    \draw[thick, blue] (x) -- (u) node[midway, sloped] 
    {\tikz[baseline=-0.5ex]\draw[-{Stealth}] (0,0)--(0.5,0);} ; 
    \draw[thick, blue] (u) -- (y) node[midway, sloped] {\tikz[baseline=-0.5ex]\draw[-{Stealth}] (0,0)--(0.5,0);};
    \draw[thick, blue, decorate, decoration={snake, amplitude=1mm, segment length=3mm}] (u) -- (G);
\end{tikzpicture}}\nonumber
\end{equation}

Here, $\mathcal{G}$ is a graph whose vertices are all connected by wavy lines, and in which the number of solid lines plus the number of triangles is equal to $3$. (We allow $\mathcal{G}$ to ``interact with" the vertices $x,y$; i.e. there can be an edge connecting $\mathcal{G}$ to $x$ and/or $y$.) In particular, for $i=1,2$, the graph $\mathcal{G}_{i3,s}(z)_{ab}$ has the form
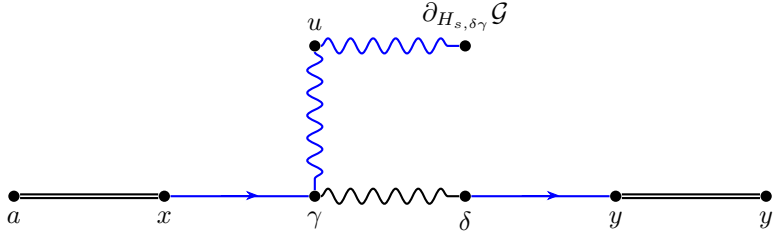
\begin{equation}
\hbox{\begin{tikzpicture}
    \tikzstyle{vertex} = [circle, fill=black, inner sep=1.5pt]
    
    \node[vertex, label=below:$a$] (a) at (0, 0) {};
    \node[vertex, label=below:$x$] (x) at (2, 0) {};
    \node[vertex, label=below:$\gamma$] (gamma) at (4, 0) {};
    \node[vertex, label=below:$\delta$] (delta) at (6, 0) {};
    \node[vertex, label=below:$y$] (y) at (8, 0) {};
    \node[vertex, label=below:$y$] (b) at (10, 0) {};
    \node[vertex, label=above:$u$] (u) at (4,2) {};
    \node[vertex, label=above:$\partial_{H_{s,\delta\gamma}}\mathcal{G}$] (G) at (6,2) {};

    \draw[double, thick] (a) -- (x);
    \draw[double, thick] (y) -- (b);                  
    \draw[thick, blue] (x) -- (gamma) node[midway, sloped] 
    {\tikz[baseline=-0.5ex]\draw[-{Stealth}] (0,0)--(0.5,0);} ; 
    \draw[thick, blue] (delta) -- (y) node[midway, sloped] {\tikz[baseline=-0.5ex]\draw[-{Stealth}] (0,0)--(0.5,0);};
    \draw[thick, blue, decorate, decoration={snake, amplitude=1mm, segment length=3mm}] (u) -- (G);
    \draw[thick, decorate, decoration={snake, amplitude=1mm, segment length=3mm}] (gamma) -- (delta);
    \draw[thick, blue, decorate, decoration={snake, amplitude=1mm, segment length=3mm}] (u) -- (gamma);
\end{tikzpicture}}\nonumber
\end{equation}
where $\partial_{H_{s,\delta\gamma}}\mathcal{G}$ means the following. Take the product of all $G_{s}(z)_{\alpha\beta}$ factors appearing in $\mathcal{G}$, and replace this entire product by its derivative with respect to $H_{s,\delta\gamma}$. Now, recall from resolvent perturbation that $\partial_{H_{s,\delta\gamma}}G_{s}(z)_{\alpha\beta}=-G_{s}(z)_{\alpha\delta}G_{s}(z)_{\gamma\beta}$ and $\partial_{H_{s,\delta\gamma}}\overline{G}_{s}(z)_{\alpha\beta}=-\overline{G}_{s}(z)_{\alpha\gamma}\overline{G}_{s}(z)_{\delta\beta}$. Thus, $\partial_{H_{s,\delta\gamma}}\mathcal{G}$ is a finite linear combination of graphs in which the number of solid lines plus the number of triangles is equal to $4$. Thus, we can bound the ${\color{blue}\leftrightsquigarrow}\partial_{H_{s,\delta\gamma}}\mathcal{G}$-part of the graph by $\mathrm{O}(W^{\delta_{\mathrm{stop}}/5}W^{-2}|\mathrm{Im}w_{s}|^{-1})$; we recall that all the vertices in each graph in $\partial_{H_{s,\delta\gamma}}\mathcal{G}$ are connected by wavy lines, so we can bound each solid line and triangle in $\partial_{H_{s,\delta\gamma}}\mathcal{G}$ and sum out the wavy lines. Then, we sum out $u$ to get a factor of $\mathrm{O}(1)$. The remaining horizontal part of the graph was controlled at the end of the proof of Lemma \ref{lemma:g1-g4;bulk1}; see \eqref{eq:g1-g4;bulk3}. In particular, the bound \eqref{eq:g1-g4;bulk3} holds also for $\max_{i=1,2}|\mathcal{G}_{i3,s}(z)_{ab}|$; after integrating it over $s\in[0,t]$ and making the change-of-variables $\sigma=|\mathrm{Im}w_{s}|$ as we did after \eqref{eq:g1-g4;bulk3}, we ultimately deduce the desired bound \eqref{eq:g1-g4;bulk2main} for $i=1,2$.

We are left to control the LHS of \eqref{eq:g1-g4;bulk2main} for $i=3,4$. For these indices, the graph $\mathcal{G}_{i,s}(z)_{ab}$ has the following form (which we explain afterwards):
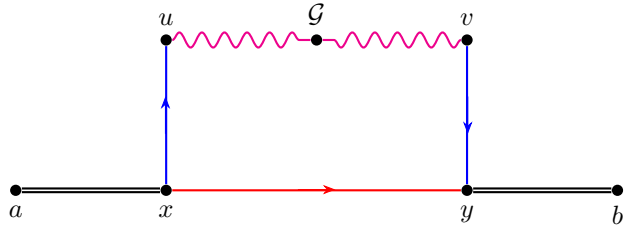
\begin{equation}
\hbox{\begin{tikzpicture}
    \tikzstyle{vertex} = [circle, fill=black, inner sep=1.5pt]
    
    \node[vertex, label=below:$a$] (a) at (0, 0) {};
    \node[vertex, label=below:$x$] (x) at (2, 0) {};
    \node[vertex, label=above:$u$] (u) at (2, 2) {};
    \node[vertex, label=above:$v$] (v) at (6, 2) {};
    \node[vertex, label=below:$y$] (y) at (6, 0) {};
    \node[vertex, label=below:$b$] (b) at (8, 0) {};
    \node[vertex, label=above:$\mathcal{G}$] (G) at (4,2) {};

    \draw[double, thick] (a) -- (x);
    \draw[double, thick] (y) -- (b);  
    \draw[thick, red] (x) -- (y) node[midway, sloped] 
    {\tikz[baseline=-0.5ex]\draw[-{Stealth}] (0,0)--(0.5,0);} ; 
    \draw[thick, blue] (x) -- (u) node[midway, sloped] 
    {\tikz[baseline=-0.5ex]\draw[-{Stealth}] (0,0)--(0.5,0);} ; 
    \draw[thick, blue] (v) -- (y) node[midway, sloped] {\tikz[baseline=-0.5ex]\draw[-{Stealth}] (0,0)--(0.5,0);};
    \draw[thick, magenta, decorate, decoration={snake, amplitude=1mm, segment length=3mm}] (u) -- (G);
    \draw[thick, magenta, decorate, decoration={snake, amplitude=1mm, segment length=3mm}] (v) -- (G);
\end{tikzpicture}}\nonumber
\end{equation}
Above, the pink wavy lines indicate that the wavy lines can be blue or black; the color will not be important for us. The graph $\mathcal{G}$ is a graph whose vertices are all connected by wavy lines, and the number of solid lines in $u{\color{magenta}\leftrightsquigarrow}\mathcal{G}{\color{magenta}\leftrightsquigarrow}v$ is $2$ (we note that $\mathcal{G}$ is allowed to interact with $u,v$, in that there may be lines connecting $\mathcal{G}$ to $u$ and/or $v$ that are not depicted by the pink wavy lines; it is after counting these possible lines that the number of solid lines in $u{\color{magenta}\leftrightsquigarrow}\mathcal{G}{\color{magenta}\leftrightsquigarrow}v$ is $2$). When we do the regular vertex expansion in Lemma \ref{lemma:expansion}, it either happens at an incoming and outgoing arrow at $u$ or an incoming and outgoing arrow at $v$. By symmetry, it is enough to only handle the case where it happens at $u$. In this case, the integration-by-parts term in the regular vertex expansion in Lemma \ref{lemma:expansion} produces $\mathcal{G}_{i3,s}(z)_{ab}$, which turns out to be a linear combination of two different types of graphs. The first type is the following, in which the derivative hits a solid line in $\mathcal{G}$:
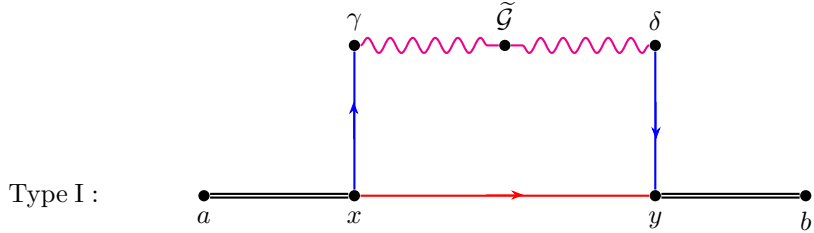
\begin{equation}
\hbox{\begin{tikzpicture}
    \tikzstyle{vertex} = [circle, fill=black, inner sep=1.5pt]

    \node at (-2, 0) {$\mathrm{Type} \ \mathrm{I}:$};
    
    \node[vertex, label=below:$a$] (a) at (0, 0) {};
    \node[vertex, label=below:$x$] (x) at (2, 0) {};
    \node[vertex, label=above:$\gamma$] (u) at (2, 2) {};
    \node[vertex, label=above:$\delta$] (v) at (6, 2) {};
    \node[vertex, label=below:$y$] (y) at (6, 0) {};
    \node[vertex, label=below:$b$] (b) at (8, 0) {};
    \node[vertex, label=above:$\widetilde{\mathcal{G}}$] (G) at (4,2) {};

    \draw[double, thick] (a) -- (x);
    \draw[double, thick] (y) -- (b);  
    \draw[thick, red] (x) -- (y) node[midway, sloped] 
    {\tikz[baseline=-0.5ex]\draw[-{Stealth}] (0,0)--(0.5,0);} ; 
    \draw[thick, blue] (x) -- (u) node[midway, sloped] 
    {\tikz[baseline=-0.5ex]\draw[-{Stealth}] (0,0)--(0.5,0);} ; 
    \draw[thick, blue] (v) -- (y) node[midway, sloped] {\tikz[baseline=-0.5ex]\draw[-{Stealth}] (0,0)--(0.5,0);};
    \draw[thick, magenta, decorate, decoration={snake, amplitude=1mm, segment length=3mm}] (u) -- (G);
    \draw[thick, magenta, decorate, decoration={snake, amplitude=1mm, segment length=3mm}] (v) -- (G);
\end{tikzpicture}}\nonumber
\end{equation}
Above, $u{\color{magenta}\leftrightsquigarrow}\widetilde{\mathcal{G}}{\color{magenta}\leftrightsquigarrow}v$ has $3$ solid lines, since differentiating a solid line in $u{\color{magenta}\leftrightsquigarrow}\mathcal{G}{\color{magenta}\leftrightsquigarrow}v$ has the effect of producing one more solid line. (Note that $u{\color{magenta}\leftrightsquigarrow}\widetilde{\mathcal{G}}{\color{magenta}\leftrightsquigarrow}v$ still has all of its vertices connected by wavy lines.) The second type of graph which appears comes from differentiating the $\overline{G}_{s}(z)_{xy}$ line; it has the form
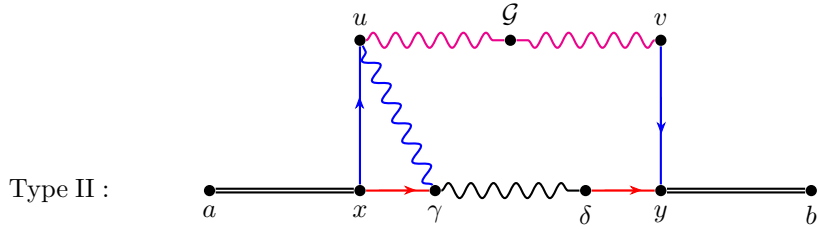
\begin{equation}
\hbox{\begin{tikzpicture}
    \tikzstyle{vertex} = [circle, fill=black, inner sep=1.5pt]

    \node at (-2, 0) {$\mathrm{Type} \ \mathrm{II}:$};
    
    \node[vertex, label=below:$a$] (a) at (0, 0) {};
    \node[vertex, label=below:$x$] (x) at (2, 0) {};
    \node[vertex, label=above:$u$] (u) at (2, 2) {};
    \node[vertex, label=above:$v$] (v) at (6, 2) {};
    \node[vertex, label=below:$y$] (y) at (6, 0) {};    \node[vertex, label=below:$\gamma$] (gamma) at (3, 0) {};    
    \node[vertex, label=below:$\delta$] (delta) at (5, 0) {};
    \node[vertex, label=below:$b$] (b) at (8, 0) {};
    \node[vertex, label=above:$\mathcal{G}$] (G) at (4,2) {};

    \draw[double, thick] (a) -- (x);
    \draw[double, thick] (y) -- (b);  
    \draw[thick, red] (x) -- (gamma) node[midway, sloped] {\tikz[baseline=-0.5ex]\draw[-{Stealth}] (0,0)--(0.5,0);} ; 
    \draw[thick, red] (delta) -- (y) node[midway, sloped] {\tikz[baseline=-0.5ex]\draw[-{Stealth}] (0,0)--(0.5,0);};
    \draw[thick, blue] (x) -- (u) node[midway, sloped] {\tikz[baseline=-0.5ex]\draw[-{Stealth}] (0,0)--(0.5,0);} ; 
    \draw[thick, blue] (v) -- (y) node[midway, sloped] {\tikz[baseline=-0.5ex]\draw[-{Stealth}] (0,0)--(0.5,0);} ; 
    \draw[thick, magenta, decorate, decoration={snake, amplitude=1mm, segment length=3mm}] (u) -- (G);
    \draw[thick, magenta, decorate, decoration={snake, amplitude=1mm, segment length=3mm}] (v) -- (G);
    \draw[thick, decorate, decoration={snake, amplitude=1mm, segment length=3mm}] (gamma) -- (delta);
    \draw[thick, blue, decorate, decoration={snake, amplitude=1mm, segment length=3mm}] (gamma) -- (u);
\end{tikzpicture}}\nonumber
\end{equation}
For the Type $\mathrm{I}$ graph, we bound the $\delta{\color{blue}\to}y$ line by $\mathrm{O}(W^{\delta_{\mathrm{stop}}/20}W^{-1/2}|\mathrm{Im}w_{s}|^{-1/4})$. Then, we bound the ${\color{magenta}\leftrightsquigarrow}\widetilde{\mathcal{G}}{\color{magenta}\leftrightsquigarrow}\delta$ piece by $\mathrm{O}(W^{3\delta_{\mathrm{stop}}/20}W^{-3/2}|\mathrm{Im}w_{s}|^{-3/4})$ using its three solid lines and the fact that all of its vertices are connected by wavy lines (so we get $\mathrm{O}(1)$ when we sum them out). Finally, we use \eqref{eq:g1-g4;bulk1} to sum out the $u$-vertex and get $\mathrm{O}(W^{\delta_{\mathrm{stop}}/20}W^{1/2+\varepsilon}|\mathrm{Im}w_{s}|^{-3/4})$. We are then left with the horizontal diagram in $\mathrm{Type} \ \mathrm{I}$; we controlled this earlier by \eqref{eq:g1-g4;bulk0}. Ultimately, we deduce that the $\mathrm{Type} \ \mathrm{I}$ graph admits a bound of 
\begin{align}
&\lesssim W^{\varepsilon}W^{\frac{\delta_{\mathrm{stop}}}{4}}W^{-\frac32}|\mathrm{Im}w_{s}|^{-\frac74}\times\eqref{eq:g1-g4;bulk0}\nonumber\\
&= W^{3\varepsilon}W^{\frac{3\delta_{\mathrm{stop}}}{10}}W^{-2}\left\{|\mathrm{Im}w_{t}|^{-\frac32}|\mathrm{Im}w_{s}|^{-\frac12}+|\mathrm{Im}w_{t}|^{-1}|\mathrm{Im}w_{s}|^{-1}\right\}.\label{eq:g1-g4;bulk4}
\end{align}
After integrating this over $s\in[0,t]$ and using the change-of-variables $\sigma=|\mathrm{Im}w_{s}|$, we deduce that the $\mathrm{Type} \ \mathrm{I}$ graph admits a bound of $\lesssim W^{3\delta_{\mathrm{stop}}/10}W^{-2}|\mathrm{Im}w_{t}|^{-3/2}$, which is controlled by the RHS of the desired estimate \eqref{eq:g1-g4;bulk2main}. Thus, it remains to control the $\mathrm{Type} \ \mathrm{II}$ graph. To this end, we note that compared to the $\mathrm{Type} \ \mathrm{I}$ graph, the $\mathrm{Type} \ \mathrm{II}$ graph has an additional red line that we must sum over. This gives a factor of $\mathrm{O}(W^{\delta_{\mathrm{stop}}/20}W^{1/2+\varepsilon}|\mathrm{Im}w_{s}|^{-3/4})$ by \eqref{eq:g1-g4;bulk1}. However, we recover this factor once we sum over $u$ through the $\gamma{\color{blue}\leftrightsquigarrow}u$ line instead of the $x{\color{blue}\to}u$ line. Finally, the extra $x{\color{blue}\to}u$ line makes up for the one-less solid line in $u{\color{magenta}\leftrightsquigarrow}\mathcal{G}{\color{magenta}\leftrightsquigarrow}v$ as compared to $u{\color{magenta}\leftrightsquigarrow}\widetilde{\mathcal{G}}{\color{magenta}\leftrightsquigarrow}v$. In particular, $\mathrm{Type} \ \mathrm{II}$ admits the same bound as $\mathrm{Type} \ \mathrm{I}$.
\end{proof}

Finally, we control the fluctuations arising in Lemma \ref{lemma:level2expand}.
\begin{lemma}\label{lemma:g1-g4;bulk3}
Recall that $N\ll W^{11/8}$. Then, the fluctuation terms coming from the second expansion are bounded by
\begin{align}\label{eq-second-fluct-bulk}
    \max_{a,b} \max_{i\in\llbracket1,4\rrbracket} \left|\int_0^{\tau_{\mathrm{stop}}\wedge t} \mathcal{F}_{i,s}(z)_{ab} ds\right| \prec W^{\frac{3\delta_{\mathrm{stop}}}{10}} W^{-1+2\varepsilon}|\mathrm{Im}w_{t}|^{-1}\cdot W^{-1}|\mathrm{Im}w_{t}|^{-\frac12}.
\end{align}
The fluctuation term coming from the first expansion is bounded by
\begin{align}\label{eq-first-fluct-bulk}
\max_{a,b}\left|\int_0^{\tau_{\mathrm{stop}}\wedge t} \mathcal{F}_{0,s}(z)_{ab} ds\right| \prec W^{\frac{3\delta_{\mathrm{stop}}}{10}} W^{-1+2\varepsilon}|\mathrm{Im}w_{t}|^{-1}\cdot W^{-1}|\mathrm{Im}w_{t}|^{-\frac12}.
\end{align}
\end{lemma}

\begin{proof}
In order to bound the fluctuations $\mathcal{F}_{i,s}(z)_{ab}$ for $i\in\llbracket1,4\rrbracket$ in the sense of stochastic domination, we consider their even moments $\mathfrak{M}_{i,s,ab}^{2p}(z)$ with $p\in \Z_+$. We integrate by parts with respect to every $H_s$ appearing in the moments. Then $\mathfrak{M}_{i,s,ab}^{2p}(z)$ decomposes into a sum of several terms, each represented by the expectation of a graph. We introduce the parameter $k$ -- the number of times the integration by parts with respect to one of the $H_s$ hit another $H_s$. It is straightforward to verify that each term in the expansions satisfies the following properties.
\begin{itemize}
    \item The graph has $2p$ double edges connected to $a$ and $2p$ double edges connected to $b$. For convenience, we will treat these edges as not adjacent to each other by treating $a$ and $b$ as $2p$ distinct copies of themselves. We will refer to these vertices as outer vertices and all other vertices -- as inner vertices. Then the $4p$ double edges are connected to $4p$ distinct inner vertices.

    \item All of the vertices that are not incidental to a double edge form $2p - k$ connected components with respect to the waved edges. Each of the components is a tree with respect to the waved edges.

    \item The graph has $12p - 2k$ oriented $G$-edges, all of which split into disjoint loops.

    \item Each vertex is connected by a path to at least one $a$-vertex.
\end{itemize}
Now we consider $s\in[0.\tau_{\mathrm{edge}}]$ and explain how the can bound the size of each term in the expansion of $\mathfrak{M}_{i,s,ab}^{2p}(z)$. The strategy is similar to the proof of the bound on the main terms in Lemma \ref{lemma:g1-g4;bulk1}. Some of the edges will be bounded by the isotropic bound \eqref{eq:bulkgbound}. And the other edges will be summed out in the appropriate order using the bound \eqref{eq:g1-g4;bulk1} and the bound $\sum_y|B_s(z)_{xy}|+\sum_y|S_{xy}| = O(1)$ from Lemma \ref{lemma:Bestimates}. First, since every connected component consisting of waved edges is a tree, we will be able to sum out every waved edge. Then to pick out the remaining summation edges, we collapse each waved connected component into a single vertex and cover the resulting graph with $2p$ disjoint trees rooted at $2p$ $a$-vertices. This selection of the summation edges ensures that we take advantage of the waved edges to the full extent. 

Since there are $2p-k$ connected components of waved edges and $4p$ inner vertices incidental to a double edge, the selected trees have $4p - k$ oriented $G$-edges. Then the number of unselected $G$-edges is $8p-k$. Bounding each of them with an isotropic bound introduces a multiplicative factor $\left(W^{\delta_{\mathrm{edge}}/20}W^{-1/2}|\im w_s|^{-1/8}\right)^{8p-k}$. We bound every double edge incidental to $b$ by a sub-optimal isotropic bound 
\[
|S^{1/2}_{\alpha b}+(t-s)(\Theta_{t}S^{1/2})_{\alpha b}|\lesssim W^{-1}|\im w_t|^{-1/2}|\im w_s|^{1/2}.
\]
Then we sum out all inner vertices not incidental to a double edge one by one, taking one of the remaining leaves of the trees at a time. This introduces a factor $\left(W^{\delta_{\mathrm{stop}}/20}W^{1/2+\varepsilon}|\im w_s|^{-3/4}\right)^{4p - k}$ coming from the summation of the $G$-edges, and a factor $O(1)$ from the summation of the waved edges. Finally, summing out each double edge connected to $a$ introduces a factor $|\im w_t|^{-1}|\im w_s|$. In total, we get
\begin{align*}
\mathfrak{M}_{s,ab}^{2p}(z) &\lesssim \max_{k\in\llbracket0,p\rrbracket}\Big[\left(W^{\delta_{\mathrm{stop}}/20}W^{-1/2}|\im w_s|^{-1/4}\right)^{8p-k}\left(W^{-1}|\im w_t|^{-1/2}|\im w_s|^{1/2}\right)^{2p}\\
&\times\left(W^{\delta_{\mathrm{stop}}/20}W^{1/2+\varepsilon}|\im w_s|^{-3/4}\right)^{4p - k}\left(|\im w_t|^{-1}|\im w_s|\right)^{2p}\Big] \\
&=\max_{k\in\llbracket0,p\rrbracket}\left[\left(W^{6\delta_{\mathrm{stop}}/20}W^{-2+2\varepsilon}|\im w_s|^{-1} |\im w_t|^{-3/2} \right)^{2p} \left(W^{\delta_{\mathrm{stop}}/10+\varepsilon}|\im w_s|^{-1}\right)^{-k}\right]\\
&=\left(W^{3\delta_{\mathrm{stop}}/10}W^{-2+2\varepsilon}|\im w_s|^{-1} |\im w_t|^{-3/2} \right)^{2p}.
\end{align*}
Using Chebyshev inequality, this gives us for $i\in\llbracket1,4\rrbracket$
\begin{equation*}
\left|\mathcal{F}_{i,s}(z)_{ab}\right| \prec W^{3\delta_{\mathrm{edge}}/10+2\varepsilon} W^{-2} |\im w_t|^{-3/2} |\im w_s|^{-1}.
\end{equation*}
By applying a continuity argument similar to Lemma \ref{lemma:mstop}, we see that this bound holds uniformly in $a, b\in\llbracket1,N\rrbracket$ and $s, t\in[0,\tau_{\mathrm{edge}}]$ with $s\le t$. Finally, integrating this bound in $s$ completes the proof of \eqref{eq-second-fluct-bulk}.

Proving \eqref{eq-first-fluct-bulk} requires additional expansions. Again, we consider the $2p$-th moments of the fluctuation $\mathcal{F}_{0,s}(z)_{ab}$ when $s\in[0,\tau_{\mathrm{stop}}]$ -- $\mathfrak{M}_{0,s,ab}^{2p}(z)$ -- and expand it using integration by parts. Similarly to the computation for $\mathfrak{M}_{i,s,ab}^{2p}(z)$ above, we can apply the same summation strategy. The resulting bound on the size of the terms will be worse due to the number of $G$ edges that we bound by the isotropic bound being reduced by $2p$. More precisely, the terms of $\mathfrak{M}_{0,s,ab}^{2p}(z)$ with $k$ integrations by parts hitting $H_s$ are bounded by 
\begin{align*}
\left(W^{5\delta_{\mathrm{stop}}/20}W^{-3/2+2\varepsilon}|\im w_s|^{-3/4} |\im w_t|^{-3/2} \right)^{2p} \left(W^{\delta_{\mathrm{stop}}/10+\varepsilon}|\im w_s|^{-1}\right)^{-k}.
\end{align*}
This bound is insufficient. We show now how this bound can be improved by expanding the terms of $\mathfrak{M}_{0,s,ab}^{2p}(z)$ further. This requires a more precise understanding of the graph structure of the terms. The graph representation of each term consists of $2p$ pieces with the following structure. The baseline piece is given by the graph \eqref{eq:fluct-piece} or by its complex conjugate. 
\begin{equation}
\hbox{\begin{tikzpicture}
    \tikzstyle{vertex} = [circle, fill=black, inner sep=1.5pt]

    \node[vertex, label=below:$a$] (a) at (0, 0) {};
    \node[vertex, label=below:$x$] (x) at (2, 0) {};
    \node[vertex, label=below:$u$] (u) at (4, 0) {};
    \node[vertex, label=below:$y$] (y) at (6, 0) {};
    \node[vertex, label=below:$b$] (b) at (8, 0) {};
    \node[vertex, label=above:$v$] (v) at (4, 2) {};
    \node[vertex, label=above:$\alpha$] (alpha) at (6, 2) {};
    \node[vertex, label=above:$\beta$] (beta) at (8, 2) {};

    \draw[double, thick] (a) -- (x);
    \draw[double, thick] (y) -- (b);                  
    \draw[thick, blue] (x) -- (u) node[midway, sloped] 
    {\tikz[baseline=-0.5ex]\draw[-{Stealth}] (0,0)--(0.5,0);} ; 
    \draw[thick, blue] (u) -- (y) node[midway, sloped] {\tikz[baseline=-0.5ex]\draw[-{Stealth}] (0,0)--(0.5,0);} ;
    \draw[thick, red, bend right=30] (x) to node[midway, sloped] {\tikz[baseline=-0.5ex]\draw[-{Stealth}] (0,0)--(0.5,0);} (y);
    \draw[thick, decorate, decoration={snake, amplitude=1mm, segment length=3mm}] (u) -- (v);
    \draw[thick, decorate, decoration={snake, amplitude=1mm, segment length=3mm}, blue] (v) -- (alpha);
    \draw[thick, decorate, decoration={snake, amplitude=1mm, segment length=3mm}] (alpha) -- (beta);
    \draw[thick, blue, bend left=30] (beta) to node[midway, sloped] {\tikz[baseline=-0.5ex]\draw[-{Stealth}] (0.5,0)--(0,0);} (alpha);
\end{tikzpicture}}\label{eq:fluct-piece}
\end{equation}
The different pieces can be attached by identifying their $(\alpha, \beta)$ waved edges with each other. This corresponds to the case where the integration by parts hits an $H_s$. Each remaining integration by parts modifies the graph by replacing any $(\gamma, \delta)$ $G$-edge by splitting it in half and attaching the two new inner vertices to the $\alpha, \beta$ vertices of another piece. The key property of this operation is that the direction and color of the $G$-edges attached to the $x,y,u$ vertices of each piece is preserved. Then each piece has its $u$ vertex that has one incoming and one outgoing $G$-edge of the same color, one black waved edge corresponding to an $S$-factor and no other edges. We will keep track of this set of vertices $U_{reg}$ satisfying these conditions as we continue to modify the graphs. Initially it contains $2p$ vertices. As long as a vertex is in this set, we can use it to perform a regular vertex expansion from Lemma \ref{lemma:expansion}. We perform $2p$ of the expansions with respect to a vertex $u\in U_{reg}$ consecutively. Initially, every waved-edge connected component is a tree. Each step the graph is modified in one of the following ways. 
\begin{itemize}
    \item The second and third terms of the expansion of Lemma \ref{lemma:expansion} introduce two new vertices $\gamma, \delta$ connected by two waved edges to an existing connected component, one of the  new vertices has a $G$-loop attached to it, $u$ becomes disconnected from $G$-edges and instead $\gamma$ or $\delta$ is connected to the former $G$-neighbors of $u$ by the same edges. Now we apply the previously describes summation procedure to this term. This expansion causes the bound on the term to improve by a factor $W^{\delta_{\mathrm{stop}}/20}W^{-1/2}|\im w_s|^{-1/4}$ due to one additional $G$-edge and no change to the number of $G$-edge summations. Additionally, we exclude the vertex $u$ from $U_{reg}$. Other vertices in $U_{reg}$ maintain their regular property.

    \item The third term of the regular vertex expansion of Lemma \ref{lemma:expansion} introduces two new vertices $\gamma, \delta$ connected by two waved edges to an existing connected component, the two $G$-edges formerly attached to $u$ become attached to $\gamma$ and $\delta$ respectively. Due the differentiation operation $\partial_{H_{\delta\gamma}}$, one of the other edges of the graph gets split into two edges attached to $\gamma$ and $\delta$ respectively. Here, again, there is no change to the number of waved-edge-connected components and there is one additional $G$-edge, thus the bound is improved by a factor $W^{\delta_{\mathrm{stop}}/20}W^{-1/2}|\im w_s|^{-1/4}$. The set $U_{reg}$ maintains all of its vertices other than $u$.

    \item Finally, the first term of the regular vertex expansion of Lemma \ref{lemma:expansion} introduces no new vertices, replaces the two $G$-edges incidental to $u$ by one $G$-edge connecting the $G$-neighbors of $u$. The expansion vertex $u$ instead gets connected to one of its former neighbors by a blue or red waved edge. In case this new waved edge connects two waved-edge-components, the bound gets improved by a factor of $\left(W^{\delta_{\mathrm{stop}}/20}W^{1/2+\varepsilon}|\im w_s|^{-3/4}\right)^{-1}$ due to one fewer $G$-edge summations and no change in the number of $G$s bounded isotropically. In case the new waved edge connects the vertices in the same waved-edge-connected component, we designate the black waved edge incidental to $u$ to be bounded isotropically by $W^{-1}$. The remaining edges of the component form a tree and will be used for summation. This way the bound gets an additional factor of $\left(W^{\delta_{\mathrm{stop}}/20}W^{-1/2}|\im w_s|^{-1/4}\right)^{-1}$ due to one fewer $G$-edges available to bound isotropically and a $W^{-1}$ factor from $S$-edge. Thus the bound is improved by $W^{-\delta_{\mathrm{stop}}/20}W^{-1/2}|\im w_s|^{1/4}$. This procedure also maintains the rest of the set $U_{reg}$ excluding $u$ and the waved-edge-components can be considered a tree as we will continue to use the isotropic bound on the $S$-edge in the following steps.
\end{itemize}

In conclusion, this shows that we are able to carry out $2p$ expansions and the total improvement to the bound is at least $\left(W^{\delta_{\mathrm{stop}}/20}W^{-1/2}|\im w_s|^{-1/4}\right)^{2p}$. Thus, we get
\begin{align*}
\mathfrak{M}_{0,s,ab}^{2p}(z) &\lesssim \left(W^{\delta_{\mathrm{stop}}/20}W^{-3/2+2\varepsilon}|\im w_s|^{-3/4} |\im w_t|^{-3/2} \right)^{2p} \left(W^{\delta_{\mathrm{stop}}/20}W^{-1/2}|\im w_s|^{-1/4}\right)^{2p}\\
&= \left(W^{6\delta_{\mathrm{stop}}/20}W^{-2+2\varepsilon}|\im w_s|^{-1} |\im w_t|^{-3/2} \right)^{2p}.
\end{align*}
Then by Chebyshev inequality,
\begin{equation*}
\left|\mathcal{F}_{0,s}(z)_{ab}\right| \prec W^{6\delta_{\mathrm{stop}}/20}W^{-2+2\varepsilon}|\im w_s|^{-1} |\im w_t|^{-3/2}.
\end{equation*}
Using a union bound, continuity argument and integrating, we conclude the proof.
\end{proof}
\subsection{Proof of Proposition \ref{prop:dstop}}
We combine Lemmas \ref{lemma:expand1}, \ref{lemma:level2expand}, \ref{lemma:g1-g4;bulk1}, \ref{lemma:g1-g4;bulk2}, and \ref{lemma:g1-g4;bulk3}. This shows that the graph in \eqref{eq:d1}, when evaluated at time $s$ and integrated over $s\in[0,\tau_{\mathrm{stop}}\wedge t]$, is $\prec W^{-3/4}|\mathrm{Im}w_{t}|^{-1}\cdot W^{-1}|\mathrm{Im}w_{t}|^{-1/2}$. Its complex conjugate admits the same bound as well. Now, we recall from \eqref{eq:d0} that $[\{\mathrm{Id}+(t-s)\Theta_{t}\}S^{\frac12}\Omega_{s}(z)S^{\frac12}\{\mathrm{Id}+(t-s)\Theta_{t}\}]_{ab}$ is the sum of the graph in \eqref{eq:d1} and its complex conjugate. Therefore, this paragraph implies that %
\begin{align*}
\max_{a,b}\left|\int_{0}^{\tau_{\mathrm{stop}}\wedge t}[\{\mathrm{Id}+(t-s)\Theta_{t}\}S^{\frac12}\Omega_{s}(z)S^{\frac12}\{\mathrm{Id}+(t-s)\Theta_{t}\}]_{ab}ds\right|\prec W^{-\frac34}|\mathrm{Im}w_{t}|^{-1}\cdot W^{-1}|\mathrm{Im}w_{t}|^{-\frac12}.
\end{align*}
The integrand on the LHS of the above display is $\mathrm{O}(W^{C})$ for $C=\mathrm{O}(1)$; this follows by naive polynomial-in-$W$ bounds for all matrix entries appearing on the LHS above. Thus, a standard net argument gives us
\begin{align*}
\sup_{t\in[0,1]}\max_{a,b}\frac{\left|\int_{0}^{\tau_{\mathrm{stop}}\wedge t}[\{\mathrm{Id}+(t-s)\Theta_{t}\}S^{\frac12}\Omega_{s}(z)S^{\frac12}\{\mathrm{Id}+(t-s)\Theta_{t}\}]_{ab}ds\right|}{W^{-\frac34}|\mathrm{Im}w_{t}|^{-1}\cdot W^{-1}|\mathrm{Im}w_{t}|^{-\frac12}}\prec 1.
\end{align*}
The integral on the LHS is the term $\mathcal{E}^{D,\mathrm{stop}}_{t}(z)_{ab}$ that we want to bound in Proposition \ref{prop:dstop}. In particular, the proof of Proposition \ref{prop:dstop} follows immediately from the previous display, so we are done. \qed
%
%
%
\section{Some refinements of the analysis}\label{section:conjecture}
For any $u\in\{1,\ldots,N\}$, we re-introduce the matrices $S^{u}_{xy}=\delta_{xy}S_{ux}$ and $S^{1/2,u}_{xy}=\delta_{xy}S^{1/2}_{ux}$. We claim that the following estimate is true for any $s\leq\tau_{\mathrm{stop}}$ (see \eqref{eq:taubulk}):
\begin{align}
\|G_{s}(z)S^{u}G_{s}(z)^{\ast}\|_{\mathrm{op}}+|\mathrm{Im}w_{s}|^{-1}\|\sqrt{S^{\frac12,u}}\mathrm{Im}G_{s}(z)\sqrt{S^{\frac12,u}}\|_{\mathrm{op}}\lesssim W^{-\frac12}|\mathrm{Im}w_{s}|^{-\frac54}.\label{eq:conj}
\end{align}
We briefly explain where this comes from, assuming that $T_{s}(z)\approx\Theta_{s}$ at the level of entries before time $\tau_{\mathrm{stop}}$. (This approximation will be justified in the following.) We use Ward and a trace bound for the operator norm in the following fashion:
\begin{align*}
\|G_{s}(z)S^{u}G_{s}(z)^{\ast}\|_{\mathrm{op}}&\lesssim|\mathrm{Im}w_{s}|^{-1}\|\sqrt{S^{u}}\mathrm{Im}G_{s}(z)\sqrt{S^{u}}\|_{\mathrm{op}}\lesssim|\mathrm{Im}w_{s}|^{-1}|\left\{\mathrm{Tr}S^{u}\mathrm{Im}G_{s}(z)S^{u}\mathrm{Im}G_{s}(z)^{\ast}\right\}^{\frac12}\\
&\lesssim|\mathrm{Im}w_{s}|^{-1}\left\{\sum_{a,x,y,v} S^{u}_{ax}\mathrm{Im}G_{s}(z)_{xy}S^{u}_{yv}\mathrm{Im}G_{z}(s)^{\ast}_{va}\right\}^{\frac12}\\
&\lesssim|\mathrm{Im}w_{s}|^{-1}\left\{\sum_{a,y} S_{ua}|G_{s}(z)_{ay}|^{2}S_{yu}\right\}^{\frac12}\\
&\lesssim |\mathrm{Im}w_{s}|^{-1}\max_{a,b}\{T_{s}(z)_{ab}\}^{\frac12}\lesssim W^{-\frac12}|\mathrm{Im}w_{s}|^{-\frac54}.
\end{align*}
where the last line above follows because $S^{1/2}$ is an averaging operator, and before time $\tau_{\mathrm{stop}}$, the entries of $T_{s}(z)$ are $\mathrm{O}(W^{-1}|\mathrm{Im}w_{s}|^{-1/2})$ (since they are controlled by entries of $\Theta_{s}$ before time $\tau_{\mathrm{stop}}$; now use Lemma \ref{lemma:thetaestimates} to control the entries of $\Theta_{s}$). This controls the first term on the LHS of \eqref{eq:conj}. The second term is treated the same, since $S$ and its square root have the same bounds by assumption. We also conjecture that 
\begin{align}
\|G_{s}(z)S^{u}G_{s}(z)^{\ast}\|_{\mathrm{op}}+|\mathrm{Im}w_{s}|^{-1}\|\sqrt{S^{\frac12,u}}\mathrm{Im}G_{s}(z)\sqrt{S^{\frac12,u}}\|_{\mathrm{op}}\lesssim W^{-1}|\mathrm{Im}w_{s}|^{-\frac32}.\label{eq:conj2}
\end{align}
The relevance of these estimates are from the proof of Lemma \ref{lemma:mstop}. More precisely, the previous estimate is an improvement over the bounds \eqref{eq:mstop3a} and \eqref{eq:mstop7c} by a factor of $|\mathrm{Im}w_{s}|^{-1/2}$. (We also note that the two operator norms, except for the difference in $S$ versus $S^{1/2}$, are equivalent by  Ward, i.e. $GG^{\ast}=|\mathrm{Im}w_{s}|^{-1}\mathrm{Im}G$.) The reason why we believe that \eqref{eq:conj2} is true is because our estimates \eqref{eq:mstop3a} and \eqref{eq:mstop7c} for the LHS of \eqref{eq:conj} did not use the fact that $S^{u}$ and $S^{1/2,u}$ restrict to a block of $G_{s}(z)$ of size $\mathrm{O}(W)\times\mathrm{O}(W)$. Since we essentially extend this block to the entire support of $G_{s}(z)$, which has size $\mathrm{O}(W|\mathrm{Im}w_{s}|^{-1/2})\times\mathrm{O}(W|\mathrm{Im}w_{s}|^{-1/2}$), we are morally losing a factor of $|\mathrm{Im}w_{s}|^{-1/2}$. (Note that $G_{s}(z)$ has the same support length as $T_{s}(z)$, and that $T_{s}(z)\approx\Theta_{s}$; now use Lemma \ref{lemma:sthetasrange} to get the claimed estimate on the support length.) 
%
%
%
\appendix
\section{Auxiliary diffusion estimates}\label{section:aux}
We collect below estimates for the diffusion profile matrix $\Theta_{t}=|m(z)|^{2}S(1-t|m(z)|^{2}S)^{-1}$ for $t\in[0,1]$ and $z$ in the bulk, i.e. $|E|<2$ fixed. First, we recall $w_{t}=-m(z)^{-1}-tm(z)$.
\begin{lemma}\label{lemma:thetaestimates}
Suppose that $|E|<2$ is fixed. If $\eta\gtrsim W^{2}/N^{2}$, then 
\begin{align*}
(\Theta_{t})_{ab}+|(\Theta_{t}S^{\frac12})_{ab}|+|(S^{\frac12}\Theta_{t}S^{\frac12})_{ab}|&\lesssim W^{-1}|\mathrm{Im}w_{t}|^{-\frac12},\\
\max_{a}\sum_{b}|(\Theta_{t})_{ab}|+\max_{a}\sum_{b}|(\Theta_{t}S^{\frac12})_{ab}|&\lesssim|\mathrm{Im}w_{t}|^{-1}.
\end{align*}
\end{lemma}
\begin{proof}
We start with the pointwise estimates. Note that it suffices to control only $(\Theta_{t})_{ab}$, since $S^{1/2}$ is an averaging operator, i.e. it has non-negative entries and row and column sums equal to $1$. We use the formula $(\zeta-A)^{-1}=\int_{0}^{\infty}\mathrm{e}^{-r\zeta}\mathrm{e}^{rA}dr$ for any $\zeta>0$ and $A\geq0$; this gives

\begin{align}
\Theta_{t}&=\frac{|m(z)|^{2}S}{1-t|m(z)|^{2}-t|m(z)|^{2}(S-\mathrm{Id})}\nonumber\\
&=|m(z)|^{2}\int_{0}^{\infty}\mathrm{e}^{-(1-t|m(z)|^{2})r}S\mathrm{e}^{rt|m(z)|^{2}(S-\mathrm{Id})}dr.\label{eq:thetaformula}
\end{align}
We first consider the case where $t\in[0,1/2]$. In this case, we know that $1-t|m(z)|^{2}\geq1-\frac12|m(z)|^{2}\asymp1$ and $|m(z)|^{2}\asymp1$ (see Lemma 6.2 in \cite{EY}, which gives $|\mathrm{Im}m(z)|\asymp1$ in the bulk). Now, note that $\exp\{rt|m(z)|^{2}(S-\mathrm{Id})\}$ is the semigroup for a random walk with transition matrix $S$. Thus, its row and column sums are equal to $1$, so that $|(S\mathrm{e}^{rt|m(z)|^{2}(S-\mathrm{Id})})_{ab}|\lesssim\max_{\beta}S_{a\beta}\lesssim W^{-1}$. In particular, we have 
\begin{align*}
\max_{a,b}|(\Theta_{t})_{ab}|&\lesssim W^{-1}|m(z)|^{2}\int_{0}^{\infty}\mathrm{e}^{-(1-t|m(z)|^{2})r}dr\lesssim W^{-1}\frac{|m(z)|^{2}}{1-t|m(z)|^{2}}\lesssim W^{-1}.
\end{align*}
This proves the desired pointwise estimate for $t\in[0,1/2]$, since the far RHS is $\lesssim W^{-1}|\mathrm{Im}w_{t}|^{-\alpha}$ for any $\alpha\geq0$ (recall that $\mathrm{Im}w_{t}=\eta+(1-t)\mathrm{Im}m(z)=\mathrm{O}(1)$). We now focus on $t\in[1/2,1]$. We start with the same formula for $\Theta_{t}$. However, for this, we will use the following random walk heat kernel bound:
\begin{align}
|\exp\{rt|m(z)|^{2}(S-\mathrm{Id})\}_{ab}|\lesssim W^{-1}r^{-\frac12}t^{-\frac12}+N^{-1}.\label{eq:hkbound}
\end{align}
(Intuitively, the LHS is the heat kernel at time $\asymp rt$ for a random walk on $\mathbb{T}_{N}=\Z/N\Z$ whose step distribution has variance proportional to $W^{2}$; this explains the familiar $W^{-1}t^{-1/2}r^{-1/2}$ factor on the RHS. The term $N^{-1}$ comes from the fact that the law of the random walk converges to the uniform measure on $\Z/N\Z$ in the long-time limit. The estimate itself is standard, but we could not find a reference, so we give a brief but complete argument for it later.) Assuming that \eqref{eq:hkbound} is true, then since $t^{-1/2}\lesssim1$ for $t\in[1/2,1]$, we get
\begin{align*}
\max_{a,b}|(\Theta_{t})_{ab}|&\lesssim \int_{0}^{\infty}\mathrm{e}^{-(1-t|m(z)|^{2})r}\{W^{-1}r^{-\frac12}+N^{-1}\}dr\\
&\lesssim W^{-1}|1-t|m(z)|^{2}|^{-\frac12}+N^{-1}|1-t|m(z)|^{2}|^{-1},
\end{align*}
where the last bound follows by change-of-variables $(1-t|m(z)|^{2})r\mapsto r$ and $\int_{0}^{\infty}\mathrm{e}^{-r}r^{-1/2}dr\lesssim1$. Moreover, the second term in the last line is controlled by the first. Indeed, we have $1-t|m(z)|^{2}\geq1-|m(z)|^{2}\gtrsim\eta$ by Lemma 3.5 in \cite{EKYY}, and $N^{-1}\eta^{-1/2}\lesssim W^{-1}$ by assumption. Now, write $1-t|m(z)|^{2}=1-|m(z)|^{2}+(1-t)|m(z)|^{2}$ and $\mathrm{Im}w_{t}=\eta+(1-t)\mathrm{Im}m(z)$. Combining these two gives
\begin{align*}
1-t|m(z)|^{2}&=1-|m(z)|^{2}+|m(z)|^{2}\frac{(1-t)\mathrm{Im}m(z)}{\mathrm{Im}m(z)}\\
&=1-|m(z)|^{2}+|m(z)|^{2}\frac{\mathrm{Im}w_{t}}{\mathrm{Im}m(z)}-\frac{|m(z)|^{2}\eta}{\mathrm{Im}m(z)}=|m(z)|^{2}\frac{\mathrm{Im}w_{t}}{\mathrm{Im}m(z)},
\end{align*}
where the last identity follows by (3.3) in \cite{EKYY}. Now, we note that $|m(z)|^{2},|\mathrm{Im}m(z)|\asymp1$ (see Lemma 3.5 in \cite{EKYY}), so that $1-t|m(z)|^{2}\asymp\mathrm{Im}w_{t}$. In particular, the previous two displays yield the desired pointwise estimate on $|(\Theta_{t})_{ab}|$. Finally, we have 
\begin{align*}
\sum_{b}|(\Theta_{t})_{ab}|&\lesssim|m(z)|^{2}\int_{0}^{\infty}\mathrm{e}^{-(1-t|m(z)|^{2})r}\sum_{b}(S\mathrm{e}^{rt|m(z)|^{2}(S-\mathrm{Id})})_{ab}dr\\
&\lesssim|m(z)|^{2}\int_{0}^{\infty}\mathrm{e}^{-(1-t|m(z)|^{2})r}dr\lesssim\frac{|m(z)|^{2}}{1-t|m(z)|^{2}},
\end{align*}
since the row and column sums of both $S$ and $\exp\{rt|m(z)|^{2}(S-\mathrm{Id})\}$ are equal to $1$. We showed earlier that $1-t|m(z)|^{2}\gtrsim|\mathrm{Im}w_{t}|$ for $|E|<2$. Since $|m(z)|^{2}\lesssim1$, the desired summed estimate for $(\Theta_{t})_{ab}$ (over the $b$ index) follows. The summed estimate for $\Theta_{t}S^{1/2}$ follows as well since $S^{1/2}$ is an averaging operator (its entries are non-negative, and its row and column sums are equal to $1$). 

Ultimately, it suffices to now show the heat kernel bound \eqref{eq:hkbound}. First, note that $\exp\{rt|m(z)|^{2}(S-\mathrm{Id})\}$ is the transition operator at time $rt|m(z)|^{2}$ for a random walk on the torus $\mathbb{T}_{N}\simeq\{1,\ldots,N\}$ whose step distribution has jump-range $\lesssim W$ and variance $\asymp W^{2}$. Next, we let $\mathcal{L}$ be the generator for the same random walk but on $\Z$ (i.e. without periodic boundary conditions), and let $\exp\{rt|m(z)|^{2}\mathcal{L}\}$ be its transition operator at time $rt|m(z)|^{2}$. Because the underlying random walks are space-time homogeneous, we have 
\begin{align*}
\exp\{rt|m(z)|^{2}(S-\mathrm{Id})\}_{ab}=\sum_{k\in\Z}\exp\{rt|m(z)|^{2}\mathcal{L}\}_{a,b+kN}.
\end{align*}
(Indeed, the random walk on the torus is the same as the random walk on the line and then projecting $\Z\to\Z/N\Z\simeq\mathbb{T}_{N}$.) We now claim that for any ${K}>0$, we have the following sub-exponential tail estimate for a random walk heat kernel whose step distribution has variance proportional to $W^{2}$:
\begin{align*}
\exp\{rt|m(z)|^{2}\mathcal{L}\}_{a,b+kN}\lesssim_{{K}} W^{-1}r^{-\frac12}t^{-\frac12}|m(z)|^{-1}\mathrm{e}^{-{K}[1\wedge(W^{2}rt|m(z)|^{2})^{-1/2}]|a-b-kN|}.
\end{align*}
This follows by (A.12) in \cite{DT16}. (Technically, (A.12) in \cite{DT16} makes the assumption that $W\asymp1$; however, an immediate inspection of its proof, namely (A.20) in \cite{DT16}, implies that for general $W$, the variance of the underlying random walk hits each factor of $t$ in the bound (A.12) in \cite{DT16}. Since the variance of our random walk is $\asymp W^{2}$, this is where the $W^{-1}$ and $W^{2}$ factors in the previous display come from. We do not reproduce the elementary calculation verbatim.) Now, we note that $|m(z)|\asymp1$ as well as the geometric series bound
\begin{align*}
\sum_{k\in\Z}\mathrm{e}^{-{K}[1\wedge(W^{2}rt|m(z)|^{2})^{-1/2}]|a-b-kN|}\lesssim\left\{1-\mathrm{e}^{-{K}[1\wedge(W^{2}rt|m(z)|^{2})^{-1/2}]N}\right\}^{-1}.
\end{align*}
By combining the previous three displays, we have 
\begin{align*}
\exp\{rt|m(z)|^{2}(S-\mathrm{Id})\}_{ab}&\lesssim_{{K}}W^{-1}r^{-\frac12}t^{-\frac12}|m(z)|^{-1}\left\{1-\mathrm{e}^{-{K}[1\wedge(W^{2}rt|m(z)|^{2})^{-1/2}]N}\right\}^{-1}.
\end{align*}
If the exponent $[1\wedge(W^{2}rt|m(z)|^{2})^{-1/2}]N$ is smaller than some universal constant $\varepsilon>0$, then the last factor on the RHS of the previous display is $\lesssim Wr^{1/2}t^{1/2}|m(z)|N^{-1}$. (Indeed, this follows by $1-\mathrm{e}^{-\gamma}\gtrsim\gamma$ for $\gamma$ small enough.) At this point, we deduce \eqref{eq:hkbound}. If the exponent $[1\wedge(W^{2}rt|m(z)|^{2})^{-1/2}]N$ is larger than $\varepsilon>0$ (and thus uniformly bounded away from $0$), then the last factor on the RHS of the previous display is $\mathrm{O}(1)$; at this point, we also get \eqref{eq:hkbound} (since $|m(z)|\asymp1$ as noted earlier in this proof).
\end{proof}
We now give a purely technical estimate giving off-diagonal behavior of $S^{1/2}\Theta_{t}S^{1/2}$. It is sub-optimal by $W^{\varepsilon}$ factors, but this is enough for our purposes.
\begin{lemma}\label{lemma:sthetasrange}
Assume that $\eta\gtrsim W^{2}/N^{2}$ and that $|E|<2$ is fixed. Assume also $\eta\gg W^{C}$ for some $C=\mathrm{O}(1)$, and that $W\gg1$. For any $\varepsilon,{K}>0$, there exists $\delta>0$ such that for any indices $a,b$, we have
\begin{align*}
(\Theta_{t})_{ab}+(S^{\frac12}\Theta_{t}S^{\frac12})_{ab}&\lesssim_{{K}} W^{-1+\varepsilon}|\mathrm{Im}w_{t}|^{-\frac12}\exp\left\{-\frac{{K}|\mathrm{Im}w_{t}|^{1/2}|a-b|_{N}}{W^{1+\varepsilon}}\right\}+\mathrm{O}(\mathrm{e}^{-W^{\delta}}).
\end{align*}
\end{lemma}
\begin{proof}
We claim that it is enough to show that there exists $C=\mathrm{O}(1)$ such that
\begin{align}
(\Theta_{t})_{ab}+(S^{\frac12}\Theta_{t}S^{\frac12})_{ab}&\lesssim_{{K}} W^{C}\exp\left\{-\frac{{K}|\mathrm{Im}w_{t}|^{1/2}|a-b|_{N}}{W^{1+\varepsilon}}\right\}+\mathrm{O}(\mathrm{e}^{-W^{\delta}}).\label{eq:sthetarangemain}
\end{align}
Indeed, if this is true, then we can interpolate it with Lemma \ref{lemma:thetaestimates} to get the following for any $c\in(0,1)$:
\begin{align*}
(\Theta_{t})_{ab}+(S^{\frac12}\Theta_{t}S^{\frac12})_{ab}\lesssim (W^{-1}|\mathrm{Im}w_{t}|^{-\frac12})^{1-c}\left\{W^{C}\exp\left\{-\frac{{K}|\mathrm{Im}w_{t}|^{1/2}|a-b|_{N}}{W^{1+\varepsilon}}\right\}+\mathrm{O}(\mathrm{e}^{-W^{\delta}}\right\}^{c}.
\end{align*}
Choosing $c>0$ small depending on $C$ gives the desired estimate. We are now left to show \eqref{eq:sthetarangemain}.

We consider the matrix $\mathcal{Q}S\mathrm{e}^{rt|m(z)|^{2}(S-\mathrm{Id})}\mathcal{Q}$, where either $\mathcal{Q}=\mathrm{Id}$ or $\mathcal{Q}=S^{1/2}$. This is the transition operator for the following Markov process. First, take a step according to the transition matrix $\mathcal{Q}$. Then, for time $rt|m(z)|^{2}$, do a continuous-time simple random walk with transition matrix $S$. Next, take one step according to the transition matrix $\mathcal{Q}S$. The entry $(\mathcal{Q}S\mathrm{e}^{rt|m(z)|^{2}(S-\mathrm{Id})}\mathcal{Q})_{ab}$ is then the probability $\mathbb{P}[a\to_{k+3}b]$ of going $a\to b$ under such a dynamic. Now, note that the number of steps taken in this ``patched" random walk dynamic is $3$ plus the number of steps in the continuous-time part (corresponding to $\exp\{rt|m(z)|^{2}(S-\mathrm{Id})\}$). The speed of taking a jump in this dynamic is $1$ (since the rows and columns of $S$ sum to $1$). Thus, the probability of taking $k+3$-steps according to the random walk defined by the operator $\mathcal{Q}S\mathrm{e}^{rt|m(z)|^{2}(S-\mathrm{Id})}\mathcal{Q}$ is the probability of equaling $k$ for a Poisson random variable of speed parameter $rt|m(z)|^{2}\lesssim1$. Therefore,
\begin{align*}
(\mathcal{Q}S\mathrm{e}^{rt|m(z)|^{2}(S-\mathrm{Id})}\mathcal{Q})_{ab}&=\sum_{k=0}^{\infty}\frac{r^{k}t^{k}|m(z)|^{2k}\mathrm{e}^{-rt|m(z)|^{2}k}}{k!}\mathbb{P}[a\to_{k+3}b].
\end{align*}
Note, now that $\mathbb{P}[a\to_{k+3}b]$ is bounded above by the probability that a martingale travels distance greater than $|a-b|_{N}$ after $k+3$ steps. This martingale has steps of length $\mathrm{O}(W)$. Thus, by the Azuma martingale inequality, we deduce that for some universal constant $c>0$, we have
\begin{align*}
\mathbb{P}[a\to_{k+3}b]&\lesssim\exp\left\{-c\frac{|a-b|_{N}^{2}}{(k+3)W^{2}}\right\}.
\end{align*}
We combine the last two displays and decompose the sum over $k$ into $k\leq W^{\varepsilon/2}(1+r)$ and $k>W^{\varepsilon/2}(1+r)$. In the former summation, we bound the previous display by replacing $k\mapsto W^{\varepsilon/2}(1+r)$ and then bounding the remaining sum over $k$ by $1$. In the latter summation, we bound $\mathbb{P}[a\to_{k+3}b]\leq1$. We deduce, for some universal constant $c>0$, that
\begin{align*}
(\mathcal{Q}S\mathrm{e}^{rt|m(z)|^{2}(S-\mathrm{Id})}\mathcal{Q})_{ab}&\lesssim \mathrm{e}^{-c\frac{|a-b|_{N}^{2}}{W^{2+\varepsilon/2}(1+r)}}+\sum_{k=W^{\varepsilon/2}(1+r)}^{\infty}\frac{r^{k}t^{k}|m(z)|^{2k}\mathrm{e}^{-rt|m(z)|^{2}k}}{k!}\\
&\lesssim \mathrm{e}^{-c\frac{|a-b|_{N}^{2}}{W^{2+\varepsilon/2}(1+r)}}+\mathrm{e}^{-cW^{\varepsilon/2}}.
\end{align*}
(The last line follows by the Stirling bound $k!\gtrsim k^{k}e^{-2k}$ so for $k\geq W^{\varepsilon/2}(1+r)$, we have $r^{k}/k!\lesssim \varepsilon^{k}$ for any small but fixed $\varepsilon>0$ if $W$ is large. This implies that the last term in the first line is controlled by $\sum_{k\geq W^{\varepsilon/2}}\varepsilon^{k}$, leading to the second line above.) We now return to \eqref{eq:thetaformula} and plug in the previous display:
\begin{align*}
(\mathcal{Q}\Theta_{t}\mathcal{Q})_{ab}&\lesssim|m(z)|^{2}\int_{0}^{\infty}\mathrm{e}^{-(1-t|m(z)|^{2})r}(\mathcal{Q}S\mathrm{e}^{rt|m(z)|^{2}(S-\mathrm{Id})}\mathcal{Q})_{ab}dr\\
&\lesssim|m(z)|^{2}\int_{0}^{\infty}\mathrm{e}^{-(1-t|m(z)|^{2})r}\left\{\mathrm{e}^{-c\frac{|a-b|_{N}^{2}}{W^{2+\varepsilon/2}(1+r)}}+\mathrm{e}^{-cW^{\varepsilon/2}}\right\}dr\\
&\lesssim|m(z)|^{2}\int_{0}^{\infty}\mathrm{e}^{-(1-t|m(z)|^{2})r}\mathrm{e}^{-c\frac{|a-b|_{N}^{2}}{W^{2+\varepsilon/2}(1+r)}}dr+\frac{|m(z)|^{2}}{1-t|m(z)|^{2}}\mathrm{e}^{-cW^{\varepsilon/2}}.
\end{align*}
We decompose the remaining integral into integrals on the intervals $[0,W^{\varepsilon/2}(1-t|m(z)|^{2})^{-1}]$ and $[W^{\varepsilon/2}(1-t|m(z)|^{2})^{-1},\infty)$. On the former integration domain, we can bound the sub-Gaussian factor by replacing $r\mapsto W^{\varepsilon/2}(1-t|m(z)|^{2})^{-1}$. On the latter, we can bound the sub-Gaussian factor by $1$. This gives
\begin{align*}
|m(z)|^{2}\int_{0}^{\infty}\mathrm{e}^{-(1-t|m(z)|^{2})r}\mathrm{e}^{-c\frac{|a-b|_{N}^{2}}{W^{2+\varepsilon/2}(1+r)}}dr&\lesssim |m(z)|^{2}\mathrm{e}^{-c\frac{|a-b|_{N}^{2}(1-t|m(z)|^{2})}{2W^{2+\varepsilon}}}\int_{0}^{W^{\varepsilon/2}(1-t|m(z)|^{2})^{-1}}\mathrm{e}^{-(1-t|m(z)|^{2})r}dr\\
&+|m(z)|^{2}\int_{W^{\varepsilon/2}(1-t|m(z)|^{2})^{-1}}^{\infty}\mathrm{e}^{-(1-t|m(z)|^{2})r}dr\\
&\lesssim \frac{|m(z)|^{2}}{1-t|m(z)|^{2}}\mathrm{e}^{-c\frac{|a-b|_{N}^{2}(1-t|m(z)|^{2})}{2W^{2+\varepsilon}}}+\frac{|m(z)|^{2}}{1-t|m(z)|^{2}}\mathrm{e}^{-W^{\varepsilon/2}}.
\end{align*}
We showed in the proof of Lemma \ref{lemma:thetaestimates} that $1-t|m(z)|^{2}\gtrsim\mathrm{Im}w_{t}\geq\eta\geq W^{-C}$ for some $C>0$. (The last inequality is an assumption, and the inequality $\mathrm{Im}w_{t}\geq\eta$ follows by $\mathrm{Im}w_{t}=\eta+(1-t)\mathrm{Im}m(z)$.) We also recall that $|m(z)|\asymp1$. Thus, the previous two displays give, for some $C=\mathrm{O}(1)$ and $c>0$ universal, that
\begin{align*}
|(\mathcal{Q}\Theta_{t}\mathcal{Q})_{ab}|&\lesssim W^{C}\exp\left\{-c\frac{|a-b|_{N}^{2}|\mathrm{Im}w_{t}|}{W^{2+\varepsilon}}\right\}+\mathrm{e}^{-W^{\varepsilon/2}}.
\end{align*}
Now apply $\exp\{-\gamma^{2}\}\lesssim_{{K}}\exp\{-{K}\gamma\}$, which holds for all ${K},\gamma>0$, to the first term on the RHS.
\end{proof}
Next, we present an estimate for the $B_{t}$ matrix from Lemma \ref{lemma:expansion}.
\begin{lemma}\label{lemma:Bestimates}
Assume that $|E|<2$ is fixed. For any $a,b$, we have
\begin{align*}
\sum_{\alpha}|(B_{t})_{\alpha b}|+\sum_{\beta}|(B_{t})_{a\beta}|&\lesssim1.
\end{align*}
\end{lemma}
\begin{proof}
We first record the following analog of \eqref{eq:thetaformula}:
\begin{align*}
B_{t}&=\left\{1-tm(z)^{2}+tm(z)^{2}(S-\mathrm{Id})\right\}^{-1}=m(z)^{-2}\left\{m(z)^{-2}(1-tm(z)^{2})+t(S-\mathrm{Id})\right\}^{-1}\\
&=m(z)^{-2}\int_{0}^{\infty}\mathrm{e}^{-m(z)^{-2}(1-tm(z)^{2})r}\mathrm{e}^{rt(S-\mathrm{Id})}dr.
\end{align*}
As in the proof of Lemma \ref{lemma:thetaestimates}, the matrix $\exp\{rt(S-\mathrm{Id})\}$ is the transition operator for a random walk, and thus its entries are non-negative with row and column sums equal to $1$. Moreover, we know that $|m(z)|\asymp1$ as noted in the proof of Lemma \ref{lemma:thetaestimates} (see also Lemma 6.2 in \cite{EY} for the lower bound $\mathrm{Im}m(z)\gtrsim1$ in the bulk $|E|<2$). Thus, we have 
\begin{align*}
\sum_{\alpha}|(B_{t})_{\alpha b}|+\sum_{\beta}|(B_{t})_{a\beta}|&\lesssim\int_{0}^{\infty}\mathrm{e}^{-|m(z)|^{-2}|1-tm(z)^{2}|r}\left\{\sum_{\alpha}\mathrm{e}^{rt(S-\mathrm{Id})}_{\alpha b}+\sum_{\beta}\mathrm{e}^{rt(S-\mathrm{Id})}_{a\beta}\right\}dr\\
&\lesssim\int_{0}^{\infty}\mathrm{e}^{-|m(z)|^{-2}|1-tm(z)^{2}|r}dr=\frac{|m(z)|^{2}}{|1-tm(z)^{2}|}.
\end{align*}
Recall that $|E|<2$. In this case, we have $m(z)^{2}=\frac14(2z^{2}-4-2z\sqrt{z^{2}-4})$. Write $z=E+i\eta$, and assume that $\eta<\varepsilon$ for some $\varepsilon>0$ to be determined in the sense that $z\approx E$. In particular, we have $m(z)^{2}=\frac14[2E^{2}-4-2E\sqrt{E^{2}-4}+\mathrm{O}(\varepsilon)]$. Since $|E|<2$, the real part of $m(z)^{2}$ is $\frac14[2E^{2}-4]+\mathrm{O}(\varepsilon)$, which is bounded away from $1$ if $\varepsilon>0$ is small enough. In particular, for $z=E+i\eta$ and $|E|<2$, we know that $|1-tm(z)^{2}|\geq c$ for some universal constant $c>0$. On the other hand, we also know that $|1-tm(z)^{2}|\geq1-t|m(z)|^{2}\geq1-|m(z)|^{2}\gtrsim\eta$ as in the proof of Lemma \ref{lemma:thetaestimates} (see Lemma 3.5 in \cite{EKYY}). So, if $\eta\geq\varepsilon>0$, then we know $|1-tm(z)^{2}|\geq c$ for a possibly different universal $c>0$ as well. Ultimately, we know that $|1-tm(z)^{2}|\geq c$ for some universal $c>0$. Using this and $|m(z)|\asymp1$, we have 
\begin{align*}
\frac{|m(z)|^{2}}{|1-tm(z)^{2}|}\lesssim 1,
\end{align*}
so the previous two displays give the desired bound.
\end{proof}

\bibliographystyle{abbrv}
\bibliography{random-band-2}
\end{document}